\documentclass[a4paper,11pt,oneside,reqno]{amsart}
\usepackage[english]{babel}
\usepackage{graphicx}
\usepackage[a4paper,margin=1in]{geometry}
\usepackage{amsmath}
\usepackage{bm}
\usepackage{amssymb}
\usepackage{amsthm}
\usepackage{graphicx}
\usepackage[colorlinks=true, allcolors=blue]{hyperref}
\usepackage{graphicx}
\usepackage{geometry}
\usepackage{verbatim}
\usepackage{amsmath}
\usepackage{amssymb}
\usepackage{amsthm}
\numberwithin{equation}{section}
\usepackage{graphicx}
\usepackage{float}
\usepackage[section]{placeins}
\newtheorem{theorem}{Theorem}[section]
\newtheorem{lemma}{Lemma}[section]
\newtheorem{pro}{Proposition}[section]
\newtheorem{definition}{Definition}[section]
\newtheorem{remark}{Remark}[section]
\newtheorem{assume}{Assumption}[section]
\newtheorem{corollary}{Corollary}[section]
\usepackage{hyperref}
\hypersetup{hidelinks} 
\usepackage{dsfont}
\allowdisplaybreaks[1]
\usepackage{enumerate}
\usepackage{color}
\usepackage{booktabs}
\usepackage{graphicx}
\usepackage{subcaption}
\allowdisplaybreaks

\newcommand{\blue}[1]{{\color{blue}#1}}

\title{On a Merton problem with irreversible healthcare investment}
\date{\today.\ A previous version of this paper circulated under the title ``Consumption decision, portfolio choice and healthcare irreversible investment{.''}}
\author[1]{ Giorgio Ferrari$^\ast$ \quad and \quad Shihao Zhu$^\dagger$}

\thanks{$^\ast$Center for Mathematical Economics (IMW), Bielefeld University, Universit{\"a}tsstrasse 25, 33615, Bielefeld, Germany (giorgio.ferrari@uni-bielefeld.de).}
\thanks{$^\dagger$Institute of Insurance Science, Ulm University, Helmholtzstr.\ 20, 89069, Ulm, Germany (shihao.zhu@uni-ulm.de).}
\begin{document}
\newpage

\begin{abstract}

We propose a tractable dynamic framework for the joint determination of optimal
consumption, portfolio choice, and irreversible healthcare investment. Our
model is based on Merton's portfolio and consumption problem, where, in
addition, {the agent can optimally choose the time at which she undertakes
healthcare investment at a fixed continuous rate.} Health depreciates with age
and directly affects the agent's force of mortality, so that investment in
healthcare reduces the agent's mortality risk. The resulting optimization
problem is formulated as a stochastic control-stopping problem with a random
time horizon and state variables given by the agent's wealth and health capital.
We transform this problem into its dual version, which is a two-dimensional
optimal stopping problem with interconnected dynamics. Regularity of the
optimal stopping value function is derived, and the related free boundary is characterized
through a nonlinear integral equation, which we compute numerically. In the
original coordinates, the agent thus invests in healthcare whenever her wealth
exceeds a health-dependent transformed version of the optimal stopping
boundary. We also provide numerical illustrations of the optimal strategies and
discuss some financial implications.

\end{abstract}

\maketitle

\vspace{2mm}


\vspace{2mm}

\vspace{2mm}

\vspace{2mm}

\section{Introduction}

It has been recognized that expenditures on medical services, annual physical
exams, and exercise can be viewed as investments in health capital and analyzed
using the tools of capital theory. This approach has enabled economists to
derive propositions about the pattern of healthcare spending over an
individual's lifetime and to describe the behavior of health capital over the
life cycle. For example, the demand-for-health model developed by
\cite{grossman1972concept} extends human capital theory by explicitly
incorporating health and recognizing that there are both consumption and
investment motives for investing in health. The basic features of the model are
(1) that health can be viewed as a durable capital stock that produces an output
of healthy time, (2) that individuals inherit an initial stock of health that
depreciates with age, (3) that the stock of health can be increased by
investment, and (4) that the individual demands health (a) for its
utility-enhancing effects (the consumption motive), and (b) for its effect on
the amount of healthy time (the investment motive).

Based on the aforementioned standard model assumptions, various health
economists have extended Grossman's dynamic health investment model. These
extensions address, for example, the introduction of uncertainty into the
theoretical model (see, e.g., \cite{bolin2020consumption},  \cite{cropper1977health},
 and \cite{ehrlich2000uncertain}) or the
distribution of health within the family (see \cite{bolin2001family} and \cite{jacobson2000family}, among many others).

Empirical evidence suggests that health crucially influences an agent's
financial decisions (see, e.g., \cite{atella2012household}, \cite{rosen2004portfolio},
\cite{smith2009impact}). In particular, the
literature reveals that health status is positively correlated with income,
consumption, and asset holdings, and negatively correlated with health
expenditures. To account for this fact, \cite{hugonnier2013health} propose a
dynamic framework for the joint determination of optimal consumption, portfolio
holdings, and health investment. They solve for the optimal rules in closed form
and provide parameter estimates that confirm the relevance of all the main
characteristics of the model. More recently, \cite{guasoni2019consumption}
focus on a representative household that makes consumption, investment, and
healthcare spending decisions in order to maximize welfare under time-separable
utilities. In \cite{guasoni2019consumption}, the endogenous force of mortality is
taken as the sole state variable, in addition to wealth, and the resulting
optimal stochastic control problem is reduced to the study of a nonlinear
ordinary differential equation. This equation is shown to have a unique
solution, which has an explicit expression in the old-age limit. Further,
\cite{aurand2021mortality} study optimal consumption, investment, and
healthcare spending under Epstein--Zin preferences.

As in \cite{hugonnier2013health}, in this paper we combine two well-established
frameworks from the financial and health economics literature within a unified
setup. However, differently from \cite{hugonnier2013health}, health-related
decisions are approached from a different viewpoint. Specifically, we start from
\cite{merton1971optimum}'s portfolio and consumption choice problem and append
to this model the determination of the time of health investment (e.g., joining
a gym club or undertaking precautionary health measures). Meanwhile, the above essential features of
\cite{grossman1972concept}'s canonical model are retained.

In what follows, a distinction is made and maintained between curative and
precautionary health investments. This distinction is important since the two
types of investment may behave quite differently over the life cycle.
Specifically, curative health investments are defined as investments that have
direct effects on the stock of health or the rate of depreciation of health, or
both, and are produced using medical-care goods and services (e.g., taking
antibiotics to cure a bacterial infection or undergoing surgery to remove a
tumor). As a result, curative health investments are not among an individual's
choice variables. In contrast, precautionary health investments (e.g., regular
exercise, a healthy diet, or regular health check-ups and screenings) are
defined as those that are under the control of an agent and indirectly affect
the rate of depreciation of the stock of health by directly affecting the stock
of health itself. That is, the current stock of health, which is directly
influenced by precautionary health investments, determines the rate of
depreciation of health. The view that precautionary health investments
indirectly affect the rate of depreciation of health in the aforementioned
manner is consistent with ample medical evidence. Therefore, we focus on
precautionary health investments in our model (see \eqref{2-3} in Section \ref{sec-2}).

The introduction of the option to choose the time of a precautionary health
investment raises several questions. First, if an individual is faced with the
choice of when to buy preventive health services, how should she optimally
behave? In particular, a non-trivial trade-off arises: if the agent invests in
healthcare too early, she reduces her wealth, thus affecting consumption and
portfolio choice; if she invests in health too late, this will negatively affect
utility and survival probability. Second, how do optimal consumption and
investment strategies react to the introduction of health factors?

The answers to these questions, which are collected in Sections \ref{sec5} and \ref{sec:numerics}, are
intuitively convincing. We show that it is optimal to invest in health when the
agent's wealth first reaches an endogenously determined boundary curve, which
depends on the agent's health status. Intuitively, if the agent is sufficiently
rich (her wealth exceeds the corresponding boundary), then health investment
should be undertaken immediately; otherwise, it is optimal to wait for an
increase in wealth. {Moreover, our model highlights how health affects consumption and asset
allocation. We first consider a no-investment benchmark in which $I=K=0$.
In this case, the optimal consumption--wealth ratio is independent of wealth
and varies only mildly with health, while the optimal portfolio share is
constant. Under the parameter values used for this benchmark, healthier agents
consume a slightly smaller fraction of their wealth, reflecting the longer
effective planning horizon induced by lower mortality. When healthcare
investment is available, the optimal consumption and portfolio policies differ
before and after the agent undertakes healthcare investment, and they exhibit
visible changes around the endogenous investment boundary.

The numerical analysis in Section \ref{sec:numerics} also illustrates the
economic implications of the optimal healthcare investment boundary under the
baseline calibration reported in Table \ref{tab1}, where $I>0$ and
$K=I^\beta$. Under this calibration, the boundary in the primal variables is
increasing in health: healthier agents require a higher level of wealth to
initiate precautionary healthcare investment. This is because the same absolute
improvement in the future health path generates a larger marginal utility gain
and a larger reduction in mortality risk when the agent is in poor health,
whereas its marginal benefit is smaller when the agent is already healthy.}

\subsection{Overview of the mathematical analysis}

From a mathematical point of view, our model leads to a
\emph{random-horizon, two-dimensional stochastic control problem with
discretionary stopping}. The random horizon is the agent's time of death,
\(\eta\). We do not require \(\eta\) to be a stopping time with respect to the
financial-market filtration \(\mathbb F\). Instead, its conditional distribution
is determined by an intensity that depends on the agent's health status. Through
health investment, the agent improves her health capital, thereby reducing the
force of mortality and changing the distribution of \(\eta\).

The two state variables are the wealth process \(X\) and the health capital
process \(H\). The agent chooses the consumption rate $c$, the portfolio $\pi$, and the
time $\tau$ at which she starts a costly continuous healthcare investment
at the fixed rate $I$. At time \(\tau\), the dynamics change:
health capital is improved through the investment, reducing mortality risk,
whereas wealth is reduced by the investment cost. The objective is to maximize
expected intertemporal utility from consumption and health up to the random time
of death.

\begin{sloppypar}
Problems with a similar structure arise, for instance, in retirement-timing
models, where the agent consumes, invests in risky assets, and decides when to
retire (see, e.g., \cite{jin2006disutility}, \cite{ferrari2023optimal}, and \cite{yang2018optimal}). Combined stochastic control and optimal stopping
problems also arise in mathematical finance, for example in the pricing of
American contingent claims under constraints and in utility maximization
problems with discretionary stopping; see, e.g.,
\cite{karatzas1998hedging} and \cite{karatzas2000utility}. In order to handle
the interaction between consumption, portfolio choice, and the healthcare
investment decision, we combine duality methods with a free-boundary analysis.
The analysis proceeds as follows.
\end{sloppypar}

\textbf{Step 1.} We first transform the original stochastic control-stopping
problem, whose value function is denoted by \(V\), into a dual problem using
martingale and duality methods, in the spirit of
\cite{karatzas2000utility} and \cite{yang2018optimal}. This transformation
reduces the consumption-portfolio component of the problem and leads to a dual
optimal stopping problem with value function \(J\).

\textbf{Step 2.} We then study the dual problem, which is a
\emph{two-dimensional optimal stopping problem with interconnected dynamics}.
The dual variable \(Z\) evolves as a geometric diffusion whose drift depends on
the health capital process \(H\). At the same time, \(H\) affects the force of mortality and hence the discount factor in the stopping functional. This coupling
makes the optimal stopping problem non-standard.

A particular difficulty is that no monotonicity of the free boundary with
respect to health is available in general. Indeed, the mortality channel and the
investment-benefit channel may interact in opposing ways. Since monotonicity of
the boundary is often a key ingredient in proving boundary continuity or higher
regularity in optimal stopping problems, standard monotonicity-based arguments
cannot be directly applied here. For related discussions in free-boundary
problems, see, for example, \cite{de2019lipschitz}.

{Our analysis therefore proceeds in two stages. First, we study the transformed
stopping value \(\widehat J(z,h):=J(z,h)-\widehat W(z,h)\), where
\(\widehat W\) is the payoff from immediate investment. {We prove that
\(\widehat J\) is continuous and use this regularity to show
that the stopping region is closed.} Together with the monotonicity of
\(\widehat J\) in the dual variable \(z\), this allows us to represent the
stopping set through a boundary \(b(h)\), so that
\(\mathcal I=\{(z,h)\in\mathcal O:0<z\le b(h)\}\). Moreover, the boundary
\(b\) is shown to be upper-semicontinuous.}

{Second, in order to characterize the boundary, we introduce a time-inhomogeneous
reformulation. For a fixed initial health level \(h\), we write
\(\widetilde h(t)=he^{-\delta t}\) and
\(\widetilde b(t)=b(\widetilde h(t))\). This transforms the two-dimensional autonomous stopping problem into an
equivalent one-dimensional time-inhomogeneous stopping problem with value
function \(\widetilde J\). In this formulation, we establish interior
\(C^{1,2}\)-regularity of \(\widetilde J\) in the continuation region and a
spatial smooth-fit property at finite boundary points. These regularity
properties are sufficient to apply a weak version of Dynkin's formula and to
derive an integral representation of \(\widetilde J\). Evaluating this
representation at the boundary yields a nonlinear integral equation for
\(\widetilde b\), and hence for the original boundary \(b\).}

\textbf{Step 3.} We finally return to the original wealth-health coordinates.
Using the duality relations and the regularity achieved in Step 2, we express the optimal consumption and portfolio
policies, as well as the optimal healthcare investment time, in terms of the
dual value function and the optimal stopping boundary. In particular, in the
primal variables the agent invests in healthcare when wealth first reaches a
health-dependent transformed version of the dual stopping boundary.

In summary, our contribution is twofold. Economically, we introduce a timing
margin for irreversible healthcare investment into a continuous-time
consumption-portfolio framework with health-dependent mortality. This allows us
to study when an agent should start a sustained preventive healthcare
investment, and how this decision interacts with consumption and portfolio
choice. Mathematically, the paper analyzes a two-dimensional optimal stopping
problem with coupled state dynamics and a boundary that need not be monotone. We
establish the regularity needed to represent the stopping region, derive a
nonlinear integral equation for the boundary, and solve this equation
numerically to study the sensitivity of the optimal healthcare investment
strategy to the model parameters.


\subsection{Plan of the paper}
The rest of the paper is organized as follows. In Section \ref{sec-2}, we introduce the model. We transform the original stochastic control-stopping problem into a pure stopping problem in Section \ref{sec-3}, while in Section \ref{sec4} we study the (candidate) dual optimal stopping problem. In Section \ref{sec5}, we provide the optimal health investment boundary in primal variables, and in Section \ref{sec:numerics} we present a detailed numerical study and provide some financial implications. Section \ref{sec7} concludes. 
Appendix \ref{seca} collects the proofs of some results of Sections \ref{sec-3}, \ref{sec4} and \ref{sec5}. In Appendix \ref{numerical}, we give the details of our numerical method used in Section \ref{sec:numerics}.

\section{Setting and problem formulation}\label{sec-2}

\subsection{Setting}
Let $({\Omega}, {\mathcal{F}}, {\mathbb{P}})$ be a complete probability space, endowed with the filtration $\mathbb{F}:=\{ \mathcal{F}_t, t \geq 0 \}$ generated by a one-dimensional Brownian motion $B:=\{ B_t, t \geq 0\}$ and augmented by the $\mathbb{P}$-null sets of $\mathcal{F}$. We assume that there exists a random variable $\Theta$, constructed on $(\Omega, \mathcal{F})$, independent of $\mathcal{F}_\infty$ and such that 
\begin{align*}
    \mathbb{P}[\Theta>v] =e^{-v}, \ v \geq 0.
\end{align*}

\begin{sloppypar}
Consider an agent whose lifespan is determined through health capital.\ In the spirit of \cite{hugonnier2013health}, we model the force of mortality process $M := \{ M_t, t \geq 0  \}$ as
\begin{align}\label{2-0}
M_t = m^0+m^1H_t^{-\kappa}, 
\end{align}
for some non-negative constants $m^0, m^1$, and $\kappa > 0$. Here, $H:=\{ H_t, t \geq 0 \}$ is the health capital process. {In (\ref{2-0}), the endogenous component of the mortality process is modeled as a function of the agent's current health status, rather than as a direct function of her current health investment (cf.\ \cite{guasoni2019consumption}). This specification precludes the possibility of arbitrarily modifying the force of mortality through health investment. Although the framework in \cite{guasoni2019consumption} similarly avoids unbounded mortality manipulation, it achieves this by allowing healthcare investment to affect mortality only through an efficacy function that constrains its marginal impact.}

\end{sloppypar}

 We define the random death time of the agent $\eta$ as\footnote{Notice that $\eta$ is not an $\mathbb{F}$-stopping time. It is instead a $\mathbb{G}$-stopping time, where $ \mathcal{G}_t:=\mathcal{F}_t \vee \sigma(\mathds{1}_{\{\eta \leq u\}}; 0\leq u \leq t)$, for $t\geq 0$. This is the enlarged filtration generated by the underlying filtration $\mathbb{F}$ and the process $\{\mathds{1}_{\{\eta \leq u\}}, 0 \leq u \leq t\}$. The filtration $\mathbb{G}$ is the smallest one which contains $\mathbb{F}$ and such that $\eta$ is a $\mathbb{G}$-stopping time (see, e.g., Chapter 7 in \cite{jeanblanc2009mathematical}).}
\begin{align*}
    \eta := \inf \bigg\{t \geq 0: \int^t_0 M_u du \geq \Theta \bigg\},
\end{align*}
which is such that $\{\eta \geq t      \}= \big\{  \int^t_0 M_u du \leq \Theta  \big \}$, where we have assumed that {$\int^t_0 M_u du < \infty$ a.s.} for any $t\geq 0$. The conditional distribution function of $\eta$ is such that (see, e.g., Lemma 7.3.2.1 in \cite{jeanblanc2009mathematical}),
\begin{align}\label{2-1} 
\mathbb{P}[\eta >t | \mathcal{F}_{t}] =\exp{\bigg\{-\int^t_0 M_u du\bigg\}}, \quad t \geq 0.
\end{align}

Let $\tau$ be an $\mathbb{F}$-stopping time representing the time at which the agent invests in health. Before investing in health, the agent's health status $H^1:=\{H^1_t, 0 \leq t \leq \tau\}$ evolves as
\begin{align}\label{2-2}
dH^1_t = -\delta H_t^1dt, \ \text{for all} \ t \in (0,\tau], \ H^1_0=h>0,
\end{align}
where $\delta>0$ represents the decay rate of the health. After the agent {starts investing in healthcare at a positive predetermined rate} $I$, the agent's health status $H^2:= \{ H^2_t, t \geq \tau   \}$ {has its drift increased by} the positive constant $K$, so that
\begin{align}\label{2-3}
dH^2_t = (-\delta H^2_t+K)dt, \ \text{for all} \ t > \tau, \ H^2_\tau=he^{-\delta \tau}.
\end{align}
From (\ref{2-2}) and (\ref{2-3}) one then has that the overall health capital $H$ evolves as 
\begin{align}\label{2-2-1}
dH_t= (-\delta H_t+K\mathds{1}_{\{t > \tau\}})dt, \ \text{for all} \ t \geq 0, \ H_0 =h>0.
\end{align}

In the remark below, we comment on some features of the above health investment model and explain the connections with the existing literature. 

\begin{remark}\label{remark}

(1) The fact that health investment is positive is a standard requirement in Health Economics. Health investment is irreversible in the sense that the agent cannot reduce her health through negative expenditure. Irreversibility of investment is a key economic feature that makes health fundamentally different from financial assets or housing (see, e.g., \cite{yogo2016portfolio}).

(2) We represent the health variables $H^1$ (cf.\ (\ref{2-2})) and $H^2$ (cf.\ (\ref{2-3})) as deterministic processes for the sake of mathematical convenience. When it comes to modeling, it is indeed possible to introduce noise or even incorporate random jumps to simulate sudden morbidity shocks. However, this would introduce mathematical intricacies, which we defer to future research.

(3) The state equation (\ref{2-3}) is similar to \cite{bolin2020consumption}. In (\ref{2-3}), a constant $K$ can be regarded as a kind of health production function, mapping precautionary health investment $I$ into the gross rate of change of the health stock. In \cite{bolin2020consumption} it is assumed that the marginal product of health investment is positive. The health production function therefore captures the direct influence of precautionary health investment on the health stock, a defining feature of such investment as it is discussed in the Introduction. On the other hand, Equation (\ref{2-3}) implies that the agent is able to commit to investing forever at a fixed {rate} $I$, resulting in a constant {contribution} $K$ {to the drift} of the stock of health. This assumption might be somewhat restrictive; however, it still manages to offer fundamental economic insights. Specifically, we adopt $K:=I^\beta, \beta \in (0,1),$ in our numerical study (cf.\ Section 6), similar to that used in \cite{dalgaard2014optimal} and \cite{yogo2016portfolio}. The parameter $\beta$ specifies the degree of decreasing returns of health investment.

 (4) A classical model for force of mortality is the so-called Gompertz-Makeham law (see for instance \cite{makeham1860law}), which corresponds to 
\begin{align*}
M_t =  \bar{A} e^ {\bar{B} t}+\bar{C}, \ t > 0, \qquad M_0 =\bar{A}+\bar{C}.
\end{align*}
Here, $\bar{A}$ is known as the baseline mortality, the term $\bar{B}$ can be thought of as the `actuarial aging rate', in that its magnitude determines how fast the rate of dying will increase with the addition of extra years, while $\bar{C}$ is a constant representing age-independent mortality. As a matter of fact, before health investment, our choice of the force of mortality reads as $M_t=m^0+m^1h^{-\kappa}e^{\delta \kappa t}$ (cf.\ (\ref{2-0}) and (\ref{2-2})), which has a structure compatible {with} the Gompertz-Makeham law. {After healthcare investment, the force of mortality is given by
\begin{align*}
	M_t=m^0+m^1 \bigg[(h-\frac{K}{\delta} e^{\delta \tau} )e^{-\delta t}+\frac{K}{\delta}\bigg]^{-\kappa}.\end{align*}
When the initial health stock satisfies $h>\frac{K}{\delta}e^{\delta \tau}$ (i.e.,\ the health stock at investment time $H_\tau^1$ is larger than $\frac{K}{\delta}$), the expression reduces to a form that conforms with the Gompertz-Makeham law, exhibiting the characteristic exponential increase in mortality over time. In this case, a precautionary healthcare investment serves mainly to moderate the rate at which mortality increases (see Figure \ref{simulation}). In contrast, when $h<\frac{K}{\delta} e^{\delta \tau}$, the agent begins with a relatively poor health status. Under this condition, healthcare investment has a strong effect, enabling the system to converge to a steady-state health level ($\frac{K}{\delta}$). As a result, the force of mortality is significantly reduced, though it remains bounded below by the baseline mortality level $m^0$ (see Figure \ref{simulation1}). The deviation from the Gompertz-Makeham law highlights the realistic heterogeneity in individual responses to healthcare investments, where the aggregated outcome over a large population still conforms to the classical law. 
 Simulations of the health process $H$ from (\ref{2-2-1}) and force of mortality process $M$ from (\ref{2-0}) are displayed in Figures \ref{simulation} and \ref{simulation1}, where the parameter $\delta=0.0055$ as in \cite{hugonnier2013health}.} 


\begin{figure}[htbp]
    \centering
    \begin{subfigure}{0.45\textwidth}
        \includegraphics[width=\linewidth]{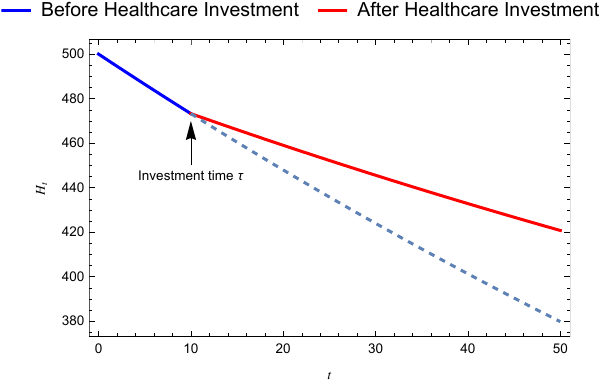}
        \caption{Health process $H_t$.}
        \label{p1}
    \end{subfigure}
        \begin{subfigure}{0.45\textwidth}
        \includegraphics[width=\linewidth]{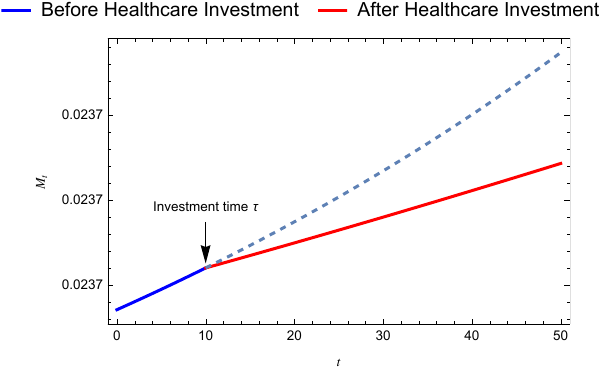}
        \caption{Force of mortality process $M_t$.}
        \label{p2}
    \end{subfigure}
    \caption{{Simulations of the health process $H_t$ and force of mortality process $M_t$, assuming a healthcare investment time $\tau=10$ and an initial health level $h=500$, with all other parameters as specified in Table \ref{tab1}. Left: Evolution of health capital (vertical axis) over time $t$ (in years, horizontal axis). Right: Evolution of the force of mortality (vertical axis) over time $t$ (in years, horizontal axis).}}
    \label{simulation}
\end{figure}

\begin{figure}[htbp]
    \centering
    \begin{subfigure}{0.45\textwidth}
        \includegraphics[width=\linewidth]{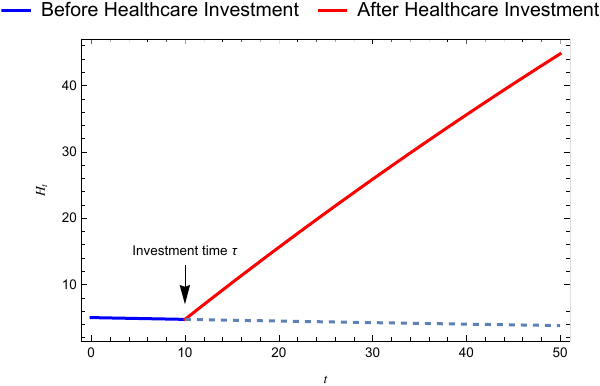}
        \caption{Health process $H_t$.}
        \label{p3}
    \end{subfigure}
        \begin{subfigure}{0.45\textwidth}
        \includegraphics[width=\linewidth]{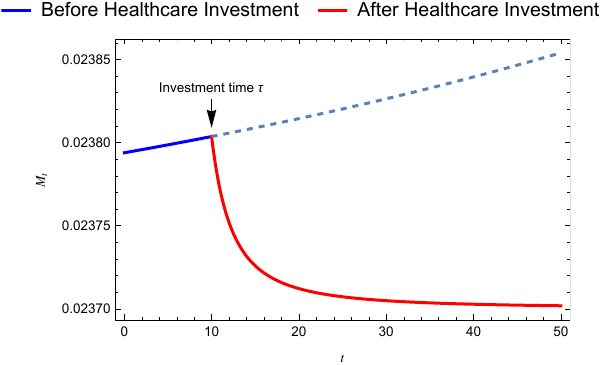}
        \caption{Force of mortality process $M_t$.}
        \label{p4}
    \end{subfigure}
    \caption{{Simulations of the health process $H_t$ and force of mortality process $M_t$, assuming a healthcare investment time $\tau=10$ and an initial health level $h=5$, with all other parameters as specified in Table \ref{tab1}. Left: Evolution of health capital (vertical axis) over time $t$ (in years, horizontal axis). Right: Evolution of the force of mortality (vertical axis) over time $t$ (in years, horizontal axis).}  }
    \label{simulation1}
\end{figure}

\end{remark}

We assume that the agent also invests in a financial market with two assets. One of them is a risk-free bond, whose price $S^0:=\{ S^0_t, t \geq 0  \}$ evolves as 
\begin{align*}
dS^0_t = rS^0_tdt, \quad S^0_0 =s^0>0,
\end{align*}
where $r>0$ is a constant risk-free rate. The second one is a stock, whose price is denoted by $S:=\{S_t, t \geq 0\}$ and it satisfies the stochastic differential equation
\begin{align*}
dS_t &= \mu S_t dt + \sigma S_t dB_t, \quad S_0 =s>0,
\end{align*}
where $\mu \in \mathbb{ R}$ and $\sigma>0$ are given constants. 

The agent also consumes from her wealth, while investing in the financial market. Denoting by $\pi_t$ the amount of wealth invested in the stock at time $t$, the agent then chooses $\pi_t$ as well as the rate of spending in consumption $c_t$ at time $t$. Therefore, the agent's wealth $X:= \{ X^{c,\pi, \tau}_t, t\geq 0 \}$ evolves as 
\begin{align}\label{2-4}
d X^{c,\pi, \tau}_t =  [\pi_t(\mu-r)+rX^{c,\pi,\tau}_t-c_t-I\mathds{1}_{\{t \geq \tau\}}]dt +\pi_t \sigma dB_t, \quad X_0^{c,\pi, \tau}= x.
\end{align}
In the following, we shall simply write $X$ to denote $X^{c,\pi, \tau}$, where needed.

\subsection{The optimization problem}
 Here and in the sequel, we set $\mathcal{O}:=\mathbb{R}_+^2$ with $\mathbb{R}_+:=(0,\infty)$, and we denote by $\mathcal{S}$ the class of $\mathbb{F}$-stopping times $\tau: \Omega \to [0,+\infty]$. Then we introduce the class of admissible strategies {as follows.}
\begin{definition}\label{admissiblecontrol}
Let $(x,h) \in \mathcal{O}$ be given and fixed. The triplet of choices $(c,\pi, \tau)$ is called an \textbf{admissible strategy} for $(x,h)$, and we write $(c,\pi, \tau) \in \mathcal{A}(x,h)$, if it satisfies the following conditions: 
\begin{enumerate}[(i)]
	\item  $c$ and $\pi$ are progressively measurable with respect to $\mathbb{F}$, $\tau \in \mathcal{S}$;
	\item $c_t \geq 0$ for all $t \geq 0$ and $\int_0^T (c_t+|\pi_t|^2)dt <\infty \ \mathbb{P}$-a.s., for any $T>0$; 
	\item  {$X^{c,\pi, \tau}_t \geq  \frac{I}{r} \mathds{1}_{\{t\geq  \tau\}}$} for all $t \geq 0$.
\end{enumerate}

\end{definition}
The constant $\frac{I}{r}$ in Condition (iii) is the present value of the future health payment of the agent, under the assumption that the agent is always {alive}. {By} (iii), the agent is able to consume and invest as long as her wealth level is above $\frac{I}{r}$  at time $t\geq \tau$. Before health investment, she should keep her wealth positive for further consumption or financial investment.

The agent's aim is then to maximize the expected utility
\begin{align}\label{2-5}
\mathbb{E}\bigg[  \int_0^{\eta  } e^{-\rho t} u(c_t, H_t)dt    \bigg]
\end{align}
over all $(c,\pi,\tau) \in \mathcal{A}(x,h).$ In (\ref{2-5}), $\rho$ is a positive discount rate and $u(c,h)=c^{\alpha}h^{1-\alpha}$, where $0<\alpha <1$. Thanks to Fubini's Theorem and the tower property, we can disentangle the market risk and the mortality risk and write
\begin{align*}
\mathbb{E}\bigg[  \int_0^{\eta  } e^{-\rho t} u(c_t,H_t)dt  \bigg] &= \mathbb{E}\bigg[  \int_0^{\infty } e^{-\rho t} u(c_t,H_t) \mathds{1}_{\{t <\eta  \}}dt  \bigg ] =   \int_0^{\infty} \mathbb{E}\bigg[e^{-\rho t}  u(c_t ,H_t) \mathds{1}_{\{ t <\eta  \}}  \bigg]dt \\
&=   \int_0^{\infty} \mathbb{E}\bigg[ \mathbb{E}\Big[e^{-\rho t}  u(c_t ,H_t) \mathds{1}_{\{ t <\eta  \}}\Big| \mathcal{F}_t\Big]  \bigg]dt \\
&= \int_0^{\infty } \mathbb{E}\bigg[e^{-\rho t} u(c_t ,H_t) \mathbb{E}\Big[\mathds{1}_{\{ t <\eta  \}}\Big| \mathcal{F}_t\Big]  \bigg]dt \\
&= \mathbb{E} \bigg[  \int_0^{\infty }  e^{-\int^t_0 (\rho+ M_u) du} u(c_t,H_t)dt   \bigg ],
\end{align*} 
where (\ref{2-1}) has also been used. Hence, given the underlying Markovian setting, the agent aims at determining
{\begin{align}\label{2-6}
V(x,h):= \sup_{(c, \pi, \tau) \in  \mathcal{A}(x,h)  } \mathbb{E}_{x,h} \bigg[ \int_0^{\infty }  e^{-\int^t_0 (\rho+ M_u) du} u(c_t, H_t)  dt     \bigg],
\end{align}}
where $\mathbb{E}_{x,h}$ denotes the expectation under $\mathbb{P}_{x,h}(\cdot):=\mathbb{P}(\cdot|X_0=x,H_0=h).$ In the rest of the paper, we shall focus on (\ref{2-6}).

\section{From control-stopping to pure stopping}\label{sec-3}
\subsection{The static budget constraint}

We define the market price of risk $\theta:= \frac{\mu-r}{\sigma}$. {Throughout the paper we assume \(\theta\neq0\) (equivalently, \(\mu\neq r\)).} For $\tau \in \mathcal{S}$, an application of It\^o's formula
leads on $\{t\geq \tau\}$ to
\begin{align}\label{3-1-2}
&e^{-rt-\theta B_t -\frac{1}{2}\theta^2 t}(X_t-\frac{I}{r})+\int^t_\tau e^{-ru-\theta B_u -\frac{1}{2}\theta^2 u} c_u du \nonumber\\
&=e^{-r\tau-\theta B_\tau -\frac{1}{2}\theta^2 \tau}(X_\tau-\frac{I}{r})+ \int^t_\tau e^{-ru-\theta B_u -\frac{1}{2}\theta^2 u}\bigg(\pi_u\sigma - \theta(X_u-\frac{I}{r})  \bigg)dB_u, 
\end{align}
and on $\{0\leq t < \tau \}$ to
\begin{align}\label{3-1-3}
&e^{-rt-\theta B_t -\frac{1}{2}\theta^2 t}X_t+\int^t_0 e^{-ru-\theta B_u -\frac{1}{2}\theta^2 u} c_u du =x+ \int^t_0 e^{-ru-\theta B_u -\frac{1}{2}\theta^2 u}(\pi_u\sigma - \theta X_u  )dB_u.
\end{align}
Since $X_t-\frac{I}{r}\mathds{1}_{\{t \geq  \tau\}} \geq  0$ for any $t \geq 0$, we can deduce that $X_\tau \geq \frac{I}{r}\geq  0$. For an admissible plan $(c,\pi,\tau) \in \mathcal{A}(x,h)$, the left-hand side of (\ref{3-1-3}) is nonnegative for $t \leq  \tau$, and so the It\^o's integral on the right-hand side is not only a continuous $\mathbb{P}$-local martingale, but also a supermartingale by {Fatou's lemma}. Thus, letting $\gamma_{t}:=e^{-rt-\theta B_t -\frac{1}{2}\theta^2 t}$, the optional sampling theorem implies the so-called budget constraint:
\begin{align}\label{3-2-1}
\mathbb{E}_{x,h}\big[{\gamma_{t \wedge \tau }}X_{t \wedge \tau}\big]+\mathbb{E}_{x,h}\bigg[ \int_0^{t \wedge \tau } \gamma_{u} c_u du                         \bigg]  \leq x, \ \text{if} \ t \geq 0.
\end{align}
By similar arguments on $(\ref{3-1-2})$ we also have 
\begin{align}\label{3-2-2}
\mathbb{E}_{x,h}\big[{\gamma_{t \vee \tau}}(X_{t \vee \tau}-\frac{I}{r})\big]+\mathbb{E}_{x,h}\bigg[ \int_\tau^{t \vee \tau } \gamma_{u} c_u du                         \bigg] { \leq \mathbb{E}_{x,h}\bigg[\gamma_\tau(  X_\tau -\frac{I}{r})\bigg]}, \ \text{if} \  t \geq 0.
\end{align}

\subsection{The agent's optimization problem after health investment}
In this subsection we consider the agent's optimization problem after health investment, and during this time period only consumption and portfolio choice have to be determined. Formally, the model in the previous section covers this case if we let $\tau =0$, and the force of mortality is specified as \(M_t^{2}\), \(t\geq 0\), where
\begin{align}\label{newm2}
	 M_t^{2}:=m^0+m^1\bigl(H_t^{2}\bigr)^{-\kappa},
	 \end{align}
and \(H_t^{2}\) is defined in \eqref{2-3}. Then, letting $\mathcal{A}_0(x,h):= \{ (c,\pi): (c,\pi,0) \in \mathcal{A}(x,h)\},$ where the subscript $0$ indicates that the investment time into health $\tau$ is equal to $0$, the agent's value function after health investment is 
\begin{align}\label{3-4-0}
\widehat{V}(x,h):=\sup_{(c, \pi) \in \mathcal{A}_0(x,h) }   \mathbb{E}_{x,h} \bigg[ \int_0^{\infty }  e^{-\int^t_0 (\rho+ M^2_u) du} u(c_t, H^2_t)  dt    \bigg],
\end{align}
with $H^2$ as defined in (\ref{2-3}) (with $\tau=0$). 

From the budget constraint (\ref{3-2-2}), recalling that $\gamma_{t}=e^{-rt-\theta B_t -\frac{1}{2}\theta^2 t}$ and for any pair $(c,\pi)\in \mathcal{A}_0(x,h)$ with a Lagrange multiplier $z>0$, we have
\begin{align}\label{3-4}
 \mathbb{E}_{x,h}& \bigg[    \int_0^{\infty }  e^{-\int^t_0 (\rho+ M^2_u) du} u(c_t, H^2_t)  dt  \bigg] \nonumber \\ \leq & \
 \mathbb{E}_{x,h} \bigg[    \int_0^{\infty }  e^{-\int^t_0 (\rho+ M^2_u) du} u(c_t, H^2_t)  dt  \bigg]   - z      \mathbb{E}_{x,h}\bigg[  \int_0^{\infty} \gamma_{t} c_t dt                         \bigg]  +z( x-\frac{I}{r})                  \nonumber               \\
 = &  \ \mathbb{E}_{x,h}\bigg[    \int_0^{\infty }  e^{-\int^t_0 (\rho+ M^2_u) du} u(c_t, H^2_t)  dt \bigg]\nonumber\\
 &- \mathbb{E}_{x,h}\bigg[ \int_0^{\infty}  e^{-\int^t_0 (\rho+ M^2_u) du}z P_t^2(h)  c_t dt             \bigg]+z( x-\frac{I}{r})  \nonumber  \\
= & \  \mathbb{E}_{x,h}\bigg[    \int_0^{\infty }  e^{-\int^t_0 (\rho+ M^2_u) du} \bigg(u(c_t, H^2_t) -z P_t^2(h)c_t\bigg) dt  \bigg]+ z( x-\frac{I}{r})  \nonumber \\
\leq & \  \mathbb{E}_{x,h}\bigg[    \int_0^{\infty }  e^{-\int^t_0 (\rho+ M^2_u) du}\widehat{u}(zP_t^2(h), H^2_t) dt  \bigg]+ z(x-\frac{I}{r}), 
\end{align}
where
\begin{align}\label{3-4-1}
P^2_t(h):= \gamma_{t}e^{\int^t_0 (\rho+ M^2_u) du } \quad \text{and} \quad \widehat{u}(z, h):= \sup_{c \geq 0}[u(c,h)-c z].
\end{align}

Let then $Z^2_t :=z P^2_t(h).$ By It\^o's formula, we obtain that the dual variable $Z^2$ satisfies
\begin{align}\label{eqz2}
dZ^2_t=  (\rho-r+ M_t^{2})Z^2_tdt - \theta Z^2_t dB_t, \quad Z^2_0= z>0,
\end{align}
and we set
\begin{align}\label{eqw}
W(z,h):=  \mathbb{E}_{z,h}\bigg[    \int_0^{\infty }  e^{-\int^t_0 (\rho+ M^2_u) du} \widehat{u}(Z^2_t, H^2_t) dt  \bigg],
\end{align}
with $ \mathbb{E}_{z,h}$ being the expectation under $\mathbb{P}_{z,h}(\cdot):= \mathbb{P}(\cdot|Z^2_0=z,H^2_0=h)$. Hence,
\begin{align*}
\mathbb{E}_{x,h} \bigg[    \int_0^{\infty }  e^{-\int^t_0 (\rho+ M^2_u) du} u(c_t, H^2_t)  dt  \bigg] \leq W(z,h)+z(x-\frac{I}{r}),
\end{align*}
 for $z>0$ and $(x,h) \in \mathcal{O}'$, where $\mathcal{O}':=\{ (x,h)\in \mathcal{O}: x\geq \frac{I}{r} \}$.

We make the following \textbf{standing assumption}.
\begin{assume}\label{assume1}
	We assume $\rho > \alpha(r+\frac{1}{2}\theta^2)+\frac{\theta^2 \alpha^2}{2(1-\alpha)}+\delta (2\kappa+1)$ throughout the paper.
\end{assume}
Assumption \ref{assume1} gives a sufficient condition to ensure the finiteness of $W$ defined in (\ref{eqw}). Moreover, this condition is used in the subsequent analysis, in particular when applying the dominated convergence theorem.

\begin{pro}\label{Wregularity}
 $W$ is finite, strictly convex with respect to $z$, and such that
 {\(W\in C^{2,1}(\mathcal{O})\) and
 \(W_z\in C^{2,1}(\mathcal O)\).} Moreover, $W$ satisfies 
\begin{align}\label{3-6}
-\widehat{\mathcal{L}}W=\widehat{u}, \ \text{on} \quad  \mathcal{O}, 
\end{align}
where 
\begin{align}\label{3-7}
\widehat{\mathcal{L}}W:= \frac{1}{2}\theta^2 z^2W_{zz}+(\rho-r+m^0+m^1h^{-\kappa})zW_z  \nonumber\\+(-\delta h+K)W_h-(\rho+m^0+m^1h^{-\kappa})W.
\end{align}
\end{pro}
\begin{proof}
	The proof is given in Appendix \ref{proWregularity}.
\end{proof}

It is possible to relate $\widehat{V}$ to $W$ through the following duality relation.
\begin{theorem}\label{dualityrelation}
For given $(x,h) \in \mathcal{O}'$,  the
following dual relations hold:
\begin{align*}
\widehat{V}(x,h) = \inf_{z >0} [ W(z,h)+z(x-\frac{I}{r}) ], \quad W(z,h) =  \sup_{x\geq  \frac{I}{r}}[\widehat{V}(x,h)-z(x-\frac{I}{r})].
\end{align*}
Moreover, {if \(x>I/r\), then} $c_t^*=\mathcal{I}^u(Z_t^{2*},H_t^2) $ with $\mathcal{I}^u(\cdot,h)$ being the inverse of the marginal utility function $u_c(\cdot, h)$ is the optimal consumption process,  $\pi^*_t= \frac{\theta}{\sigma} Z^{2*}_t W_{zz}(Z^{2*}_t,H_t^2)$ is the optimal portfolio process, and the corresponding optimal wealth process is given by $X_t^*=-W_z(Z^{2*}_t,H_t^2)+\frac{I}{r}$, where $Z_t^{2*}$ is the solution to Equation (\ref{eqz2}) with the initial condition $z^*:=z^*(x,h)$ uniquely satisfying  $x=-W_z(z^*,h)+\frac{I}{r}$.
\end{theorem}

\begin{proof}

The proof is given in Appendix \ref{produalityralation}.
\end{proof}

\subsection{The dual optimal stopping problem}

Notice that, for any $(x,h)\in \mathcal{O}$, it holds by definition of $V$ in (\ref{2-6}) and strong Markov property that 
\begin{align}\label{3-8-1}
V(x,h) \leq & \sup_{(c,\pi,\tau) \in \mathcal{A}(x,h)} \mathbb{E}_{x,h} \bigg[ \int_0^{\tau}  e^{-\int^t_0 (\rho+ M^{1}_u) du} u(c_t, H^1_t)  dt  +e^{-\int^{\tau}_0 (\rho+ M^{1}_u) du}  \widehat{V}(X_{\tau},H^{1}_{\tau})    \bigg],
\end{align}
where $M^1$ is given by
\begin{align}\label{newm1}
	 M_t^{1}:=m^0+m^1\bigl(H_t^{1}\bigr)^{-\kappa},
\end{align}
   and \(H_t^{1}=h e^{-\delta t}\) is defined in \eqref{2-2}. In the sequel, whenever necessary, we write $X^x$ (similarly, $H^{1,h}, M^{1,h}, M^{2,h}$) to stress the {dependence} of the considered processes on their initial datum. 

Now, for any $(x,h)\in \mathcal{O}$ and Lagrange multiplier $z>0$, from the budget constraint (\ref{3-2-1}) and from (\ref{3-8-1}), letting $P^1_t(h):= \gamma_{t}e^{\int^t_0 (\rho+ M^{{1,h}}_u) du}$, we have 
\begin{align}\label{3-8}
 \mathbb{E}&_{x,h} \bigg[ \int_0^{\infty }  e^{-\int^t_0 (\rho+ M_u) du} u(c_t, H_t)  dt      \bigg] \nonumber \\ \leq &
{\sup_{(c,\pi,\tau) \in \mathcal{A}(x,h)}\bigg\{\mathbb{E}_{x,h}   \bigg[ \int_0^{\tau}  e^{-\int^t_0 (\rho+ M^{1}_u) du} u(c_t, H^{1}_t)  dt   +e^{-\int^{\tau}_0 (\rho+ M^{1}_u) du}  \widehat{V}( X_{\tau},H^{1}_{\tau})    \bigg]} \nonumber \\& { - z      \mathbb{E}_{x,h} \bigg[   \gamma_{\tau}X_{\tau}+ \int_0^{\tau} \gamma_{t}c_t dt                         \bigg]  +z x    \bigg\}  }                  \nonumber       \\
=&   \sup_{(c,\pi,\tau) \in \mathcal{A}(x,h)} \mathbb{E}_{x,h}\bigg[    \int_0^{\tau }  e^{-\int^t_0 (\rho+ M^{1}_u) du} \bigg(u(c_t, H^{1}_t) -zP^1_t(h)c_t\bigg) dt  \nonumber \\&+ e^{-\int^{\tau}_0 (\rho+ M^{1}_u) du} \widehat{V}(X_{\tau},H^{1}_{\tau}) - e^{-\int^{\tau}_0 (\rho+ M^{1}_u) du} zP^1_{\tau}(h)X_{\tau}      \bigg]+z x  \nonumber \\
=&   \sup_{(c,\pi,\tau) \in \mathcal{A}(x,h)} \mathbb{E}_{x,h}\bigg[    \int_0^{\tau }  e^{-\int^t_0 (\rho+ M^{1}_u) du} \bigg(u(c_t, H^{1}_t) -z P^1_t(h) c_t \bigg) dt  \nonumber \\&+ e^{-\int^{\tau}_0 (\rho+ M^{1}_u) du} \bigg( \widehat{V}(X_{\tau},H^{1}_{\tau}) - zP^1_{\tau}(h)X_{\tau}   +z P^1_{\tau}(h)\frac{I}{r}    -z P^1_{\tau}(h) \frac{I}{r}   \bigg)   \bigg]+zx\nonumber \\
\leq & \ \sup_{\tau \in \mathcal{S}} \mathbb{E}_{z,h}\bigg[    \int_0^{\tau }  e^{-\int^t_0 (\rho+ M^{1}_u) du} \widehat{u}(Z^1_t, H^{1}_t) dt
\nonumber\\&+ e^{-\int^{\tau}_0 (\rho+ M^{1}_u) du}\bigg(W( Z^{1}_{\tau}, H_{\tau}^{1})-Z^{1}_{\tau}\frac{I}{r} \bigg) \bigg]+ zx, 
\end{align}
where we recall that $\widehat{u}(z, h)= \sup_{c \geq 0}[u(c,h)-c z]$, and we have defined $Z^{1}_t :=z P^1_t(h)$ such that
\begin{align}\label{3-9}
d Z^{1}_t =  (\rho-r+ M_t^{{1}})Z^{1}_tdt - \theta Z^{1}_t dB_t,  \ t \in (0,\tau], \quad Z^{1}_0= z>0.
\end{align}

With a slight abuse of notation, we write \(\mathbb E_{z,h}\) for the
expectation under the law of the Markov process \((Z^1,H^1)\) starting from
\((z,h)\). Hence, defining the (candidate) dual value function
\begin{align}\label{3-10}
J(z,h):=\sup_{ \tau \in \mathcal{S}
} \mathbb{E}_{z,h}\bigg[    \int_0^{\tau }  e^{-\int^t_0 (\rho+ M^{1}_u) du} \widehat{u}(Z^{1}_t, H^{1}_t) dt \nonumber\\+ e^{-\int^{\tau}_0 (\rho+ M^{1}_u) du}\Big(W(Z^{1}_{\tau}, H^{1}_{\tau})-Z^{1}_{\tau}\frac{I}{r} \Big) \bigg],
\end{align}
we have an infinite-horizon, two-dimensional optimal stopping problem, with interconnected dynamics $(Z^1,H^1)$ as in (\ref{3-9}) and (\ref{2-2}). {{In \eqref{3-10} above, as well as throughout the rest of this paper, we set 
\begin{equation*}
\label{eq:convention}
   e^{-\int^{\tau}_0 (\rho+ M^{1}_u) du}\Big(W(Z^{1}_{\tau}, H^{1}_{\tau})-Z^{1}_{\tau}\frac{I}{r} \Big) = \limsup_{t\uparrow \infty} \bigg\{ e^{-\int^{t}_0 (\rho+ M^{1}_u) du}\Big(W(Z^{1}_{t}, H^{1}_{t})-Z^{1}_{t}\frac{I}{r} \Big) \bigg\}
\end{equation*}
}
on $\{\tau=+\infty\}$.} Again, when necessary, we shall write $Z^{1,z}$ to stress the dependence of the solution to \((\ref{3-9})\) on its initial value. {We shall close the duality gap and show that $J$ is indeed the dual value function in Section \ref{sec5}  (see Theorem \ref{dualityrelation2} below).} In Section \ref{sec4}, we perform a study of (\ref{3-10}).

\section{Study of the (candidate) dual optimal stopping problem}\label{sec4}
\subsection{Preliminary properties of the value function}
To study the optimal stopping problem (\ref{3-10}), we find it convenient to introduce the function 
\begin{align}\label{3-14-1}
\widehat{J}(z,h):=J(z,h)-\widehat{W}(z,h)
\end{align}
with 
\begin{align}\label{3-14}
\widehat{W}(z,h):=W(z,h)-z\frac{I}{r}.
\end{align}
 Applying It\^o's formula to $\{ e^{-\int^{t}_0 (\rho+ M^{1}_u) du}[W(Z^{1}_{t}, H^{1}_{t})-Z^{1}_{t}\frac{I}{r} ], t \in[0,\tau \wedge T]\}$, taking conditional expectations, and letting $T\uparrow \infty$ (upon using the power representation of $W$ obtained in the proof of Proposition \ref{Wregularity} and Assumption \ref{assume1}) we have 
\begin{align*}
&\mathbb{E}_{z,h}\Big[e^{-\int^{\tau}_0 (\rho+ M^{{1}}_t) dt}\Big(W( Z^{1}_{\tau}, H^{1}_{\tau})-Z^{1}_{\tau}\frac{I}{r} \Big)\Big]= W(z,h)-z\frac{I}{r}\\
& + \mathbb{E}_{z,h}\bigg[ \int_{0}^{\tau} e^{-\int^{t}_0 (\rho+ M^{1}_u) du}\big(\mathcal{L}\widehat{W}\big)( Z^{1}_t,H^{1}_t) dt \bigg ],
\end{align*}
where, for any $F \in C^{2,1}(\mathcal{O})$, the second order differential operator $\mathcal{L}$ is such that
\begin{align}\label{defineoprator}
\mathcal{L}F:=\frac{1}{2}\theta^2 z^2F_{zz}+(\rho-r+m^0+m^1h^{-\kappa})zF_z  -\delta hF_h-(\rho+m^0+m^1h^{-\kappa})F.
\end{align}

Combining (\ref{3-10}), (\ref{3-14-1}) and (\ref{3-14}), we have 
 \begin{align}\label{4-2}
\widehat{J}(z,h) =&\sup_{ \tau \in \mathcal S
} \mathbb{E}_{z,h}\bigg[    \int_0^{\tau }  e^{-\int^t_0 (\rho+ M^{1}_u) du} \widehat{u}(Z^{1}_t, H^{1}_t) dt  \nonumber \\&+ \int_{0}^{\tau} e^{-\int^{t}_0 (\rho+ M^{1}_u) du}\big(\mathcal{L}\widehat{W}\big)( Z^{1}_t,H^{1}_t)  dt  \bigg ] \nonumber \\
=&\sup_{ \tau \in \mathcal S } \mathbb{E}_{z,h}\bigg[ \int_0^{\tau }  z  \gamma_{t}{I\,dt}- \int_0^{\tau }  e^{-\int^t_0 (\rho+ M^{1}_u) du} KW_h(Z^{1}_t,H^{1}_t)dt  \bigg],\nonumber \\
=&\sup_{ \tau \in \mathcal S } \mathbb{E}_{z,h}\bigg[    \int_0^{\tau }  e^{-\int^t_0 (\rho+ M^{1}_u) du} \bigg(Z^{1}_t I-KW_h(Z^{1}_t,H^{1}_t)\bigg)dt  \bigg] 
\end{align}
where we have used the fact that (cf.\ (\ref{3-6}))
\begin{align*}
\big(\mathcal{L}\widehat{W}\big)(z,h)&= (\mathcal{L}W)( z,h) -\mathcal{L}(z \frac{I}{r})=(\widehat{\mathcal{L}}W)(z,h)-KW_h( z,h)-\mathcal{L}(z \frac{I}{r})\\
&=-\widehat{u}(z,h)-KW_h( z,h)+I z,
\end{align*}
with $\mathcal{L}F :=\widehat{\mathcal{L}}F-KF_h$ (cf. (\ref{3-7})), for any $F \in C^{2,1}(\mathcal{O})$.

As usual in optimal stopping theory, we let
\begin{align}\label{4-1-3}
\mathcal{W}:=\{ (z,h) \in \mathcal{O}: \widehat{J}( z,h)>0  \},\quad
\mathcal{I}:=\{ (z,h) \in \mathcal{O}: \widehat{J}( z,h)=0  \}
\end{align}
be the so-called continuation (waiting) and stopping (investing) regions, respectively. We denote by $\partial \mathcal{W}$ the boundary of the set $\mathcal{W}.$ Since, for any stopping time $\tau$, the mapping
$(z,h)\mapsto
\mathbb E_{z,h}[\int_0^\tau e^{-\int_0^t(\rho+M^1_u)du}
\left(Z^1_tI-KW_h(Z^1_t,H^1_t)\right)dt]$
is continuous, $\widehat J$ is lower semicontinuous on $\mathcal O$.
Hence, $\mathcal W$ is open. As will follow from Proposition \ref{contJhat} below, $\widehat J$ is continuous.
Therefore the stopping region $\mathcal I$ is closed. We then introduce the stopping time 
\begin{align}\label{optimalstopping}
\tau^*(z,h) := \inf\{t \geq 0 :(Z^1_t,H^1_t) \in \mathcal{I}\}, \quad \mathbb{P}_{z,h}\text{-a.s.},
\end{align}
with $\inf \emptyset = + \infty$, one has that $\tau^*(z,h)$ is optimal for $\widehat{J}(z,h)$ (see, e.g., Corollary I.2.9 in \cite{peskir2006optimal}).

\begin{pro}\label{finite}
The function $\widehat{J}$ is such that $0\leq \widehat{J}(z,h)\leq \frac{Iz}{r}$ for all $(z,h) \in \mathcal{O}$.
\end{pro}
\begin{proof}
The proof is given in Appendix \ref{profinite}.

\end{proof}



Since $W_h(\cdot,h)$ is strictly decreasing (cf.\ (\ref{b3}) in Lemma \ref{whbound} in the Appendix \ref{seca}), the next monotonicity of $\widehat{J}$ follows.
\begin{pro}\label{monotonicity}
$z \mapsto \widehat{J}(z,h) $ is non-decreasing for all $h  \in \mathbb{R}_+.$
\end{pro}
On the other hand, we notice that it is not clear whether $h \to \widehat{J}(z,h)$ is monotonic or not, since it is not clear if $h \to   e^{-\int^t_0 (\rho+ M^{1}_u) du} KW_h(Z^{1}_t,H^{1}_t)$ in (\ref{4-2}) is monotonic. 

The following preliminary regularity property will be useful in the subsequent analysis of the free boundary.

\begin{pro}\label{contJhat}
{The function $\widehat{J}$ in \eqref{4-2} is continuous on $\mathcal{O}$. Consequently, by \eqref{3-14-1} and Proposition \ref{Wregularity},
the function \(J\) is continuous on \(\mathcal O\).}
\end{pro}

\begin{proof}
{The proof is given in Appendix \ref{procontJhat}.}
\end{proof}

We conclude with asymptotic limits of $\widehat{J}$.
\begin{pro}\label{limit}
$\lim_{z \to 0}\widehat{J}(z,h)=0$, $\lim_{z \to \infty}\widehat{J}( z,h)=\infty$ for all $h \in  \mathbb{R}_+.$ 
\end{pro}
\begin{proof}
The proof is given in Appendix \ref{prolimit}.
\end{proof}

\subsection{Properties of the free boundary}

In this section, we show that the boundary $\partial \mathcal{W}$ can be represented by a function $b:  \mathbb{R}_+ \to [0,\infty]$. 

First, we provide the shape of the continuation and stopping regions.  Defining $\Gamma: \mathbb{R}_+ 	\to \mathbb{R}_+$ such that 
\begin{align}\label{a11}
&\Gamma(h):=(1-\alpha)\alpha^{\frac{\alpha}{1-\alpha}}\bigg[    \int_0^{\infty }  e^{ \frac{\alpha}{1-\alpha}(rs+\frac{1}{2}\theta^2s) + \frac{\theta^2 \alpha^2s}{2(\alpha-1)^2}}e^{\int^s_0\frac{\rho+M^{{2,h}}_u}{\alpha-1}du} \Big[\frac{m^1\kappa}{1-\alpha} \nonumber\\& \times\Big(\int^s_0  \Big(h e^{-\delta u }+\frac{K}{\delta}(1-e^{-\delta u})\Big)^{-\kappa-1}e^{-\delta u}du\Big) \Big(h e^{-\delta s }+\frac{K}{\delta}(1-e^{-\delta s})\Big)+ e^{-\delta s}\Big] ds  \bigg],
\end{align}
for all $h \in  \mathbb{R}_+$, one has that $\Gamma(h)$ is finite due to Assumption \ref{assume1}.

\begin{lemma}\label{b}
There exists a free boundary 
\(b:\mathbb{R}_{+}\to [0,\infty]\)
such that 
\begin{align*}
	 \mathcal{I}
    =
    \{(z,h)\in\mathcal{O}:0<z\leq b(h)\}.
\end{align*}
   Moreover, setting
$    g(h):=(\frac{I}{K})^{\alpha-1}\Gamma(h)^{1-\alpha},
$
one has \(b(h)\leq g(h)\) for all \(h\in\mathbb R_+\), where
\(\Gamma(h)\) is defined in \eqref{a11}. {In addition, \(b\) is
upper-semicontinuous on \(\mathbb R_+\). }
\end{lemma}

\begin{proof}
Since $z \mapsto \widehat{J}(z,h)$ is nondecreasing by Proposition \ref{monotonicity}, we can define $b(h):= \sup\{z>0:\widehat{J}(z,h)\leq 0\}$ (with the convention $\sup \emptyset =0$), so that $\mathcal{I}=\{ (z,h) \in \mathcal{O} : 0<z \leq b(h)  \}$. 

Next we show $b(h)\leq g(h) $. Noticing that, due to (\ref{4-2}), 
\begin{align*}
\mathcal{R}:=\{(z,h) \in \mathcal{O} : z I-KW_h(z,h) >0         \}  \subseteq  \mathcal{W},
\end{align*}
we have $\mathcal R^C \supseteq \mathcal I $.

Recalling $W_h(z,h)$ as in (\ref{b3}), and using (\ref{a11}), we then write $W_h(z,h)= z^{\frac{\alpha}{\alpha-1}}\Gamma(h),$ so that $(z,h) \in \mathcal R^C \Longleftrightarrow z I-KW_h(z,h) \leq 0 \Longleftrightarrow z^{\frac{1}{\alpha-1}}K\Gamma(h) \geq I $ for all $h \in  \mathbb{R}_+$. But then, since $0<\alpha<1$, we have  
\begin{align*}
(z,h) \in \mathcal{R}^C \Longleftrightarrow  z \leq \bigg(\frac{I}{K\Gamma(h)}\bigg)^{\alpha-1}=\bigg(\frac{I}{K}\bigg)^{\alpha-1}\Gamma(h)^{1-\alpha},
\end{align*}
and because $\mathcal{I}=\{ (z,h) \in \mathcal{O}: 0<z\leq b(h)  \}$, we find 
\begin{align*}
	b(h)\leq g(h)=\bigg(\frac{I}{K}\bigg)^{\alpha-1}\Gamma(h)^{1-\alpha}.
\end{align*}

{Finally, since \(\widehat J\) is continuous on \(\mathcal O\) by
{Proposition \ref{contJhat}}, the stopping region \(\mathcal I\) is closed. Together
with the representation $\mathcal I=\{(z,h)\in\mathcal O:0<z\le b(h)\},$
this implies that \(b\) is upper-semicontinuous.} 

\end{proof}

\subsection{Characterization of the value function}

In order to characterize the optimal stopping boundary, we now reformulate the dual optimal stopping problem with value function $\widehat{J}$ and state variable $(z,h)$ by introducing an equivalent optimal stopping problem with value function $\widetilde{J}$ and state variable $(t,z)$.

Fix $h \in \mathbb R_+$. Throughout this subsection, the dependence of \(\widetilde J\),
\(\widetilde h\), \(\widetilde m\), and \(\widetilde b\) defined below on the fixed initial
health level \(h>0\) is suppressed for notational convenience.
 From (\ref{2-2}) and \eqref{newm1}  we have $H^1_t=he^{-\delta t}:=\widetilde{h}(t)$ and $M^{1}_t=m^0+m^1h^{-\kappa}e^{\delta \kappa t}:= \widetilde{m}(t)$ for any $t \geq 0$. Then, for \((t,z)\in[0,\infty)\times\mathbb R_+\), define
\begin{align}\label{widetJ}
\widetilde J(t,z)
:=
\sup_{\tau\in\mathcal S}
\mathbb E_z\left[
\int_0^\tau
e^{-\int_0^s(\rho+\widetilde m(t+u))\,du}
\left(
I Z_s^1-KW_h(Z_s^1,\widetilde h(t+s))
\right)ds
\right],
\end{align}
where \(Z^1\) solves (cf.\ \eqref{3-9})
\begin{align}\label{strong}
dZ_s^1
=
\big[\rho-r+\widetilde m(t+s)\big]Z_s^1\,ds
-\theta Z_s^1\,dB_s,
\qquad
Z_0^1=z>0.
\end{align}
It follows from \eqref{4-2} that
\begin{align}\label{relation}
\widetilde J(t,z)=\widehat J(z,\widetilde h(t)).
\end{align}

Also, for \(F\in C^{1,2}([0,\infty)\times\mathbb R_+)\), define the second-order differential operator \(\widetilde{\mathcal L}\) by
\begin{align*}
\widetilde{\mathcal L}F
:=
F_t
+\frac{1}{2}\theta^2z^2F_{zz}
+\big[\rho-r+\widetilde m(t)\big]zF_z
-\big(\rho+\widetilde m(t)\big)F.
\end{align*}
Moreover, setting $\widetilde{\mathcal{W}}:=\{(t,z)\in [0,\infty) \times \mathbb{R}_+: \widetilde{J}(t,z)>0 \}$, one has that $(z,\widetilde{h}(t))\in \mathcal{W} \Longleftrightarrow (t,z) \in \widetilde{\mathcal{W} }$. Analogously, set $\widetilde{\mathcal{I}}:=\{(t,z)\in [0,\infty) \times \mathbb{R}_+: \widetilde{J}(t,z)=0 \}$. The optimal stopping time for \(\widetilde J\) is therefore (cf.\ \eqref{optimalstopping})
{\[
\widetilde\tau^*(t,z)
:=
\inf\{s\geq0:(t+s,Z_s^1)\in\widetilde{\mathcal I}\},
\qquad \mathbb P_z\text{-a.s.}
\]}
Consequently,
$
\tau^*(z,\widetilde h(t))
=
\widetilde\tau^*(t,z),
 \mathbb P_z\text{-a.s.}
$
We also define the corresponding time-dependent boundary by
\begin{align}\label{relationb}
	\widetilde b(t):=b(\widetilde h(t)),
\qquad t\geq0.\end{align}
Notice that $\widetilde b(0)=b(h)$.

\begin{pro}\label{prop:Jtilde-cont}
{The function \(\widetilde J\) is continuous on
\([0,\infty)\times \mathbb R_+\).}
\end{pro}

\begin{proof}

{The claim follows immediately from the identity
\(\widetilde J(t,z)=\widehat J(z,\widetilde h(t))\) (see \eqref{relation}), the continuity of
\(\widetilde h\), and Proposition~\ref{contJhat}.}
\end{proof}

{{The next proposition follows from classical interior regularity results
for parabolic PDEs
(cf.\ Corollary~2.4.3 of \cite{krylov2008lectures}), as commonly used
in the optimal stopping literature
(see, e.g., Theorem~2.7.7 of \cite{karatzas1998methods}).

\begin{pro}\label{regularity1}
The function $\widetilde J$ belongs to
$
    C^{1,2}\bigl(
        \widetilde{\mathcal W}
        \cap
        ((0,\infty)\times\mathbb R_+)
    \bigr)
$
and satisfies
\begin{align}\label{classcial}
\begin{cases}
(\widetilde{\mathcal L}\widetilde J)(t,z)
=
-Iz+KW_h(z,\widetilde h(t)),
&
(t,z)\in
\widetilde{\mathcal W}
\cap
((0,\infty)\times\mathbb R_+),
\\[0.4em]
\widetilde J(t,z)=0,
&
(t,z)\in\widetilde{\mathcal I}.
\end{cases}
\end{align}
\end{pro}

For the regularity argument below, fix $h>0$ and extend the
time-inhomogeneous quantities to all $t\in\mathbb R$ by setting
$$
    \widetilde h(t):=he^{-\delta t},
    \qquad
    \widetilde m(t):=
    m^0+m^1h^{-\kappa}e^{\delta\kappa t}.
$$
We also extend $\widetilde J$ and $\widetilde b$ by defining
$$
    \widetilde J(t,z)
    :=
    \widehat J(z,\widetilde h(t)),
    \qquad
    \widetilde b(t)
    :=
    b(\widetilde h(t)),
    \qquad
    (t,z)\in\mathbb R\times\mathbb R_+.
$$
For $t\geq0$, these definitions agree with the original ones.

We denote the corresponding extended continuation and stopping
regions by
$$
    \widetilde{\mathcal W}^{\mathrm{ext}}
    :=
    \bigl\{
        (t,z)\in\mathbb R\times\mathbb R_+:
        \widetilde J(t,z)>0
    \bigr\},
$$
and
$$
    \widetilde{\mathcal I}^{\mathrm{ext}}
    :=
    \bigl\{
        (t,z)\in\mathbb R\times\mathbb R_+:
        \widetilde J(t,z)=0
    \bigr\}.
$$

We use the notation $W^{1,2}_p(Q)$ for the parabolic Sobolev space
of functions $u$ such that
$$
    u,\quad u_t,\quad u_z,\quad u_{zz}\in L^p(Q),
$$
where all derivatives are understood in the weak sense.

\begin{pro}\label{prop:sobolev-regularity}
For every $p\in(1,\infty)$,
$
    \widetilde J
    \in
    W^{1,2}_{p,\mathrm{loc}}
    \bigl(\mathbb R\times\mathbb R_+\bigr).
$
Moreover, $\widetilde J
    \in C^{1,2}\bigl(\widetilde{\mathcal W}^{\mathrm{ext}}\bigr)$.
For every $p>3$, the weak spatial derivative $\widetilde J_z$
admits a locally H\"older-continuous version on
$\mathbb R\times\mathbb R_+$. In particular, for every fixed
$t\in\mathbb R$, the map
$z\mapsto\widetilde J(t,z)$
belongs to $C^1(\mathbb R_+)$.

In addition, the identity
\begin{align}
    (\widetilde{\mathcal L}\widetilde J)(t,z)
    &=
    \bigl[
        -Iz+KW_h(z,\widetilde h(t))
    \bigr]
    \mathbf 1_{\{z>\widetilde b(t)\}}
    \label{eq:weak-generator-identity}
\end{align}
holds for Lebesgue-a.e.\
$(t,z)\in\mathbb R\times\mathbb R_+$. Consequently,
$
    \widetilde J_z(t,\widetilde b(t))=0
$
whenever $0<\widetilde b(t)<\infty$.
\end{pro}

\begin{proof}
We first verify that the preceding extension has the desired
probabilistic interpretation. If
$$
    h_t:=\widetilde h(t)=he^{-\delta t},
$$
then, for every $s\geq0$,
$$
    \widetilde h(t+s)=h_te^{-\delta s},
    \qquad
    \widetilde m(t+s)
    =
    m^0+m^1h_t^{-\kappa}e^{\delta\kappa s}.
$$
Hence the stopping problem starting at time $t$ coincides with the
original time-homogeneous problem having initial health level $h_t$.
The same reformulation as in \eqref{relation} therefore gives
$$
    \widetilde J(t,z)
    =
    \widehat J(z,h_t)
    =
    \widehat J(z,\widetilde h(t)),
    \qquad
    (t,z)\in\mathbb R\times\mathbb R_+.
$$
In particular, Proposition~\ref{contJhat} and the continuity of
$\widetilde h$ imply that the extended function $\widetilde J$ is
continuous on $\mathbb R\times\mathbb R_+$. The extension agrees with
the original value function on
$[0,\infty)\times\mathbb R_+$ and makes $t=0$ an interior time point.

By the structure of the stopping region obtained in Lemma~\ref{b},
$$
    \widetilde{\mathcal W}^{\mathrm{ext}}
    =
    \bigl\{
        (t,z)\in\mathbb R\times\mathbb R_+:
        z>\widetilde b(t)
    \bigr\},
$$
and
$$
    \widetilde{\mathcal I}^{\mathrm{ext}}
    =
    \bigl\{
        (t,z)\in\mathbb R\times\mathbb R_+:
        0<z\leq\widetilde b(t)
    \bigr\}.
$$

Set
$$
    \ell(t,z)
    :=
    Iz-KW_h(z,\widetilde h(t)),
    \qquad
    (t,z)\in\mathbb R\times\mathbb R_+.
$$
Fix a bounded parabolic cylinder
$$
    Q:=(t_0,t_1)\times(a,b)
    \Subset
    \mathbb R\times\mathbb R_+,
    \qquad
    0<a<b<\infty,
$$
and denote its parabolic boundary by
$$
    \partial_pQ
    :=
    \bigl(\{t_1\}\times[a,b]\bigr)
    \cup
    \bigl([t_0,t_1]\times\{a,b\}\bigr).
$$

Consider the parabolic obstacle problem on $Q$ with zero obstacle,
source term $\ell$, and boundary datum $\widetilde J$. Its strong
complementarity form, which is justified a posteriori by the
regularity established below, is
\begin{equation}\label{eq:localized-VI}
\left\{
\begin{aligned}
    &u\geq0,
        &&\text{a.e. in }Q,\\
    &\widetilde{\mathcal L}u+\ell\leq0,
        &&\text{a.e. in }Q,\\
    &u\bigl(\widetilde{\mathcal L}u+\ell\bigr)=0,
        &&\text{a.e. in }Q,\\
    &u=\widetilde J,
        &&\text{on }\partial_pQ.
\end{aligned}
\right.
\end{equation}
Equivalently, the first three relations in
\eqref{eq:localized-VI} take the form
$$
    \max\bigl\{
        \widetilde{\mathcal L}u+\ell,-u
    \bigr\}
    =0
    \qquad\text{a.e. in }Q.
$$

Since $Q$ is bounded away from $z=0$ and $\theta\neq0$, the
second-order coefficient
$\frac12\theta^2z^2$
is bounded and uniformly positive on $Q$. Moreover, all the
coefficients of $\widetilde{\mathcal L}$ are smooth and bounded on
$\overline Q$. By the explicit representation in
Proposition~\ref{Wregularity}, the function $\ell$ is also smooth and bounded on
$\overline Q$.

After the time reversal $s=t_1-t$, the localized problem becomes a
forward uniformly parabolic obstacle problem on
$(0,t_1-t_0)\times(a,b)$.
Classical variational and penalization theory yields a unique
continuous variational solution $u$. Moreover, the
maximum-principle estimate for the penalized equations, together
with the fact that the obstacle is identically zero and that $\ell$
is bounded on $\overline Q$, gives a bound on the penalization term
which is locally uniform in the penalization parameter. The interior
$L^p$-estimates for the corresponding linear uniformly parabolic
equations therefore yield, for every $p\in(1,\infty)$ and every
$Q'\Subset Q$,
\begin{equation}\label{eq:u-Sobolev}
    u\in W^{1,2}_p(Q').
\end{equation}

For the evolutionary optimal-stopping problem and its variational
formulation, including operators in non-divergence form, see
Chapter~3, Sections~4.2, 4.6, and~4.7 of
\cite{bensoussan2011applications}. For the penalization argument and
the resulting parabolic Sobolev regularity, see Chapter~1,
Section~8, in particular Theorem~8.2, of
\cite{friedman1026variational}. The argument is local and therefore
applies on every $Q'\Subset Q$; see also Theorem~8.4 therein for the
corresponding interior estimates.

Having established \eqref{eq:u-Sobolev}, we identify $u$ with the
optimal stopping value. Choose $p>3$ and apply the generalized
It\^o's formula for parabolic Sobolev functions, first on an increasing
sequence of cylinders compactly contained in $Q$ and then by
localization, to
$$
    \exp\left\{
        -\int_0^s
        \bigl(\rho+\widetilde m(t+v)\bigr)\,dv
    \right\}
    u(t+s,Z_s^1);
$$
see Lemma~8.1 and Theorem~8.5 in
\cite{bensoussan2011applications}.

Using the variational inequalities, the boundary condition on
$\partial_pQ$, and stopping upon first entry into the contact set
$\{u=0\}$ before the first exit time of $(t+s,Z_s^1)$ from $Q$, the
standard verification argument shows that $u$ coincides with the
corresponding localized optimal stopping value. By the dynamic
programming principle, this localized value agrees with
$\widetilde J$ on $Q$. Hence
$$
    u=\widetilde J
    \qquad\text{on }Q.
$$

Since $Q\Subset\mathbb R\times\mathbb R_+$ was arbitrary,
\eqref{eq:u-Sobolev} implies that, for every $p\in(1,\infty)$,
\begin{equation}\label{eq:J-Sobolev}
    \widetilde J
    \in
    W^{1,2}_{p,\mathrm{loc}}
    \bigl(\mathbb R\times\mathbb R_+\bigr).
\end{equation}

On the open set $\widetilde{\mathcal W}^{\mathrm{ext}}$, the obstacle
is inactive and $\widetilde J$ solves
$\widetilde{\mathcal L}\widetilde J+\ell=0$.
Since the coefficients of $\widetilde{\mathcal L}$ and the source
term $\ell$ are smooth, standard interior parabolic regularity yields
$\widetilde J
    \in
C^{1,2}\bigl(\widetilde{\mathcal W}^{\mathrm{ext}}\bigr)
$.
In particular, this regularity holds in a neighborhood of every
point
$(0,z)\in\widetilde{\mathcal W}^{\mathrm{ext}}$.

We next identify the weak generator of $\widetilde J$. On the
continuation region, the complementarity relations give
$$
    \widetilde{\mathcal L}\widetilde J
    =
    -\ell
    =
    -Iz+KW_h(z,\widetilde h(t))
    \qquad
    \text{a.e. on }
    \widetilde{\mathcal W}^{\mathrm{ext}}.
$$
On the other hand,
$$
    \widetilde J=0
    \qquad
    \text{on }
    \widetilde{\mathcal I}^{\mathrm{ext}}.
$$
Since
$\widetilde J\in W^{1,2}_{p,\mathrm{loc}}$, the standard level-set
property of Sobolev functions yields
$$
    \widetilde J_t=\widetilde J_z=0
    \qquad
    \text{a.e. on }
    \widetilde{\mathcal I}^{\mathrm{ext}}.
$$

Moreover, by \eqref{eq:J-Sobolev}, for almost every fixed $t$, the map
$z\mapsto\widetilde J(t,z)$ belongs to $W^{2,p}_{\mathrm{loc}}(\mathbb R_+)$ and is identically
zero on $(0,\widetilde b(t))$. It follows that
$$
    \widetilde J_{zz}(t,z)=0
    \qquad
    \text{for a.e. }z<\widetilde b(t).
$$
Since $\widetilde b$ is Borel measurable, its graph
$
    \bigl\{
        (t,z)\in\mathbb R\times\mathbb R_+:
        z=\widetilde b(t)
    \bigr\}
$
has zero two-dimensional Lebesgue measure. Consequently,
$$
    \widetilde{\mathcal L}\widetilde J=0
    \qquad
    \text{a.e. on }
    \widetilde{\mathcal I}^{\mathrm{ext}}.
$$
Combining the identities on the continuation and stopping regions,
we obtain
$$
    (\widetilde{\mathcal L}\widetilde J)(t,z)
    =
    \bigl[
        -Iz+KW_h(z,\widetilde h(t))
    \bigr]
    \mathbf 1_{\{z>\widetilde b(t)\}}
$$
for Lebesgue-a.e.\
$(t,z)\in\mathbb R\times\mathbb R_+$.

Finally, take $p>3$. Since the spatial dimension is one, the
parabolic Sobolev embedding theorem applied to
\eqref{eq:J-Sobolev} implies that the weak spatial derivative
$\widetilde J_z$ admits a locally H\"older-continuous version on
$\mathbb R\times\mathbb R_+$. In particular, for every fixed
$t\in\mathbb R$,
$z\mapsto\widetilde J(t,z)$ belongs to $C^1(\mathbb R_+)$. Since
$$
    \widetilde J(t,z)=0,
    \qquad
    0<z\leq\widetilde b(t),
$$
we have
$$
    \widetilde J_z(t,z)=0,
    \qquad
    0<z<\widetilde b(t).
$$
Continuity of $\widetilde J_z$ therefore gives
$
    \widetilde J_z(t,\widetilde b(t))=0
$
whenever $0<\widetilde b(t)<\infty$.
\end{proof}

We summarize the regularity and smooth-fit properties obtained above
in the following corollary. These properties will be used in the
derivation of the boundary integral equation.

\begin{corollary}\label{cor:smooth-fit}
Let $\widetilde J$ also denote the extension to
$\mathbb R\times\mathbb R_+$ introduced above. Then, for every
$p\in(1,\infty)$,
$
    \widetilde J
    \in
    W^{1,2}_{p,\mathrm{loc}}
    \bigl(\mathbb R\times\mathbb R_+\bigr).
$ On the original time domain, define
$$
    \widetilde{\mathcal W}
    :=
    \bigl\{
        (t,z)\in[0,\infty)\times\mathbb R_+:
        \widetilde J(t,z)>0
    \bigr\},
$$
and
$$
    \widetilde{\mathcal I}
    :=
    \bigl\{
        (t,z)\in[0,\infty)\times\mathbb R_+:
        \widetilde J(t,z)=0
    \bigr\}.
$$
The extension of $\widetilde J$ is $C^{1,2}$ in an open neighborhood
of every point of $\widetilde{\mathcal W}$; in this sense,
$
    \widetilde J\in C^{1,2}(\widetilde{\mathcal W}).
$
Furthermore, for every fixed $t\geq0$, the map
$z\mapsto\widetilde J(t,z)$
belongs to $C^1(\mathbb R_+)$, and
$$
\begin{cases}
(\widetilde{\mathcal L}\widetilde J)(t,z)
=
-Iz+KW_h(z,\widetilde h(t)),
&
(t,z)\in\widetilde{\mathcal W},
\\[0.4em]
\widetilde J(t,z)=0,
&
(t,z)\in\widetilde{\mathcal I},
\\[0.4em]
\widetilde J_z(t,\widetilde b(t))=0,
&
t\geq0,\quad 0<\widetilde b(t)<\infty.
\end{cases}
$$
In addition,
$$
    (\widetilde{\mathcal L}\widetilde J)(t,z)
    =
    \bigl[
        -Iz+KW_h(z,\widetilde h(t))
    \bigr]
    \mathbf 1_{\{z>\widetilde b(t)\}}
$$
for Lebesgue-a.e.\
$(t,z)\in[0,\infty)\times\mathbb R_+$.
\end{corollary}
}}

We are now in a position to derive a nonlinear integral equation that
characterizes the free boundary $\widetilde b$. The boundary $b(h), h>0$ can be {recovered} through $\widetilde{b}(0)$, as $b(h)=\widetilde{b}(0)$ (see \eqref{relationb}). As a byproduct, we obtain an integral
representation of the value function \(\widetilde J\). The result relies on
the regularity properties established above and on a weak version of Dynkin's
formula.

\begin{theorem}\label{thm:Jtilde-representation}
For every \((t,z)\in[0,\infty)\times \mathbb R_+ \), the value function
\(\widetilde J\) defined in \eqref{widetJ} admits the representation
\begin{equation}\label{eq:Jtilde-representation}
\widetilde J(t,z)
=
\mathbb E_z\!\left[
\int_0^\infty
e^{-\int_0^s(\rho+\widetilde m(t+u))\,du}
\Big(
IZ_s^1-KW_h(Z_s^1,\widetilde h(t+s))
\Big)
\mathds{1}_{\{Z_s^1\geq\widetilde b(t+s)\}}
\,ds
\right].
\end{equation}
{Consequently, for every \(t\geq0\) with
{\(0< \widetilde b(t)\leq g(\widetilde h(t))<\infty\)}, where \(g\) is given in
Lemma~\ref{b}, the free boundary satisfies
\begin{equation}\label{eq:boundary-integral-equation}
0
=
\mathbb E_{\widetilde b(t)}\!\left[
\int_0^\infty
e^{-\int_0^s(\rho+\widetilde m(t+u))\,du}
\Big(
IZ_s^1
-
KW_h(Z_s^1,\widetilde h(t+s))
\Big)
\mathds{1}_{\{Z_s^1\geq\widetilde b(t+s)\}}
\,ds
\right].
\end{equation}
}
\end{theorem}

\begin{proof}
We prove the representation \eqref{eq:Jtilde-representation}. Fix
\((t,z)\in[0,\infty)\times \mathbb R_+\), and let \(Z^1\) solve
\eqref{strong} with initial condition \(Z_0^1=z\). For \(T>0\) and \(n\ge1\),
set $\mathcal K_n:=[t,t+T]\times[1/n,n],$
and define
\[
    \tau_n
    :=
    \inf\{s\ge0:(t+s,Z_s^1)\notin \mathcal  K_n\}.
\]
For \(n\) sufficiently large, \((t,z)\in \mathcal K_n\).

{{
Fix $p>3$. By Proposition~\ref{prop:sobolev-regularity},
$\widetilde J\in W^{1,2}_{p,\mathrm{loc}}
(\mathbb R\times\mathbb R_+)$. Hence the generalized It\^o's formula
for parabolic Sobolev functions (see, for example, Lemma~8.1 and Theorem~8.5 in
\cite{bensoussan2011applications}, pp.~183-186) yields
\begin{align}
    \widetilde J(t,z)
    &=
    \mathbb E_z\bigg[
        e^{-\int_0^{\tau_n\wedge T}
        (\rho+\widetilde m(t+u))\,du}
        \widetilde J
        \bigl(
            t+\tau_n\wedge T,
            Z_{\tau_n\wedge T}^1
        \bigr)
    \bigg]
    \notag\\
    &\quad
    -
    \mathbb E_z\bigg[
        \int_0^{\tau_n\wedge T}
        e^{-\int_0^s(\rho+\widetilde m(t+u))\,du}
        (\widetilde{\mathcal L}\widetilde J)
        (t+s,Z_s^1)\,ds
    \bigg].
    \label{eq:sobolev-dynkin}
\end{align}

By \eqref{eq:weak-generator-identity},
\[
    (\widetilde{\mathcal L}\widetilde J)(v,\zeta)
    =
    \bigl[
        -I\zeta+KW_h(\zeta,\widetilde h(v))
    \bigr]
    \mathbf 1_{\{\zeta>\widetilde b(v)\}}
\]
for Lebesgue-a.e.\ $(v,\zeta)\in\mathbb R\times\mathbb R_+$.
Moreover, for every $s>0$, the random variable $Z_s^1$ admits a
density with respect to Lebesgue measure. Therefore, by Fubini's
theorem, the preceding almost-everywhere identity can be evaluated
along $(t+s,Z_s^1)$ for $\mathbb P_z\otimes ds$-a.e.\ $(\omega,s)$.
Substituting it into \eqref{eq:sobolev-dynkin} gives
\begin{align}
    \widetilde J(t,z)
    &=
    \mathbb E_z\bigg[
        e^{-\int_0^{\tau_n\wedge T}
        (\rho+\widetilde m(t+u))\,du}
        \widetilde J
        \bigl(
            t+\tau_n\wedge T,
            Z_{\tau_n\wedge T}^1
        \bigr)
    \bigg]
    \notag\\
    &\quad
    +
    \mathbb E_z\bigg[
        \int_0^{\tau_n\wedge T}
        e^{-\int_0^s(\rho+\widetilde m(t+u))\,du}
        \bigl(
            IZ_s^1
            -KW_h(Z_s^1,\widetilde h(t+s))
        \bigr)
        \mathbf 1_{\{Z_s^1 \geq \widetilde b(t+s)\}}
        \,ds
    \bigg].
    \label{eq:localized-representation}
\end{align}

In \eqref{eq:localized-representation} above, we have used that, since $Z_s^1$ has a density for every $s>0$ and
$\widetilde b$ is Borel measurable,
$\mathbb P_z
    \bigl(
        Z_s^1=\widetilde b(t+s)
    \bigr)=0,
 s>0,$
so that the strict inequality in the indicator can be equivalently replaced by
a non-strict inequality.
}}

{We now fix \(T>0\) and let \(n\to\infty\). Since \(Z^1\) is continuous and
strictly positive, we have
$    \tau_n\wedge T\rightarrow T,
   n\to\infty,
    \text{a.s.}
$
Moreover, on \([0,T]\), \(\widetilde h(t+s)= he^{-\delta(t+s)}\) is bounded away from zero and
infinity. {Hence, also employing  Lemma \ref{whbound}, there exists a constant
\(C_T>0\) such that}
\[
\left|
    IZ_s^1-KW_h(Z_s^1,\widetilde h(t+s))
\right|
\le
C_T\left(
    1+Z_s^1+(Z_s^1)^{\frac{\alpha}{\alpha-1}}
\right),
\qquad 0\le s\le T.
\]
Recalling \eqref{strong}, one has that the right-hand side is $\mathbb{P}\otimes dt$-integrable on
$\Omega \times [0,T]$. Hence, by dominated convergence,
\[
\begin{aligned}
&\mathbb E_z\bigg[
\int_0^{\tau_n\wedge T}
e^{-\int_0^s(\rho+\widetilde m(t+u))\,du}
\Big(
    IZ_s^1-KW_h(Z_s^1,\widetilde h(t+s))
\Big)
\mathds{1}_{\{Z_s^1\geq \widetilde b(t+s)\}}
\,ds
\bigg]
\\
&\quad \xrightarrow{n\to\infty}
\mathbb E_z\bigg[
\int_0^T
e^{-\int_0^s(\rho+\widetilde m(t+u))\,du}
\Big(
    IZ_s^1-KW_h(Z_s^1,\widetilde h(t+s))
\Big)
\mathds{1}_{\{Z_s^1\geq \widetilde b(t+s)\}}
\,ds
\bigg].
\end{aligned}
\]

As for the first expectation in \eqref{eq:localized-representation}, Proposition \ref{finite} and \eqref{relation} yield
$    0\le \widetilde J(t,z)\le \frac Ir z.$
Thus
\[
\begin{aligned}
0
&\le
e^{-\int_0^{\tau_n\wedge T}(\rho+\widetilde m(t+u))\,du}
\widetilde J(t+\tau_n\wedge T,Z_{\tau_n\wedge T}^1)
\le
\frac Ir
e^{-\int_0^{\tau_n\wedge T}(\rho+\widetilde m(t+u))\,du}
Z_{\tau_n\wedge T}^1 .
\end{aligned}
\]
Using the explicit solution of \(Z^1\) (cf.\ \eqref{strong}), for \(0\le s\le T\),  we have
\[
e^{-\int_0^s(\rho+\widetilde m(t+u))\,du}Z_s^1
=
z\exp\left\{
    -rs-\frac12\theta^2s-\theta B_s
\right\}\le
z\exp\left\{
    |\theta|\sup_{0\le u\le T}|B_u|
\right\}.
\]
The random variable on the right-hand side belongs to $L^1(\Omega,\mathcal{F},\mathbb{P})$. Since
\(\tau_n\wedge T\to T\) a.s.\ and \(\widetilde J\) is continuous (see {Proposition \ref{prop:Jtilde-cont}}), dominated
convergence applies to the first expectation in \eqref{eq:localized-representation} as well. Therefore, letting
\(n\to\infty\), we obtain the finite-horizon identity
\begin{align}\label{re1}
	\begin{aligned}
\widetilde J(t,z)
=&\
\mathbb E_z\bigg[
\int_0^T
e^{-\int_0^s(\rho+\widetilde m(t+u))\,du}
\Big(
    IZ_s^1-KW_h(Z_s^1,\widetilde h(t+s))
\Big)
\mathds{1}_{\{Z_s^1\geq \widetilde b(t+s)\}}
\,ds
\bigg]
\\
&+
\mathbb E_z\bigg[
e^{-\int_0^T(\rho+\widetilde m(t+u))\,du}
\widetilde J(t+T,Z_T^1)
\bigg].
\end{aligned}\end{align}
}

{It remains to let \(T\to\infty\). Again, by Proposition \ref{finite} and
\eqref{relation}, we have 
\[
\begin{aligned}
0
&\le
\mathbb E_z\bigg[
e^{-\int_0^T(\rho+\widetilde m(t+u))\,du}
\widetilde J(t+T,Z_T^1)
\bigg]
\le
\frac Ir
\mathbb E_z\bigg[
e^{-\int_0^T(\rho+\widetilde m(t+u))\,du}
Z_T^1
\bigg] \leq \frac{I}{r} z e^{-rT} .
\end{aligned}
\]
Since \(r>0\), the right-hand side of the last displayed equation above converges to zero as $T\uparrow \infty$. Moreover, Lemma \ref{whbound} and Assumption \ref{assume1} imply
\[
\mathbb E_z\bigg[
\int_0^\infty
e^{-\int_0^s(\rho+\widetilde m(t+u))\,du}
\left|
    IZ_s^1-KW_h(Z_s^1,\widetilde h(t+s))
\right|
\,ds
\bigg]<\infty.
\]
Therefore, by dominated convergence, letting \(T\to\infty\) in the
finite-horizon identity (\ref{re1}) gives
\[
\widetilde J(t,z)
=
\mathbb E_z\bigg[
\int_0^\infty
e^{-\int_0^s(\rho+\widetilde m(t+u))\,du}
\Big(
    IZ_s^1-KW_h(Z_s^1,\widetilde h(t+s))
\Big)
\mathds{1}_{\{Z_s^1\geq \widetilde b(t+s)\}}
\,ds
\bigg].
\]
This proves \eqref{eq:Jtilde-representation}.}

Finally, let \(t\ge0\) be such that {\(0< \widetilde b(t)\leq g(\widetilde h(t))<\infty\)}, and let
\(Z^1\) solve \eqref{strong} with \(Z_0^1=\widetilde b(t)\). By value matching,
$    \widetilde J(t,\widetilde b(t))=0.$
Taking \(z=\widetilde b(t)\) in \eqref{eq:Jtilde-representation} yields
\eqref{eq:boundary-integral-equation}.
\end{proof}

\begin{remark}
{It is worth emphasizing that, in the present setting, we obtain upper
semicontinuity of the free boundary \(\widetilde b\), but not its continuity.
The main reason is that the transformed problem is genuinely
time-inhomogeneous: the dependence on \(t\) enters both through the mortality
rate \(\widetilde m(t)\) and through the health process \(\widetilde h(t)\).
As a consequence, no monotonicity of the map
\(t\mapsto \widetilde J(t,z)\) is available in general, and the usual
monotonicity-based arguments for the continuity of the free boundary (see, e.g.,\ \cite{de2015note}) cannot
be applied.} {One possible alternative would be to prove stronger regularity of the
free boundary directly, without relying on monotonicity. However, the arguments
developed in \cite{de2019lipschitz,de2019free} are not directly available in
the present time-inhomogeneous setting. In particular, the application of an
implicit-function-type argument would require suitable uniform bounds, which are difficult to verify
in our framework.}

{For the same reason, the standard four-step procedure for proving uniqueness
of the boundary via an integral equation, as in Theorem~3.1 of 
\cite{peskir2005american}, is not directly applicable here. That procedure
typically requires a continuous boundary.  Therefore, the integral equation derived above should be understood as a
characterization satisfied by the free boundary, rather than as a uniqueness
result within a class of continuous boundaries. This characterization is
nevertheless sufficient for our purposes: in Section~\ref{sec:numerics}, we
solve the integral equation numerically and use the resulting boundary to
illustrate the model's financial implications.}
\end{remark}

\section{Optimal solution in terms of the primal variables}\label{sec5}
 In the previous section, we studied the properties of the dual value function $J(z,h)$ and used $(z,h)$, where $z$ denotes {the} dual variable and $h$ denotes health capital, as the coordinate system for the study. In this section, we {return to the study} of the value function $V(x,h)$ in the original coordinate system $(x,h)$, where $x$ denotes the wealth of the agent. 

\begin{pro}\label{convex}
For every fixed \(h\in\mathbb R_+\), the function \(z\mapsto J(z,h)\)
belongs to \(C^1(\mathbb R_+)\), and is strictly decreasing and strictly convex
on \(\mathbb R_+\). Moreover,
\[
\lim_{z\downarrow0}(-J_z(z,h))=+\infty,
\qquad
\lim_{z\uparrow\infty}(-J_z(z,h))=0.
\]
\end{pro}

\begin{proof}
We first show the $C^1$-property of $J(\cdot,h)$.  Fix \(h\in\mathbb R_+\), and notice that from (\ref{3-14-1}) and (\ref{3-14}) we have  \[
    J(z,h)=\widehat J(z,h)+W(z,h)-z\frac Ir .
\]
Moreover, by the time-inhomogeneous reformulation in (\ref{relation}), one has
$    \widehat J(z,h)=\widetilde J(0,z).$
Hence, by Proposition~\ref{prop:sobolev-regularity}, the map
\(z\mapsto \widehat J(z,h)\) belongs to \(C^1(\mathbb R_+)\). Since
\(W(\cdot,h)\in C^2(\mathbb R_+)\), it follows that $    J(\cdot,h)\in C^1(\mathbb R_+).$

Next, for
\(\tau\in\mathcal S\), define
\begin{align}\label{barj}
\bar J(z,h;\tau)
:=
\mathbb E_{z,h}\bigg[
&\int_0^\tau
e^{-\int_0^t(\rho+M_u^{1})\,du}
\widehat u(Z_t^1,H_t^1)\,dt
\nonumber \\
&+
e^{-\int_0^\tau(\rho+M_u^{1})\,du}
\left(
W(Z_\tau^1,H_\tau^1)-Z_\tau^1\frac{I}{r}
\right)
\bigg],
\end{align}
so that $    J(z,h)=\sup_{\tau\in\mathcal S}\bar J(z,h;\tau).$ 

We first establish the strict monotonicity and strict convexity of \(J\) with
respect to \(z\), by adapting the arguments of Lemma~8.1 in
\cite{karatzas2000utility}. To that end, recall that by \eqref{3-9-3}, the function \(z\mapsto W(z,h)\) is strictly convex
and strictly decreasing. Moreover, \(z\mapsto\widehat u(z,h)\) is strictly
convex and strictly decreasing by \eqref{3-9-1}. Hence, for every fixed
\(\tau\in\mathcal S\), the map $z\mapsto \bar J(z,h;\tau)$
is strictly convex and strictly decreasing.

{Let \(0<z_1<z_2<\infty\), and let \(\tau_2\) be an optimal stopping time for
\(J(z_2,h)\). Then, we have 
\[
    J(z_1,h)
    \geq
    \bar J(z_1,h;\tau_2)
    >
    \bar J(z_2,h;\tau_2)
    =
    J(z_2,h),
\]
where the strict inequality follows from the strict monotonicity of
\(\bar J(\cdot,h;\tau_2)\). Hence, \(J(\cdot,h)\) is strictly decreasing.}

As for the strict convexity, let again \(0<z_1<z_2<\infty\), let \(\lambda\in(0,1)\), and set $z_\lambda:=\lambda z_1+(1-\lambda)z_2.$
Let \(\tau_\lambda\) be optimal for \(J(z_\lambda,h)\). By the strict convexity
of \(\bar J(\cdot,h;\tau_\lambda)\),
\[
\begin{aligned}
J(z_\lambda,h)
&=
\bar J(z_\lambda,h;\tau_\lambda) <
\lambda \bar J(z_1,h;\tau_\lambda)
+(1-\lambda)\bar J(z_2,h;\tau_\lambda)
\\
&\leq
\lambda J(z_1,h)+(1-\lambda)J(z_2,h),
\end{aligned}
\]
which allows {us} to conclude that \(J(\cdot,h)\) is strictly convex.

{It remains to prove the limiting behavior of \(J_z\). By Proposition
\ref{finite}, we have 
\[
    W(z,h)-z\frac Ir
    \leq J(z,h)\leq W(z,h).
\]
We first consider the limit as \(z\downarrow0\). By the explicit representation
of \(W\) in (\ref{3-9-3}), one has $    \lim_{z\downarrow0}W(z,h)=+\infty.$
Therefore, the lower bound above implies $    \lim_{z\downarrow0}J(z,h)=+\infty.$
Fix now \(y>0\) and notice that since \(J(\cdot,h)\) is convex, for every \(0<z<y\), one has 
\[
    J_z(z,h)
    \leq
    \frac{J(y,h)-J(z,h)}{y-z}.
\]
Letting \(z\downarrow0\) in the latter equation and using that \(\lim_{z\downarrow 0}J(z,h)= +\infty\) we find
\[
    \lim_{z\downarrow0}(-J_z(z,h))=+\infty.
\]}

{We now consider the limit as \(z\uparrow\infty\), whose proof is {slightly} more involved. From the upper bound
\(J(z,h)\leq W(z,h)\) and from \eqref{3-9-3}, we have
\begin{align}\label{re2}
	 \limsup_{z\uparrow\infty}\frac{J(z,h)}{z}
    \leq
    \lim_{z\uparrow\infty}\frac{W(z,h)}{z}
    =0.
    \end{align}
We next show that the reverse inequality also holds. By \eqref{4-2}, for any
fixed \(T>0\), choosing the deterministic stopping time \(\tau\equiv T\) gives
\[
\widehat J(z,h)
\ge
\mathbb E_{z,h}\left[
\int_0^T
e^{-\int_0^t(\rho+M_u^{1})du}
\Big(
IZ_t^{1}-KW_h(Z_t^{1},H_t^1)
\Big)dt
\right].
\]
Noticing that
\[
\mathbb E_{z,h}\left[
e^{-\int_0^t(\rho+M_u^{1})du}Z_t^{1}
\right]
=
ze^{-rt},
\]
and recalling the power form of \(W_h\) in (\ref{b3}) as well as 
Assumption~\ref{assume1}, one has that
\begin{align*}
    \liminf_{z\uparrow\infty}\frac{\widehat J(z,h)}{z}
    \ge
    \frac Ir(1-e^{-rT}).
\end{align*}
Letting \(T\to\infty\), and using the upper bound
\(\widehat J(z,h)\le \frac Ir z\), we obtain
$    \lim_{z\uparrow\infty}\frac{\widehat J(z,h)}{z}
    =
    \frac Ir.
$
Therefore, since
\[
    J(z,h)=\widehat J(z,h)+W(z,h)-z\frac Ir
    \quad\text{and}\quad
    \lim_{z\uparrow \infty}  \frac{W(z,h)}{z}=0,
\]
we conclude that
\begin{align}\label{limitz}
	\lim_{z\uparrow\infty}\frac{J(z,h)}{z}=0.
	\end{align}
    }

{Finally, since \(J(\cdot,h)\) is convex and decreasing, \(J_z(\cdot,h)\) is
nondecreasing and nonpositive. Hence
\[
    \ell:=\lim_{z\uparrow\infty}J_z(z,h)
\]
exists in {\((-\infty,0]\), since \(J_z(\cdot,h)\) is
nondecreasing and \(J_z(z,h)>-\infty\) for each \(z>0\)}. If \(\ell<0\), then for all
\(\varepsilon>0\), there exists $N>0$ such that for all $z>N$ we have $    |J_z(z,h) -\blue{\ell}|\leq \varepsilon$.
Thus,  given $z$ and $z_0$ large enough such that $z>z_0$, \(J(z,h)\le J(z_0,h)+ (\varepsilon+{\ell})(z-z_0)\), which implies for $\varepsilon=-\frac{{\ell}}{2}$, for example, 
\[
    \limsup_{z\uparrow\infty}\frac{J(z,h)}{z}\le\frac{{\ell}}{2} <0,
\]
contradicting \eqref{limitz}. Therefore \(\ell=0\);
that is, also thanks to (\ref{re2}), we have $    \lim_{z\uparrow\infty}(-J_z(z,h))=0.$}

\end{proof}

We now provide the following theorem that establishes a dual relation between the original problem (\ref{2-6}) and the optimal stopping problem (\ref{3-10}) and the corresponding  optimal strategies.

\begin{theorem}\label{dualityrelation2}
For given $(x,h) \in \mathcal{O}$, the following duality relations hold:
\begin{align*}
V(x,h) = \inf_{z>0}[{J}(z,h)+zx], \quad J(z,h) = \sup_{x>0} [V(x,h)-zx].
\end{align*}  
{Moreover, let the optimal investment time $\tau^*:=\tau^*(z^*,h)$. Before healthcare investment, the
optimal controls are given, for $0\le t<\tau^*$, by
\[
c_t^*=\mathcal{I}^u(Z_t^{1,*},H_t^1),
\qquad
\pi_t^*=\frac{\theta}{\sigma}
Z_t^{1,*}J_{zz}(Z_t^{1,*},H_t^1),
\]}
and the corresponding {pre-investment} optimal wealth process is $X_t^*=-J_z(Z_t^{1*},H_t^1),$ where \(Z^{1*}\) is the solution to Equation \eqref{3-9} with initial
condition \(z^*:=z^*(x,h)\) satisfying \(x=-J_z(z^*,h)\). 
{After $\tau^*$, the optimal controls are the post-investment controls described in Theorem~\ref{dualityrelation}.}
\end{theorem}
\begin{proof}
The proof is given in Appendix \ref{produalityrelation2}.
\end{proof}

By Theorem~\ref{dualityrelation2}, for every \((x,h)\in\mathcal O\), $V(x,h)=\inf_{z>0}\big[J(z,h)+zx\big].$
For fixed \(h\in\mathbb R_+\), Proposition~\ref{convex} implies that \(z\mapsto J(z,h)\) is strictly convex and strictly decreasing. Hence the map $z\mapsto -J_z(z,h)$
is continuous, strictly decreasing, and maps \((0,\infty)\) onto
\((0,\infty)\). Therefore, for every \(x>0\), there exists a unique
\(z^*(x,h)>0\) satisfying $x=-J_z(z^*(x,h),h).$
The infimum in the dual representation is attained at \(z^*(x,h)\), and
\[
V(x,h)=J(z^*(x,h),h)+xz^*(x,h).
\]

Moreover, for each fixed \(h\in\mathbb R_+\), the map $x\mapsto z^*(x,h)$
is continuous and strictly decreasing from \((0,\infty)\) onto \((0,\infty)\).
Its inverse is
$
z\mapsto x^*(z,h):=-J_z(z,h),
$
which is also continuous and strictly decreasing from \((0,\infty)\) onto
\((0,\infty)\).

Let us now define, for \(h\in\mathbb R_+\),
\begin{equation}\label{newb}
\left\{
\begin{aligned}
{\widehat b(h)}
&{:=
\begin{cases}
x^*(b(h),h), & b(h)>0,\\
+\infty, & b(h)=0,
\end{cases}}\\
\widehat{\mathcal W}
&:=\{(x,h)\in\mathcal O:(z^*(x,h),h)\in\mathcal W\},\\
\widehat{\mathcal I}
&:=\{(x,h)\in\mathcal O:(z^*(x,h),h)\in\mathcal I\}.
\end{aligned}
\right.
\end{equation}
{When \(b(h)=0\), this convention is consistent with the limiting
behavior of \(x^*(b(h),h)\), since Proposition~\ref{convex} gives
\(\lim_{z\downarrow0}x^*(z,h)=\lim_{z\downarrow0}(-J_z(z,h))=+\infty\).}

Then, by Lemma~\ref{b} and the strict monotonicity of
\(x\mapsto z^*(x,h)\), we have
\[
\widehat{\mathcal W}
=
\{(x,h)\in\mathcal O:0<x<\widehat b(h)\},
\qquad
\widehat{\mathcal I}
=
\{(x,h)\in\mathcal O:x\geq \widehat b(h)\}.
\]
Consequently, the optimal healthcare investment time in the primal variables is
\[
\tau^*(x,h)
=
\inf\left\{
s\geq0:
X_s^{*}\geq \widehat b(H_s^{1,h})
\right\},
\]
where \(X^{*}\) denotes the optimal wealth process before healthcare
investment.

Thanks to \eqref{newb} and Theorem~\ref{dualityrelation2}, we can express the
optimal healthcare investment threshold \(\widehat b\) in terms of the dual
boundary \(b\).

\begin{pro}\label{hatb}
For every \(h\in\mathbb R_+\) such that \(0<b(h)<\infty\), one has
\[
    \widehat b(h)
    =
    -W_z(b(h),h)+\frac Ir=-W_z(\widetilde{ b}(0),h)+\frac Ir .
\]
\end{pro}

\begin{proof}
 Since
\(z^*(x,h)\) is characterized by $    x=-J_z(z^*(x,h),h),$
we have, for every \(z>0\),
$    x^*(z,h)=-J_z(z,h).$
Therefore, 
$    \widehat b(h)=x^*(b(h),h)=-J_z(b(h),h).$
From (\ref{3-14-1}) and (\ref{3-14}), we obtain
\[
    J_z(z,h)=\widehat J_z(z,h)+W_z(z,h)-\frac Ir.
\]
Moreover, since \(\widehat J(z,h)=\widetilde J(0,z)\) and
\(b(h)=\widetilde b(0)\), the smooth-fit condition in Corollary~\ref{cor:smooth-fit}
implies $\widehat J_z(b(h),h)=0.$
Hence
\[
    \widehat b(h)
    =
    -J_z(b(h),h)
    =
    -W_z(b(h),h)+\frac Ir=-W_z(\widetilde{ b}(0),h)+\frac Ir .
\]
This completes the proof.
\end{proof}

\section{Numerical Study}\label{sec:numerics}

In this section, we present numerical illustrations of the optimal healthcare investment boundary and the associated optimal consumption and portfolio strategies, and discuss their financial implications. We first display the healthcare investment boundary in the primal variables, as described in Proposition \ref{hatb}. We then investigate the sensitivity of the optimal boundary with respect to relevant model parameters and analyze the resulting economic implications. Finally, we illustrate the optimal consumption and portfolio rules, {both in the benchmark case without healthcare investment and in the case where healthcare investment is available.} The nonlinear integral equation in (\ref{eq:boundary-integral-equation}) is solved by the recursive iteration method proposed by \cite{huang1996pricing}; a detailed explanation of the numerical scheme is provided in Appendix \ref{numerical}. All numerical computations were performed using Mathematica 13.1. {The code is available upon request.}

\begin{table}[htbp]
\caption{The parameter set in the numerical illustrations}
\label{tab1}
\centering
	\begin{tabular}{cccccc}
		\toprule  
		 Symbol	& Interpretation &Value  \\
		  \midrule
	Panel A. Financial market and preference 	& &\\
	\midrule
	$r$ & Risk-free {interest rate} &0.048  \\
	$\mu$	& Expected risky return &0.108 \\
	 $\sigma $& Std. {deviation of} risky returns &0.20  \\
	 
	$\rho$ & Subjective discount rate &0.05  \\
	
	  	  \midrule
	Panel B. Survival and health dynamics 	& &\\
	\midrule
	$m^0$ &Minimal mortality intensity &0.0237  \\
		$m^1$	& Endogenous mortality intensity &0.0017 \\
	 $\kappa$& Mortality intensity convexity &1.80  \\
	 $\delta$ & Deterministic health depreciation rate &0.0055  \\
		$1-\alpha$	&Cobb-Douglas {health weight} &0.7742 \\
		$\beta$	& Curvature of health production function &0.19 \\	
		$I$	& {Healthcare investment rate}&2 \\			
	\bottomrule
	\end{tabular}
\end{table}

Most of the basic parameter values in Table 1 are taken from \cite{hugonnier2013health}. The values of the financial parameters in Panel A are conventional, and the calibrated value of the subjective discount rate $\rho$ is standard in Panel Study of Income Dynamics (PSID) studies. As discussed in Remark \ref{remark}, we choose the health production function $K=I^\beta$ with $\beta=0.19$ in this section, where the value of $\beta$ is taken from \cite{dalgaard2014optimal}.

We consider two types of agents. The ``healthy agent'' has an excellent initial health status, namely $h\in[100,200]$, whereas the ``sick agent'' has a poor initial health status, namely $h\in[2,10]$. {These two types of agents exhibit different health and mortality dynamics, as illustrated in Figures \ref{simulation} and \ref{simulation1}. The qualitative behavior of the optimal investment boundary is similar in several baseline experiments, but some comparative statics, in particular with respect to the subjective discount rate, differ across sick and healthy agents. We therefore discuss the economic mechanisms separately when needed.}

\begin{figure}[htbp]
    \centering
    \begin{subfigure}{0.45\textwidth}
        \includegraphics[width=\linewidth]{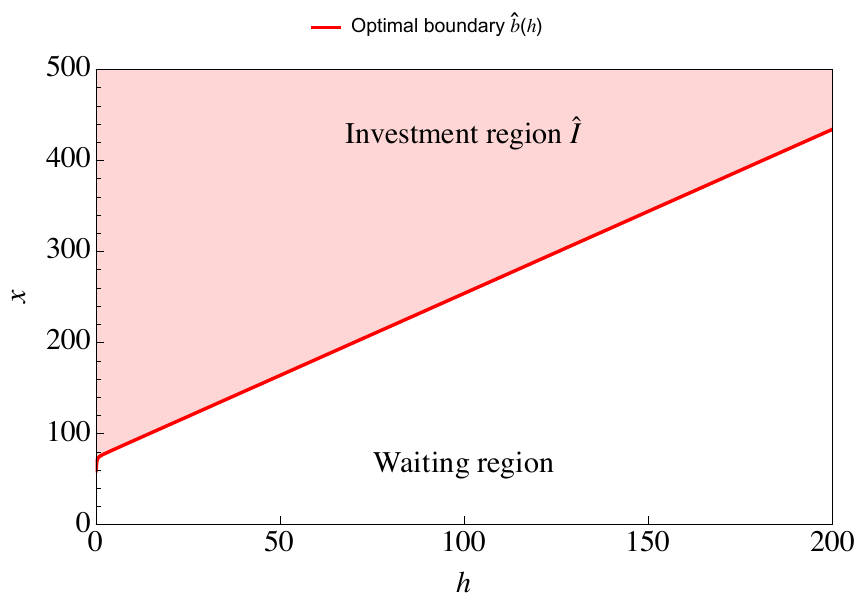}
        \label{fig4-boundary}
    \end{subfigure}
    \caption{The healthcare investment boundary $\widehat{b}(h)$ (red line) in primal variables $(x,h)$. The investing region $\widehat{\mathcal I}$ (shaded in red) indicates the set of points where the agent initiates healthcare investment---that is, when wealth $x$ (in dollars, vertical axis) enters this region. The parameter values are specified in Table \ref{tab1}. }
    \label{fig4}
\end{figure}

\subsection{Optimal boundaries and sensitivity analysis}

{From Figure \ref{fig4}, we observe that the investment boundary in the primal variables, $\widehat b(h)$, increases with the health status $h$ under the baseline parameter configuration. Equivalently, a higher level of wealth is required for healthier agents to initiate healthcare investment. This pattern can be understood from the marginal effectiveness of preventive healthcare. For agents with poor health, healthcare investment has a relatively strong effect on the future health path and on the reduction of mortality risk. As a result, the agent is willing to invest even at a relatively low wealth level. By contrast, when the agent is already in good health, the marginal benefit of further health improvement is smaller, and the agent optimally postpones healthcare investment unless her wealth is sufficiently high.}

 \begin{figure}[htbp]
    \centering
    \begin{subfigure}{0.45\textwidth}
        \includegraphics[width=\linewidth]{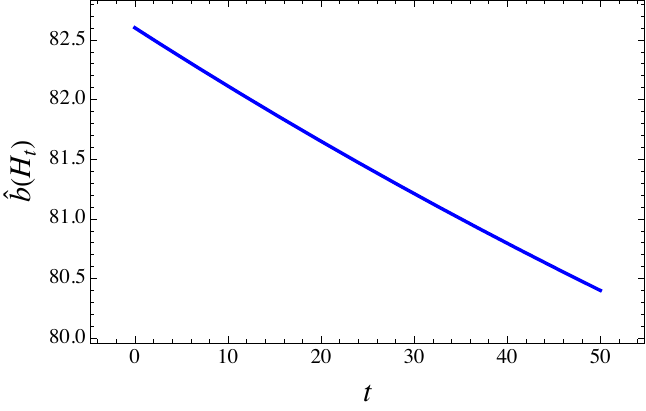}
        \caption{Sick agent with $h=5$.}
        \label{fig5-sick}
    \end{subfigure}
        \begin{subfigure}{0.45\textwidth}
        \includegraphics[width=\linewidth]{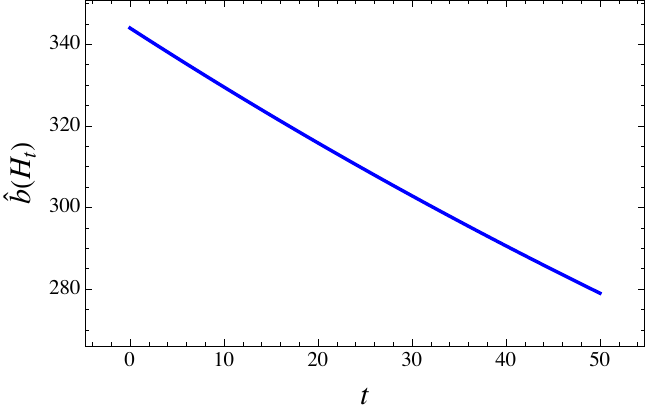}
        \caption{Healthy agent with $h=150$.}
        \label{fig5-healthy}
    \end{subfigure}
    \caption{The healthcare investment threshold $\widehat{b}(H_t)$ (in dollars, vertical axis) is shown as a function of time $t$ (in years, horizontal axis). Investment in healthcare becomes optimal when the agent's wealth $X_t$ exceeds the threshold $\widehat{b}(H_t)$ (blue line). The left panel depicts the case of an initially sick agent ($h = 5$), while the right panel corresponds to an initially healthy agent ($h = 150$).
}
    \label{fig5}
\end{figure}

{In Figure \ref{fig5}, we examine the evolution of the optimal healthcare investment boundary $\widehat b(H_t)$ over time. Under the baseline parameter configuration, both the initially sick agent, with $h=5$, and the initially healthy agent, with $h=150$, exhibit a declining investment boundary as they age. Since health capital decreases over time before healthcare investment, this pattern is consistent with the increasing relation between $\widehat b(h)$ and $h$ observed in Figure \ref{fig4}. Economically, as the agent's health gradually deteriorates, the marginal benefit of precautionary healthcare investment increases, so that healthcare investment becomes optimal at a lower wealth level. This qualitative behavior is robust for both the sick and the healthy agent under our baseline calibration, although the level of the boundary differs substantially across the two initial health statuses.}

 \begin{figure}[htbp]
    \centering
    \begin{subfigure}{0.45\textwidth}
        \includegraphics[width=\linewidth]{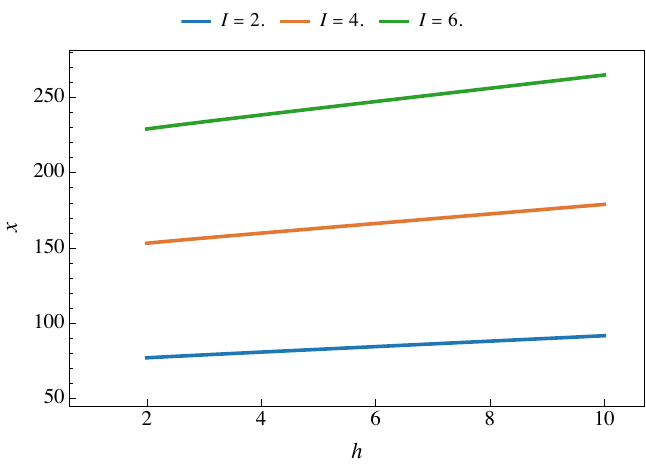}
        \caption{Sick agent with $h \in [2,10]$.}
        \label{fig6-sick}
    \end{subfigure}
        \begin{subfigure}{0.45\textwidth}
        \includegraphics[width=\linewidth]{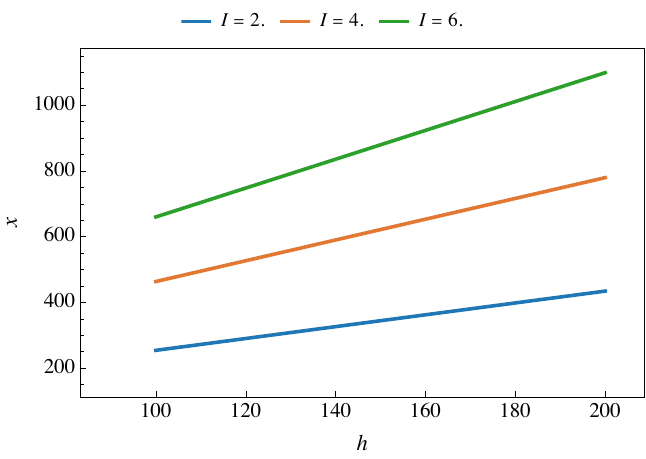}
        \caption{Healthy agent with $h \in [100,200]$.}
        \label{fig6-healthy}
    \end{subfigure}
    \caption{ The healthcare investment boundary $\widehat{b}(h)$ (colored lines) varies with the investment intensity $I$ (in dollars per unit of time). The left panel illustrates the case of an initially sick agent with $h \in [2,10]$, while the right panel corresponds to a healthy agent with $h\in [100,200]$.  }
    \label{fig6}
\end{figure}

{In Figure \ref{fig6}, we study the sensitivity of the healthcare investment boundary $\widehat b$ with respect to the investment intensity $I$. We find that, as the investment intensity increases, the boundary shifts upward and the continuation region expands for both healthy and sick agents. This pattern reflects the trade-off induced by a higher investment intensity. On the one hand, a larger $I$ increases the cost of healthcare investment; on the other hand, due to decreasing returns in the health production function, the associated marginal improvement in health becomes relatively smaller. As a result, the attractiveness of immediate investment decreases. Consequently, both healthy and sick individuals become less inclined to undertake precautionary healthcare investment when the investment intensity is high.}

\begin{figure}[htbp]
    \centering
    \begin{subfigure}{0.45\textwidth}
        \includegraphics[width=\linewidth]{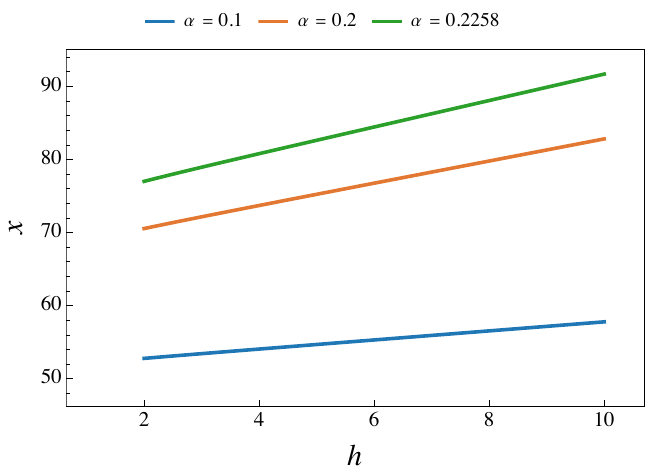}
        \caption{Sick agent with $h\in [2,10]$.}
        \label{fig8-sick}
    \end{subfigure}
        \begin{subfigure}{0.45\textwidth}
        \includegraphics[width=\linewidth]{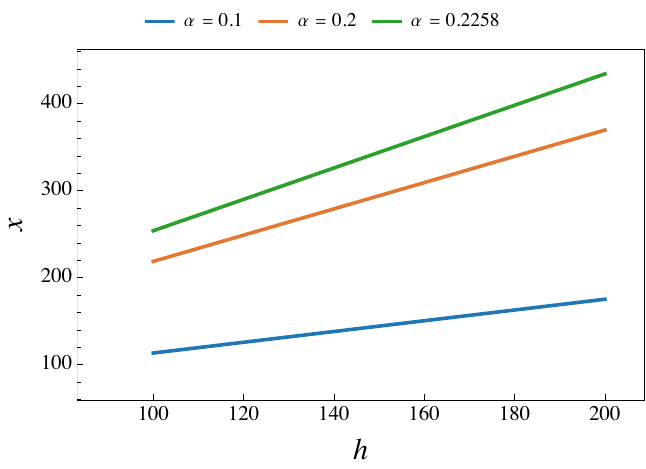}
        \caption{Healthy agent with $h\in [100,200]$.}
        \label{fig8-healthy}
    \end{subfigure}
    \caption{ The healthcare investment boundary $\widehat{b}(h)$ (colored lines) varies with the {Cobb--Douglas} parameter $\alpha$. The left panel illustrates the case of an initially sick agent with $h \in [2,10]$, while the right panel corresponds to a healthy agent with $h\in [100,200]$.  }
    \label{fig8}
\end{figure}

{Figure \ref{fig8} shows the effect of a change in the Cobb--Douglas parameter $\alpha$ on the boundary $\widehat b$. We find that a higher $\alpha$ shifts the boundary upward. Since a larger $\alpha$ means that the agent assigns relatively more weight to consumption and less weight to health capital in utility, the incentive to invest in healthcare becomes weaker. As a result, healthcare investment is delayed and a higher wealth level is required to trigger investment.}

\begin{figure}[htbp]
    \centering
    \begin{subfigure}{0.45\textwidth}
        \includegraphics[width=\linewidth]{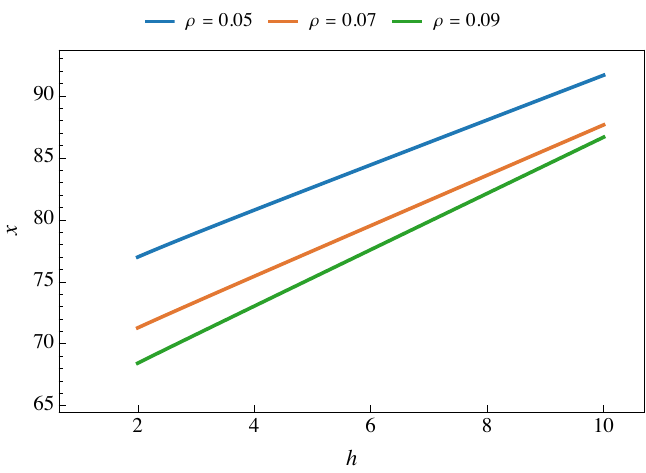}
        \caption{Sick agent with $h \in [2,10] $.}
        \label{fig9-rho-sick}
    \end{subfigure}
        \begin{subfigure}{0.45\textwidth}
        \includegraphics[width=\linewidth]{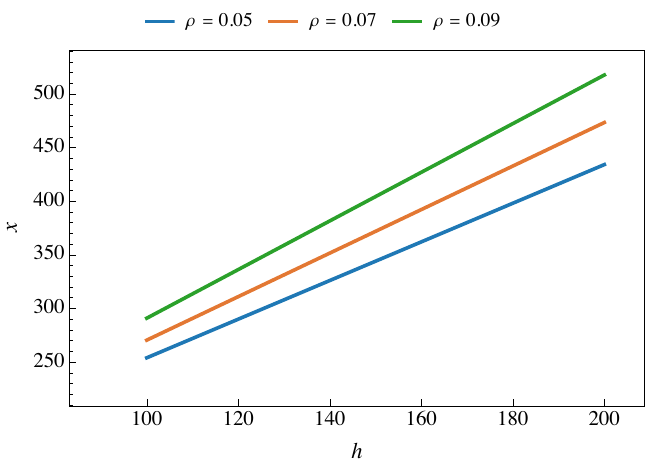}
        \caption{Healthy agent with $h \in [100,200] $.}
        \label{fig9-rho-healthy}
    \end{subfigure}
    \caption{ The healthcare investment boundary $\widehat{b}(h)$ (colored lines) varies with the subjective discount rate $\rho$. The left panel illustrates the case of an initially sick agent with $h \in [2,10]$, while the right panel corresponds to a healthy agent with $h\in [100,200]$. }
    \label{fig9}
\end{figure}

Figure \ref{fig9} shows the effect of varying the subjective discount rate $\rho$. An increase in $\rho$ reflects greater impatience, leading the agent to discount future utility more heavily and to place more value on near-term consumption. The figure reveals that the effect of $\rho$ differs between sick and healthy agents. For healthy agents, a higher $\rho$ shifts the investment boundary upward. Since these agents already hold a large stock of health capital, the current utility gain from additional health investment is relatively limited, while the cost of investment is immediate. A higher discount rate therefore makes immediate healthcare investment less attractive. By contrast, for sick agents, a higher $\rho$ shifts the boundary downward. In this case, additional health capital generates a stronger immediate utility gain and a stronger reduction in mortality risk. Therefore, even a more impatient agent may find it optimal to invest at a lower wealth level when the initial health stock is poor.

\subsection{Optimal portfolio and consumption plan}

{We now illustrate the optimal consumption and portfolio strategies. We first consider the benchmark case without healthcare investment. In the case $I=K=0$, equation (4.4) implies that $\widehat J(z,h)=0$, and the healthcare investment time is immaterial. The problem then reduces to a Merton-type consumption--portfolio problem with a random time horizon and health-dependent mortality. The resulting consumption--wealth and portfolio--wealth ratios are reported in Figure \ref{fig91}. This calibration serves as a reference point against which the effects of a positive healthcare investment intensity can be assessed.}

\begin{figure}[htbp]
    \centering
    \begin{subfigure}[t]{0.48\textwidth}
        \centering
        \includegraphics[width=\linewidth]{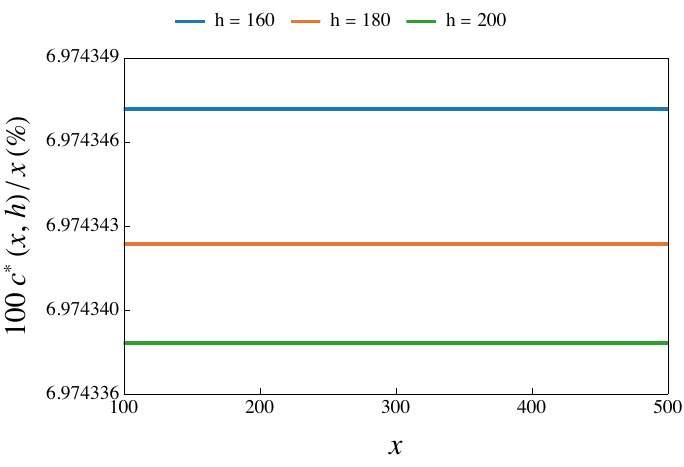}
        \caption{Optimal consumption--wealth ratio ($I=0$).}
        \label{fig91a}
    \end{subfigure}
    \hfill
    \begin{subfigure}[t]{0.48\textwidth}
        \centering
        \includegraphics[width=\linewidth]{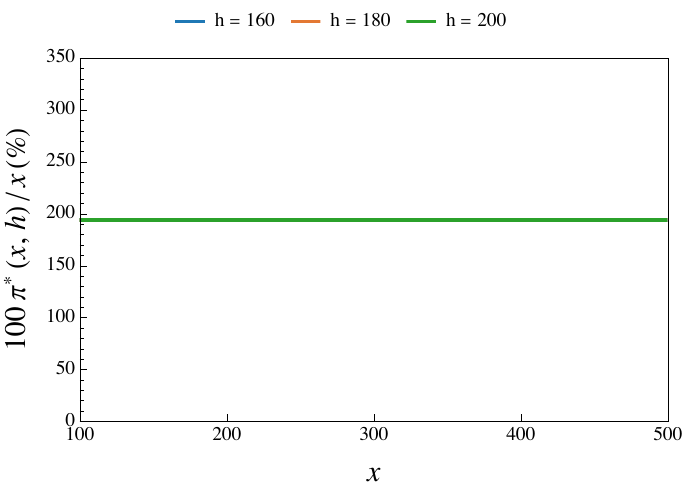}
        \caption{Optimal portfolio--wealth ratio ($I=0$).}
        \label{fig91b}
    \end{subfigure}
    \caption{Left: Optimal consumption--wealth ratio ($\%$, vertical axis) as a
    function of initial wealth $x$ (in dollars, horizontal axis), shown for
    different initial health levels $h$ (colored lines). Right: Optimal
    portfolio--wealth ratio ($\%$, vertical axis) as a function of initial wealth
    $x$, shown for the same health levels. {The Merton portfolio benchmark is
    $193.748\%$.}}
    \label{fig91}
\end{figure}

{Figure \ref{fig91a} shows that the optimal consumption--wealth ratio is independent of wealth and varies only very mildly with health. In the displayed calibration, the ratio decreases slightly with the agent's health status: healthier agents consume a marginally smaller fraction of their wealth. In the exact benchmark $I=K=0$, the direct multiplicative effect of health in the Cobb--Douglas utility function $u(c,h)=c^\alpha h^{1-\alpha}$ largely cancels from the consumption--wealth ratio. The remaining health dependence is induced by mortality: a higher health stock reduces mortality risk and lengthens the agent's effective planning horizon, which leads to a slightly lower consumption--wealth ratio. At the same time, the homotheticity of preferences implies that the ratio is flat in wealth.}

{Figure \ref{fig91b} shows that the optimal portfolio--wealth ratio is constant. In the exact benchmark $I=K=0$, the Merton first-order condition gives
\[
    \frac{\pi^*(x,h)}{x}
    =
    \frac{\mu-r}{\sigma^2(1-\alpha)} .
\]
Hence the portfolio share is independent of both wealth and health. The horizontal lines in Figure \ref{fig91b} are therefore consistent with the benchmark Merton mechanism.}

\begin{figure}[htbp]
    \centering
    \begin{subfigure}[t]{0.48\textwidth}
        \centering
        \includegraphics[width=\linewidth]{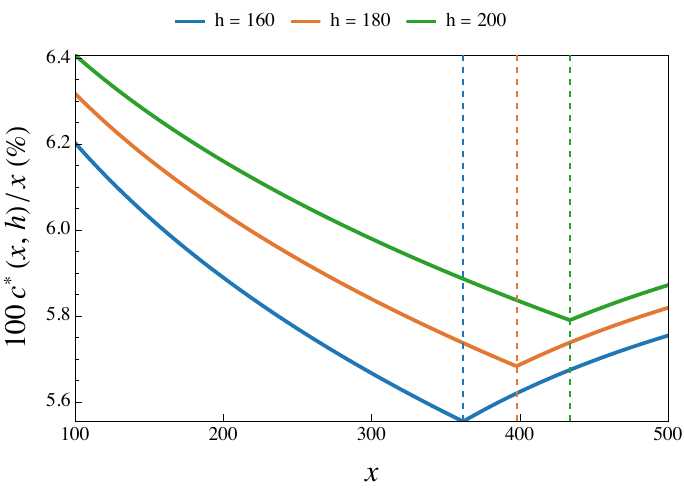}
        \caption{Optimal consumption--wealth ratio ($I=2$).}
        \label{fig10a}
    \end{subfigure}
    \hfill
    \begin{subfigure}[t]{0.48\textwidth}
        \centering
        \includegraphics[width=\linewidth]{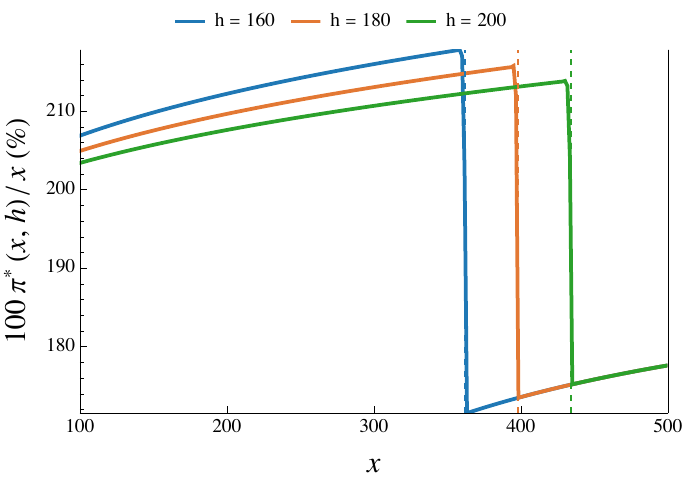}
        \caption{Optimal portfolio--wealth ratio ($I=2$).}
        \label{fig10b}
    \end{subfigure}
    \caption{Left: Optimal consumption--wealth ratio ($\%$, vertical axis) as a
    function of initial wealth $x$ (in dollars, horizontal axis), shown for
    different initial health levels $h$ (colored lines). Right: Optimal
    portfolio--wealth ratio ($\%$, vertical axis) as a function of initial wealth
    $x$, shown for the same health levels. {In both panels, the vertical
    dashed lines indicate the corresponding endogenous healthcare investment
    thresholds $\widehat b(h)$.}}
    \label{fig10}
\end{figure}

{We next analyze the optimal consumption--wealth and portfolio--wealth ratios when healthcare investment is available. In particular, we set $I=2$ and examine how the optimal policies depend jointly on wealth and health. The vertical dashed lines in Figure \ref{fig10} indicate the corresponding healthcare investment boundaries. Hence, for each health level, the region to the left of the dashed line is the continuation region, where the agent postpones healthcare investment, whereas the region to the right is the investment region, where healthcare investment has already been undertaken.}

{Figure \ref{fig10a} shows that the optimal consumption--wealth ratio decreases with wealth in the continuation region and then increases again after healthcare investment has been triggered. Before investment, as wealth approaches the endogenous boundary, the agent anticipates the imminent healthcare expenditure and preserves more wealth in order to finance the future investment cost. This precautionary channel pushes down the consumption--wealth ratio near the boundary. Once the boundary is reached and healthcare investment has been undertaken, the regime changes: the agent pays the healthcare investment cost, the health dynamics switch to the post-investment dynamics, and the consumption policy moves to the post-investment regime. In that region, the consumption--wealth ratio increases gradually with wealth.}

{Figure \ref{fig10a} also shows that health affects the consumption policy more visibly than in the benchmark case. In particular, for a given wealth level, healthier agents consume a larger fraction of wealth, consistent with \cite{bolin2020consumption} and \cite{hugonnier2013health}. This is again in line with the Cobb--Douglas utility channel: a higher health stock raises the marginal utility value of consumption. Compared with Figure \ref{fig91a}, where consumption ratios vary only slightly across health levels, Figure \ref{fig10a} shows that the availability of a substantial healthcare investment option amplifies the role of health in shaping the optimal consumption plan.}

{Figure \ref{fig10b} shows that the optimal portfolio--wealth ratio is no longer constant, in contrast to the benchmark case. In the continuation region, the portfolio ratio increases with wealth. When wealth is low, the agent behaves more cautiously because healthcare investment remains a valuable future option and because wealth must be preserved to finance both future consumption and possible healthcare costs. As wealth increases and the agent approaches the investment boundary, risk-taking becomes more attractive. Once healthcare investment is triggered, the portfolio ratio displays a visible downward jump and then increases again gradually with wealth. This discontinuity reflects the change in the effective financial position after investment: part of wealth is committed to healthcare expenditure, and the post-investment problem is governed by different health and mortality dynamics.}

\subsection{A numerical illustration of a non-monotone boundary $\widetilde{b}(t)$.}
\label{subsec:nonmonotone-boundary}

{The baseline calibration produces a monotone investment boundary. This
monotonicity, however, should not be interpreted as a general implication of
the model. Indeed, the theoretical analysis in Section~\ref{sec4} does not establish the
monotonicity of \(t\mapsto \widetilde b(t)\), or equivalently of
\(h\mapsto b(h)\). To show that this lack of a general monotonicity result is
not merely technical, we provide below a numerical example, under a different
admissible parameter configuration, in which the transformed dual boundary
\(\widetilde b(t)\) is non-monotone.}


\begin{table}[htbp]
\centering
\caption{Alternative parameter configuration generating a non-monotone $\widetilde{b}$.}
\label{tab:nonmonotone-parameters}
\begin{tabular}{cccccccccccc}
\hline
\(r\) & \(\mu\) & \(\sigma\) & \(\rho\) & \(m^0\) & \(m^1\) & \(\kappa\) & \(\delta\) & \(\alpha\) & \(\beta\) & \(I\) & \(K\) \\
\hline
0.070 & 0.134 & 0.20 & 0.360 & 0.025 & 0.400 & 3.50 & 0.016 & 0.70 & 0.25 & 2 & \(I^\beta\) \\
\hline
\end{tabular}
\end{table}

{We consider the alternative parameter configuration reported in
Table~\ref{tab:nonmonotone-parameters}, which satisfies
Assumption~\ref{assume1}. Taking \(h=2\), the resulting transformed dual
boundary \(t\mapsto \widetilde b(t)\) is non-monotone, as shown in
Figure~\ref{fig11a}. It is worth noting that, although \(\widetilde b(t)\) is
not monotone, the corresponding primal boundary \(\widehat b(h)\) remains
increasing in the plotted range; see Figure~\ref{fig11b}. This is consistent
with the monotonicity of the primal boundary observed under the baseline
calibration in Figure~\ref{fig4}.}

\begin{figure}[htbp]
    \centering
    \begin{subfigure}{0.45\textwidth}
        \includegraphics[width=\linewidth]{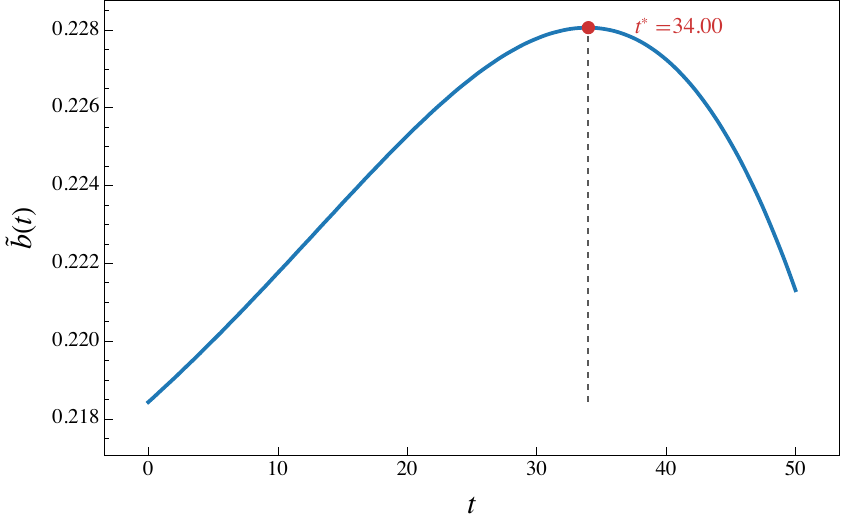}
        \caption{Non-monotone dual boundary $\widetilde{b}(t)$.}
        \label{fig11a}
    \end{subfigure}
    \begin{subfigure}{0.45\textwidth}
        \includegraphics[width=\linewidth]{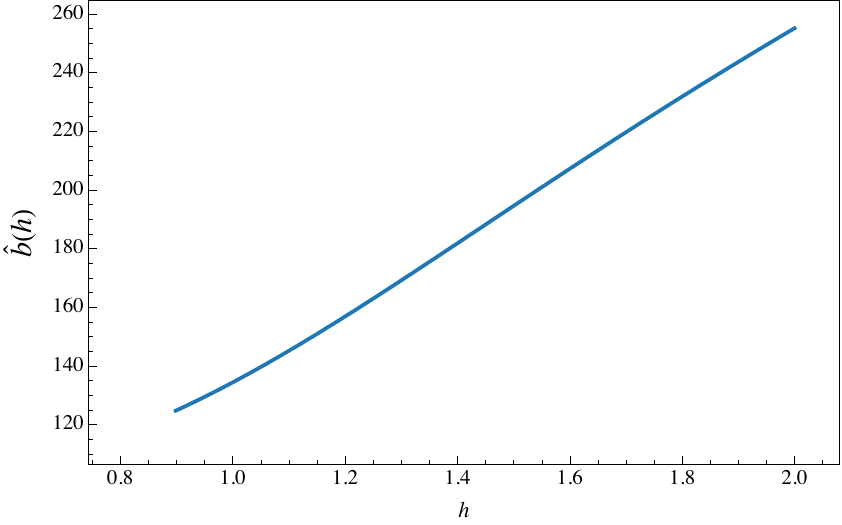}
        \caption{Primal investment boundary $\widehat{b}(h)$.}
        \label{fig11b}
    \end{subfigure}
   \caption{{Left: Transformed dual boundary \(\widetilde b(t)\) as a function of
time \(t\). The boundary first increases and then decreases, attaining its
maximum at approximately \( t^*=34.00\), as indicated by the vertical
dashed line. Right: Corresponding investment boundary in the primal variables,
\(\widehat b(h)\), as a function of the health level \(h\). Although the
transformed dual boundary is non-monotone, the primal boundary remains
increasing over the plotted health range.}}
    \label{fig11}
\end{figure}

{It should be stressed that this parameter configuration is not intended to
replace the baseline calibration. The parameters reported in Table~\ref{tab1}
are more realistic and are chosen as the main calibration for our numerical
analysis. The purpose of the present example is only to illustrate that
non-monotonicity of the transformed dual boundary $\widetilde{b}(t)$ may arise under admissible
parameter configurations.}

\section{Discussion and Conclusions}\label{sec7}

In this paper, we study a consumption/portfolio problem in which the agent can
also choose the time at which to make an irreversible precautionary investment
in health, thus facing a trade-off between a costly health investment and a
reduction in her force of mortality. The optimization problem is formulated as
a stochastic control-stopping problem over a random time horizon, with two state
variables: wealth and health capital.

We first transform the original problem, by martingale and duality methods,
into its dual problem, which is an {infinite-horizon} two-dimensional optimal
stopping problem. We then study this optimal stopping problem by probabilistic
arguments. Due to the lack of monotonicity of the optimal stopping boundary, we
establish the regularity properties needed to represent the stopping region by
an upper-semicontinuous boundary and to characterize it through a nonlinear
integral equation. Finally, we obtain the optimal strategies in terms of the
primal variables and show that the agent optimally invests in health whenever
her wealth reaches a boundary curve that depends on her health capital.

We now compare our framework and findings with those of
\cite{hugonnier2013health}, \cite{guasoni2019consumption}, and
\cite{aurand2021mortality}. The influential work of
\cite{hugonnier2013health} develops a comprehensive model of optimal
consumption, portfolio allocation, health investment, and insurance, supported
by detailed empirical calibration. Our model is inspired by this line of
research but shifts the focus to a distinct dimension: we emphasize the
\textit{timing} of healthcare investment decisions, rather than the continuous
adjustment of spending intensity.

The recent contributions by \cite{guasoni2019consumption} and
\cite{aurand2021mortality} further advance the mathematical foundations of
health-related decision-making under uncertainty. Their rigorous stochastic
control models provide important benchmarks, especially in settings without an
explicit health capital state variable. In contrast, our model incorporates
health capital as a dynamic state and adopts a control-stopping framework to
study threshold-based investment behavior. This setup captures the timing of {threshold-triggered health-spending decisions}, which are empirically relevant but difficult to
represent within standard continuous-control formulations.

While our model assumes a fixed investment {rate} \(I\) once the healthcare
decision is triggered, this simplification offers a tractable starting point for
isolating timing effects. A natural extension is to allow for flexible
investment sizes and timing. In particular, building on our approach and the
direction of \cite{guasoni2019consumption}, future work may formulate
healthcare investment as a singular control problem. Similar to consumption
choices (see  \cite{bank2001optimal}, \cite{hindy1993optimal}), this would allow
for both discrete ``gulps'' and continuous flows of health spending. Modeling
the cumulative healthcare investment process \(I_t\) as a singular control
would provide a more general and realistic framework for analyzing optimal
health investment strategies under uncertainty. We leave this {promising} and
challenging research question for future research.

\appendix
\setcounter{subsection}{0}
\renewcommand\thesubsection{A.\arabic{subsection}}

\setcounter{equation}{0}
\renewcommand\theequation{A.\arabic{equation}}
\setcounter{lemma}{0}
    \renewcommand{\thelemma}{\Alph{section}.\arabic{lemma}}

\section{Proofs from Sections 3, 4 and 5}\label{seca}
\renewcommand\theequation{A.\arabic{equation}}
\renewcommand\thesubsection{A.\arabic{subsection}}

\subsection{Proof of Proposition \ref{Wregularity}}
\begin{proof}\label{proWregularity}
First, we compute the convex dual of $u(c,h)=c^\alpha h^{1-\alpha}$ in (\ref{3-4-1}); that is,
\begin{align}\label{3-9-1}
\widehat{u}(z,h)=(1-\alpha)(\frac{z}{\alpha})^{\frac{\alpha}{\alpha-1}}h.
\end{align} 

From (\ref{2-3}) and the boundary condition $H^2_\tau=H^2_0=h$ (recall that $\tau=0$ {in} this subsection), we have 
\begin{align}\label{3-9-2}
 H_s^2 = he^{-\delta s}+ \frac{K}{\delta}(1-e^{-\delta s}), \ \forall s \geq 0.
\end{align} 

Therefore, by (\ref{3-9-1}) and (\ref{3-9-2}) we rewrite $W(z,h)$ as follows
\begin{align}\label{3-9-3}
W(z,h)=& \  \mathbb{E}_{z,h}\bigg[    \int_0^{\infty }  e^{-\int^s_0 (\rho+ M^2_u) du} \widehat{u}(Z^2_s, H^2_s) ds  \bigg] \nonumber \\
=&\ \mathbb{E}_{z,h}\bigg[    \int_0^{\infty }  e^{-\int^s_0 (\rho+ M^2_u) du} (1-\alpha)\alpha^{\frac{\alpha}{1-\alpha}}(Z^2_s)^{\frac{\alpha}{\alpha-1}} H^2_sds  \bigg]\nonumber\\
=&\ (1-\alpha)\alpha^{\frac{\alpha}{1-\alpha}} z^{\frac{\alpha}{\alpha-1}}    \int_0^{\infty }  e^{-\int^s_0 (\rho+ M^2_u) du}\mathbb{E}[(P^2_s(h))^{\frac{\alpha}{\alpha-1}}]\bigg(he^{-\delta s}+ \frac{K}{\delta}(1-e^{-\delta s})\bigg) ds \nonumber\\
=&\ (1-\alpha)\alpha^{\frac{\alpha}{1-\alpha}} z^{\frac{\alpha}{\alpha-1}}    \int_0^{\infty }  e^{\frac{1}{\alpha-1}\int^s_0 (\rho+ M^2_u) du }e^{\int^s_0 \big(\frac{\alpha(-r-\frac{1}{2}\theta^2)}{\alpha-1}+\frac{1}{2}\frac{\theta^2\alpha^2}{(\alpha-1)^2}\big)du}\nonumber  \\& \times\bigg(he^{-\delta s}+ \frac{K}{\delta}(1-e^{-\delta s})\bigg) ds,
\end{align}
where we have used the definition of $P^2_s(h)$ as in (\ref{3-4-1}) and the fact that 
\begin{align*}
\mathbb{E}[(P^2_s(h))^{\frac{\alpha}{\alpha-1}}]&=\mathbb{E}[( \gamma_{s}e^{\int^s_0 (\rho+ M^2_u) du })^{\frac{\alpha}{\alpha-1}}]
= e^{ \frac{\alpha}{\alpha-1}  \int^s_0 (\rho+ M^2_u) du }    \mathbb{E}[ \gamma_{s}^{\frac{\alpha}{\alpha-1}}]\\
&=e^{ \frac{\alpha}{\alpha-1}  \int^s_0 (\rho+ M^2_u) du }    \mathbb{E}[ (e^{-rs-\theta B_s  -\frac{1}{2}\theta^2 s})^{\frac{\alpha}{\alpha-1}}]\\
&= e^{\int^s_0 \big(\frac{\alpha}{\alpha-1}(\rho+M^2_u-r-\frac{1}{2}\theta^2)+\frac{1}{2}\frac{\theta^2\alpha^2}{(\alpha-1)^2}\big)du}.
\end{align*}
Due to Assumption \ref{assume1}, from (\ref{3-9-3}) it is easily verified that $W(z,h)<\infty$. Finally, it can also be checked from \eqref{3-9-3} that {\(W \in C^{2,1}(\mathcal{O})\) and \(W_z\in C^{2,1}(\mathcal O)\)}, it is strictly convex with respect to $z$, and it satisfies  (\ref{3-6}) by direct calculations. 
\end{proof}


\subsection{Proof of Theorem \ref{dualityrelation}}

Before we present the proof of Theorem \ref{dualityrelation}, we need the following lemmata.

{{
\begin{lemma}\label{budget2}
Let $x\geq \frac{I}{r}$ be given, and let $c\geq 0$ be a consumption process satisfying
\begin{align*}
\mathbb{E}_{x,h}\bigg[\int_0^{\infty}\gamma_s c_s\,ds\bigg]
=
x-\frac{I}{r}.
\end{align*}
Then there exists a portfolio process $\pi$ such that
$(c,\pi)\in\mathcal{A}_0(x,h)$ and
\begin{align*}
X^{c,\pi,0}_t\geq \frac{I}{r},
\qquad
\mathbb{P}_{x,h}\text{-a.s. for all }t\geq 0.
\end{align*}
\end{lemma}

\begin{proof}
Define $L_s:=\int_0^s\gamma_u c_u\,du, s\geq 0,$ and consider the nonnegative martingale
\begin{align*}
\widehat{\mathcal{M}}_s
:=
\mathbb{E}_{x,h}\big[L_\infty\mid\mathcal{F}_s\big],
\qquad s\geq 0.
\end{align*}
By the martingale representation theorem, there exists an
$\mathbb{F}$-progressively measurable process $\phi$ satisfying
\begin{align*}
\int_0^t|\phi_u|^2\,du<\infty,
\qquad
\mathbb{P}_{x,h}\text{-a.s. for every }t\geq 0,
\end{align*}
such that
\begin{align*}
\widehat{\mathcal{M}}_s
=
\widehat{\mathcal{M}}_0+\int_0^s\phi_u\,dB_u
=
x-\frac{I}{r}+\int_0^s\phi_u\,dB_u,
\qquad s\geq 0.
\end{align*}

Set $Y_s:=\widehat{\mathcal{M}}_s-L_s$
and define
\begin{align}\label{eqr1}
X_s
:=
\frac{1}{\gamma_s}
\mathbb{E}_{x,h}\bigg[
\int_s^\infty\gamma_u c_u\,du
\,\bigg|\,\mathcal{F}_s
\bigg]
+\frac{I}{r} =
\frac{Y_s}{\gamma_s}+\frac{I}{r},
\qquad s\geq 0.
\end{align}
In particular, $X_0=x$ and
$
X_s\geq \frac{I}{r}$, for $s\geq 0.$ Since $dY_s=\phi_s\,dB_s-\gamma_s c_s\,ds$
and
$\gamma_s
=
\exp\{
-rs-\theta B_s-\frac{1}{2}\theta^2s
\},$
It\^o's formula gives
\begin{align*}
d\big(\gamma_s^{-1}\big)
=
(r+\theta^2)\gamma_s^{-1}\,ds
+
\theta\gamma_s^{-1}\,dB_s.
\end{align*}
Therefore, applying It\^o's formula to $\eqref{eqr1}$ yields
\begin{align*}
dX_s
&=
\bigg[
-c_s
+
(r+\theta^2)\frac{Y_s}{\gamma_s}
+
\theta\frac{\phi_s}{\gamma_s}
\bigg]ds 
+
\frac{1}{\gamma_s}
\big(\phi_s+\theta Y_s\big)\,dB_s.
\end{align*}

Define
\begin{align*}
\pi_s
:=
\frac{1}{\sigma\gamma_s}
\big(\phi_s+\theta Y_s\big)
=
\frac{1}{\sigma\gamma_s}
\big[
\phi_s+\theta(\widehat{\mathcal{M}}_s-L_s)
\big].
\end{align*}
Using $\theta=(\mu-r)/\sigma$, we obtain
\begin{align*}
\theta^2\frac{Y_s}{\gamma_s}
+
\theta\frac{\phi_s}{\gamma_s}
=
\theta\sigma\pi_s
=
(\mu-r)\pi_s.
\end{align*}
Since $Y_s/\gamma_s=X_s-I/r$ (cf.\ (\ref{eqr1})), it follows that
\begin{align*}
dX_s
&=
\bigg[
r\bigg(X_s-\frac{I}{r}\bigg)
+
(\mu-r)\pi_s
-c_s
\bigg]ds
+
\sigma\pi_s\,dB_s \\
&=
\big[
rX_s+(\mu-r)\pi_s-c_s-I
\big]ds
+
\sigma\pi_s\,dB_s.
\end{align*}
Thus, $X$ satisfies the wealth dynamics \eqref{2-4} with $\tau=0$
and initial condition $X_0=x$.

Moreover, the local square integrability of $\phi$, together with
the continuity and strict positivity of $\gamma$ and the continuity
of $Y$, imply that
\begin{align*}
\int_0^t|\pi_s|^2\,ds<\infty,
\qquad
\mathbb{P}_{x,h}\text{-a.s.}
\end{align*}
Also, the assumed integrability of
$\int_0^\infty\gamma_s c_s\,ds$ implies that
\begin{align*}
\int_0^t c_s\,ds<\infty,
\qquad
\mathbb{P}_{x,h}\text{-a.s.}
\end{align*}
for every $t\geq 0$. Hence $(c,\pi)\in\mathcal{A}_0(x,h)$.
Finally, $X^{c,\pi,0}_s
=
X_s
\geq
\frac{I}{r}$,
$\mathbb{P}_{x,h}$-a.s.\ for all $s\geq 0$.
\end{proof}
}}

\begin{lemma}\label{lemmaa.2}
	Let $h  \in  \mathbb{R}_+ $ be given. Define $\mathcal{X}(z,h):=\mathbb{E}_{z,h}[  \int_0^{\infty} \gamma_{s} \mathcal{I}^u(z P^2_s, H^2_s) ds ]+\frac{I}{r}$. Then, for $h>0$, the function $\mathcal{X}(\cdot,h)$ is strictly decreasing and continuous on $(0,\infty)$, with 
	\begin{align*}
		\lim_{z\downarrow 0 } \mathcal{X}(z,h)=\infty, \quad \lim_{z\to +\infty}\mathcal{X}(z,h)=\frac{I}{r}.
	\end{align*}
	In particular, the function $\mathcal{X}(\cdot,h)$ on $(0,\infty)$ has a strictly decreasing inverse function $\mathcal{Z}(\cdot,h):(\frac{I}{r}
	,\infty) \to (0,\infty) $, so that 
	\begin{align*}
		\mathcal{X}(\mathcal{Z}(x,h),h)=x, \quad \forall x \in (\frac{I}{r}, \infty).
	\end{align*}
\end{lemma}
\begin{proof}
	Since $\mathcal{I}^u(z,h)= (\frac{z}{\alpha})^{\frac{1}{\alpha-1}}h$ is strictly decreasing with respect to $z$, so is $\mathcal{X}(\cdot,h)$.
{
To prove continuity, fix \(z_0>0\) and let
\(z\in[z_0/2,2z_0]\). Since \(1/(\alpha-1)<0\), the map
\(z\mapsto \mathcal I^u(z,h)\) is decreasing. Hence, for every \(s\ge0\), $    \mathcal I^u(zP_s^2,H_s^2)
    \leq
    \mathcal I^u((z_0/2)P_s^2,H_s^2).$
Moreover, the discounted process $    \gamma_s\mathcal I^u((z_0/2)P_s^2,H_s^2)$
is integrable on \(\Omega\times[0,\infty)\) by
Assumption~\ref{assume1}. Therefore, dominated convergence gives the
continuity of \(\mathcal X(\cdot,h)\) at \(z_0\). Finally, since \(z\mapsto \mathcal I^u(zP_s^2,H_s^2)\) is decreasing, the limits as $z\downarrow 0$ and $z \uparrow \infty$ follow by
monotone convergence theorem.
}
	\end{proof}

\begin{proof}[Proof of Theorem \ref{dualityrelation}]\label{produalityralation}
{First, we show the duality relations. Since $(c,\pi) \in \mathcal{A}_0(x,h)$ is arbitrary, taking the supremum over $(c,\pi)  \in \mathcal{A}_0(x,h)$ on the left-hand side of (\ref{3-4}) and recalling (\ref{3-4-0}), we get, for any $z>0$, $x\geq \frac{I}{r}$,
\begin{align*}
\widehat{V}(x,h) &\leq W(z,h)+ z(x-\frac{I}{r}),  
\end{align*}
and thus 
\begin{align}\label{dua}
W(z,h) \geq \sup_{x\geq  \frac{I}{r}}[\widehat{V}(x,h)-z(x-\frac{I}{r})] \quad \text{and} \quad \widehat{V}(x,h) \leq \inf_{z>0}[W(z,h)+z(x-\frac{I}{r})]. 
\end{align}
}

{For the reverse inequalities, observe that equality in (\ref{3-4}) holds if and only if 
\begin{align}\label{3-7-1}
c_s= \mathcal{I}^u(Z^2_s,H^2_s),
\end{align}
and 
\begin{align*}
 \mathbb{E}_{x,h}\bigg[  \int_0^{\infty} \gamma_{s}c_s ds                         \bigg]  = x -\frac{I}{r}.
\end{align*}
In particular, the equality holds in (\ref{3-4}) if (\ref{3-7-1}) is satisfied and $x=\mathcal{X}(z,h)$. This gives $W(z,h)=\widehat{V}(\mathcal{X}(z,h),h)-z(\mathcal{X}(z,h)-\frac{I}{r})$. Combining with (\ref{dua}), this completes the proof of $W(z,h) = \sup_{x\geq  \frac{I}{r}}[\widehat{V}(x,h)-z(x-\frac{I}{r})]$ and shows that, for $z>0$, the maximum in $W(z,h) = \sup_{x\geq  \frac{I}{r}}[\widehat{V}(x,h)-z(x-\frac{I}{r})]$ is attained by $x=\mathcal{X}(z,h)$.
}

For any $x\in (\frac{I}{r},+\infty)$, we introduce the candidate optimal consumption process 
\begin{align*}
	c_s^*= \mathcal{I}^u(\mathcal{Z}(x,h) P^2_s,H^2_s),
	\end{align*}
with $\mathcal{Z}(x,h)$ as defined in Lemma \ref{lemmaa.2}. Then 
\begin{align}\label{wealth}
	\mathbb{E}_{x,h}\bigg[  \int_0^{\infty} \gamma_{s}c_s^* ds                         \bigg]+\frac{I}{r}  = \mathcal{X}(\mathcal{Z}(x,h),h)=x  
	\end{align}
and Lemma \ref{budget2} guarantees the existence of a candidate optimal process $\pi^*_s$ such that $(c^*,\pi^*)\in\mathcal{A}_0(x,h)$. 
By Theorem 3.6.3 in \cite{karatzas1998methods} or Lemma 6.2 in \cite{karatzas2000utility}, one can then show that $(c^*, \pi^*)$ is optimal for the optimization problem $\widehat{V}$.

{Next we determine the optimal strategies and the corresponding optimal wealth process. From (\ref{wealth}) and the strong Markov property we have 
\begin{align*}
	X_t^*= \mathcal{X}(Z^{2*}_t,H_t^2)= \frac{I}{r}-W_z(Z^{2*}_t,H^2_t),
\end{align*}
where $Z^{2*}_t$ is the solution to Equation (\ref{eqz2}) with the initial condition $z^*(x,h):= z^*$ satisfying $W_z(z^*,h)+x-\frac{I}{r}=0$. Recalling the regularity of $W_z$ as in Proposition \ref{Wregularity}, we can apply It\^o's formula to the optimal wealth process $X_t^*$ to find 
\begin{align*}
	dX_t^*= [-W_{zz}(\rho-r+M_t^{2})Z_t^{2*}-W_{zh}(-\delta H^2_t +K)-\frac{1}{2}W_{zzz}\theta^2 (Z_t^{2*})^2] dt +\theta W_{zz}Z_t^{2*}dB_t,
\end{align*}
which, compared with the dynamics in (\ref{2-4}) yields 
\begin{align*}
	\pi^*_t= \frac{\theta}{\sigma} Z^{2*}_t W_{zz}(Z^{2*}_t,H_t^2).
\end{align*}}
\end{proof}

\subsection{Proof of Proposition \ref{finite}}
Before we present the proof of Proposition \ref{finite}, we need the following technical lemma.

\begin{lemma}\label{whbound}
Let $C_0(h):=\frac{m^1h^{-\kappa-1}(h+\frac{K}{\delta})}{(1-\alpha)\delta}$. 
Then
\begin{align*}
 0 \leq W_h(z,h)\leq z^{\frac{\alpha}{\alpha-1}}(1-\alpha)\alpha^{\frac{\alpha}{1-\alpha}} (c_0 C_0(h) +c_1),  \quad \forall (z,h) \in \mathcal{O},
\end{align*} 
where $c_0,c_1$ are suitable positive constants. 
 Moreover, 
\begin{align*}
 \lim_{z \uparrow \infty}W_h(z,h)=0 \quad \text{and} \quad \lim_{z \downarrow 0}W_h(z,h)=\infty \quad \text{for} \ h \in  \mathbb{R}_+.
\end{align*}
\end{lemma}

\begin{proof}

{ From (\ref{3-9-2}) and (\ref{3-9-3}) we compute the partial derivative with respect to $h$,}
\begin{align}\label{b3}
&W_h(z,h)
= z^{\frac{\alpha}{\alpha-1}}(1-\alpha)\alpha^{\frac{\alpha}{1-\alpha}}   \int_0^{\infty}  e^{ \frac{\alpha}{1-\alpha}(rs+\frac{1}{2}\theta^2s) + \frac{\theta^2 \alpha^2s}{2(\alpha-1)^2}}e^{\int^s_0\frac{\rho+M^{2,h}_u}{\alpha-1}du} \Big[\frac{m^1\kappa}{1-\alpha}\nonumber  \\& \times\Big(\int^s_0  \Big(h e^{-\delta u }+\frac{K}{\delta}(1-e^{-\delta u})\Big)^{-\kappa-1}e^{-\delta u}du\Big) \Big(h e^{-\delta s }+\frac{K}{\delta}(1-e^{-\delta s})\Big)+ e^{-\delta s}\Big] ds, 
\end{align}
where $M^{2,h}$ is given by (\ref{newm2}) with initial health level $h$. 
Since $\alpha<1,$ then $ W_h(z,h)\geq 0$ for any $(z,h) \in \mathcal{O}$. On the other hand, since $h e^{-\delta s } \leq h e^{-\delta s }+\frac{K}{\delta}(1-e^{-\delta s}) \leq h+\frac{K}{\delta}$, then   {\((h e^{-\delta u }+\frac{K}{\delta}(1-e^{-\delta u}))^{-\kappa-1} \leq h^{-\kappa-1}e^{(\kappa+1)\delta u}\)}. Therefore, from (\ref{b3}) we have
{
\begin{align*}
& \frac{m^1\kappa}{1-\alpha}
\bigg(\int^s_0
\Big(h e^{-\delta u}+\frac{K}{\delta}(1-e^{-\delta u})\Big)^{-\kappa-1}
e^{-\delta u}du\bigg)
\bigg(h e^{-\delta s}+\frac{K}{\delta}(1-e^{-\delta s})\bigg)
+ e^{-\delta s}
\\
&\leq
\frac{m^1\kappa}{1-\alpha}
\bigg(\int^s_0 h^{-\kappa-1}e^{(\kappa+1)\delta u}e^{-\delta u}du\bigg)
\Big(h+\frac{K}{\delta}\Big)
+ e^{-\delta s}
\\
&\leq
\frac{m^1h^{-\kappa-1}(h+\frac{K}{\delta})}{(1-\alpha)\delta}
e^{\delta \kappa s}+1
=C_0(h) e^{\delta \kappa s}+1.
\end{align*}
}

Combining the above inequality with (\ref{b3}), we have, for suitable $c_0>0,c_1>0$ independent of $(z,h)$,   
\begin{align*}
W_h(z,h) \leq z^{\frac{\alpha}{\alpha-1}}(1-\alpha)\alpha^{\frac{\alpha}{1-\alpha}} [c_0 C_0(h) +c_1],
\end{align*}
where we have used the fact that
\begin{align}\label{b4-1}
    \int_0^{\infty }  e^{ \frac{\alpha}{1-\alpha}(rs+\frac{1}{2}\theta^2s) + \frac{\theta^2 \alpha^2s}{2(\alpha-1)^2}}e^{\int^s_0\frac{(\rho+M^{{2,h}}_u)}{\alpha-1}du} [C_0(h)e^{\delta \kappa s}+1 ] ds  \leq   c_0 C_0(h) +c_1,   
\end{align}
due to Assumption \ref{assume1}.  Finally, it is easy to see from (\ref{b3}) that $\lim_{z \uparrow \infty}W_h(z,h)=0$ and $\lim_{z \downarrow 0}W_h(z,h)=\infty$.

\end{proof}

\begin{proof}[Proof of Proposition \ref{finite}]\label{profinite}
	From (\ref{4-2}) it is clear that $\widehat{J}$ is nonnegative. Moreover, again from (\ref{4-2}), and since $W_h(z,h) \geq 0$ (cf. Lemma \ref{whbound}), we find that
\begin{align*}
& \sup_{ \tau \in \mathcal S } \mathbb{E}_{z,h}\bigg[ \int_0^{\tau }  z  e^{-rs-\theta B_s-\frac{1}{2}\theta^2s}{I\,ds}- \int_0^{\tau }  e^{-\int^s_0 (\rho+ M^{{1}}_{u}) du} KW_h(Z^{1}_s,H^{1}_s)ds  \bigg] \nonumber \\
&\leq \mathbb{E}_{z,h}\bigg[ \int_0^{\infty }  z  e^{-rs-\theta B_s-\frac{1}{2}\theta^2s}{I\,ds}  \bigg] = \frac{Iz}{r},
\end{align*} 
which implies the claim.\end{proof}



\subsection{Proof of Proposition \ref{contJhat}}

\begin{proof}\label{procontJhat}
{Let \((z_n,h_n)\to(z,h)\) in \(\mathcal O\). Then, for every \(s\ge0\), one has
\[
    H_s^{1,h_n}\to H_s^{1,h},
    \qquad
    Z_s^{1,z_n,h_n}\to Z_s^{1,z,h},
    \qquad\text{a.s.}
\]
We claim that, as \(n\to\infty\),
\begin{equation}
    \label{eq:aligned-eq}
\begin{aligned}
& \mathbb E\bigg[
\int_0^\infty
\bigg|
e^{-\int_0^s(\rho+M_u^{1,h_n})du}
    IZ_s^{1,z_n,h_n}
-
e^{-\int_0^s(\rho+M_u^{1,h})du}
    IZ_s^{1,z,h}
\bigg| ds\bigg]
\\
 & 
+  \mathbb{E}\bigg[
\int_0^\infty
\bigg|
e^{-\int_0^s(\rho+M_u^{1,h_n})du}
    KW_h(Z_s^{1,z_n,h_n},H_s^{1,h_n})
\\& \quad -
e^{-\int_0^s(\rho+M_u^{1,h})du}
    KW_h(Z_s^{1,z,h},H_s^{1,h})
\bigg| ds
\bigg]
\longrightarrow0 .
\end{aligned}
\end{equation}
Indeed, for the first expectation in (\ref{eq:aligned-eq}), using the expression of \(Z^1\) as in
\eqref{strong}, we have
\[
\begin{aligned}
&\mathbb E\bigg[
\int_0^\infty
\bigg|
e^{-\int_0^s(\rho+M_u^{1,h_n})du}
    IZ_s^{1,z_n,h_n}
-
e^{-\int_0^s(\rho+M_u^{1,h})du}
    IZ_s^{1,z,h}
\bigg| ds\bigg]  \\
&\qquad =
I|z_n-z|\mathbb E\bigg[\int_0^\infty \gamma_s\,ds\bigg]
=
\frac{I|z_n-z|}{r}
\longrightarrow0 .
\end{aligned}
\]
}

{
As for the second expectation in \eqref{eq:aligned-eq}, the above pointwise
convergence, together with the continuity of \(W_h\) on \(\mathcal O\), gives
\[
\begin{aligned}
&e^{-\int_0^s(\rho+M_u^{1,h_n})du}
KW_h(Z_s^{1,z_n,h_n},H_s^{1,h_n})
\\
&\qquad\longrightarrow
e^{-\int_0^s(\rho+M_u^{1,h})du}
KW_h(Z_s^{1,z,h},H_s^{1,h}),
\qquad \text{a.s.}
\end{aligned}
\]
for every \(s\ge0\). Moreover, Lemma~\ref{whbound}, the explicit expression
of \(Z^1\) in \eqref{strong}, and the fact that the sequence \((z_n,h_n)\)
remains in a compact subset of \(\mathcal O\) around \((z,h)\), for all \(n\)
large enough, yield the estimate, uniformly in \(n\),
\[
\begin{aligned}
&e^{-\int_0^s(\rho+M_u^{1,h_n})du}
KW_h(Z_s^{1,z_n,h_n},H_s^{1,h_n})
\\
&\qquad\le
C\exp\left\{
\left[
\frac{\alpha}{\alpha-1}\left(\rho-r-\frac12\theta^2\right)
-\rho+(\kappa+1)\delta
\right]s
-\frac{\alpha}{\alpha-1}\theta B_s
\right\},
\qquad s\ge0,
\end{aligned}
\]
where the constant \(C>0\) is independent of \(n\). The same bound also
applies to the term starting from \((z,h)\). The right-hand side is
integrable with respect to \(\mathbb P\otimes ds\) by
Assumption~\ref{assume1}. Hence the claimed \(L^1\)-convergence follows from
the dominated convergence theorem.
}

{
Now fix \(\varepsilon>0\). For each \(n\), choose an
\(\varepsilon\)-optimal stopping time for \(\widehat J(z_n,h_n)\). Then, using
the \(L^1\)-convergence above and the fact that the same stopping time is
admissible when the initial point is \((z,h)\), we obtain
\[
\limsup_{n\to\infty}\widehat J(z_n,h_n)
\le
\widehat J(z,h)+\varepsilon .
\]
Conversely, choosing an \(\varepsilon\)-optimal stopping time for
\(\widehat J(z,h)\) and using it for the problem starting from \((z_n,h_n)\),
we obtain
\[
\widehat J(z,h)
\le
\liminf_{n\to\infty}\widehat J(z_n,h_n)+\varepsilon .
\]
Letting \(\varepsilon\downarrow0\) gives
\(\widehat J(z_n,h_n)\to\widehat J(z,h)\). Thus \(\widehat J\) is continuous
on \(\mathcal O\). Finally, since
\[
    J(z,h)=\widehat J(z,h)+W(z,h)-z\frac Ir
\]
and \(W\) is continuous on \(\mathcal O\), \(J\) is continuous as well.
}

\end{proof}

\subsection{Proof of Proposition \ref{limit} }

\begin{proof}\label{prolimit}

{Choosing the admissible stopping time \(\tau\equiv\infty\) in the
definition of \(\widehat J\), we obtain
\begin{equation}
\label{eq:functionaltauinfty}
\begin{aligned}
&\widehat{J}(z,h)
\geq
\mathbb{E}_{z,h}\bigg[
\int_0^{\infty}
\left(
z e^{-rs-\theta B_s-\frac{1}{2}\theta^2s}I
-
e^{-\int^s_0(\rho+M^{{1,h}}_{u})du}
KW_h(Z^{1,z}_s,H^{1,h}_s)
\right)ds
\bigg] \\
& =
\mathbb{E}_{z,h}\bigg[
\int_0^{\infty}
z e^{-rs-\theta B_s-\frac{1}{2}\theta^2s}I ds \bigg]
-
\mathbb{E}_{z,h}\bigg[
\int_0^{\infty}
e^{-\int^s_0(\rho+M^{{1,h}}_{u})du}
KW_h(Z^{1,z}_s,H^{1,h}_s)
ds
\bigg].
\end{aligned}
\end{equation}
The first expectation on the right-hand side of
\eqref{eq:functionaltauinfty} is equal to \(Iz/r\). We next show that the
second expectation converges to zero as \(z\to\infty\).
}

{
Fix \(z_0>0\). By Lemma~\ref{whbound} and the linearity of \(Z^{1,z}\) in
\(z\) (see \eqref{strong}), for \(z\ge z_0\) we have
\[
Z_s^{1,z}=\frac{z}{z_0}Z_s^{1,z_0},
\qquad s\ge0.
\]
Since \(\alpha/(\alpha-1)<0\), it follows that
\[
\left(Z_s^{1,z}\right)^{\frac{\alpha}{\alpha-1}}
=
\left(\frac{z}{z_0}\right)^{\frac{\alpha}{\alpha-1}}
\left(Z_s^{1,z_0}\right)^{\frac{\alpha}{\alpha-1}}
\le
\left(Z_s^{1,z_0}\right)^{\frac{\alpha}{\alpha-1}} .
\]
Therefore, for \(z\ge z_0\), the integrand in the second expectation on the
right-hand side of \eqref{eq:functionaltauinfty} is bounded from above by
\[
C
e^{-\int^s_0(\rho+M^{{1,h}}_{u})du}
\left(Z_s^{1,z_0}\right)^{\frac{\alpha}{\alpha-1}}
\]
for some constant \(C>0\). This upper bound is
\(\mathbb P\otimes ds\)-integrable on
\(\Omega\times[0,\infty)\) by Lemma~\ref{whbound} and
Assumption~\ref{assume1}. Moreover, for every \(s\ge0\),
\[
e^{-\int^s_0(\rho+M^{{1,h}}_{u})du}
KW_h(Z^{1,z}_s,H^{1,h}_s)
\longrightarrow0,
\qquad z\to\infty,
\]
because \(Z_s^{1,z}\to\infty\) and \(W_h(\cdot,H_s^{1,h})\) has power order
\(\alpha/(\alpha-1)<0\). Dominated convergence then yields
\[
\mathbb{E}_{z,h}\bigg[
\int_0^{\infty}
e^{-\int^s_0(\rho+M^{{1,h}}_{u})du}
KW_h(Z^{1,z}_s,H^{1,h}_s)
ds
\bigg]
\longrightarrow0 .
\]
Hence,  $\lim_{z\to\infty}\widehat J(z,h)=\infty.$ Finally, the fact that $\lim_{z\to0}\widehat J(z,h)=0$ directly follows from the bounds of $\widehat{J}$ in Proposition~\ref{finite}.
}
\end{proof}

\subsection{Proof of Theorem \ref{dualityrelation2}}

Before we present the proof of Theorem \ref{dualityrelation2}, we need the following lemma.

{{
\begin{lemma}\label{budget}
For any $\tau\in\mathcal{S}$, let $x\geq 0$ be given and let $c\geq 0$
be a consumption process. Let $\phi$ be an $\mathcal{F}_\tau$-measurable
random variable such that $\mathbb{P}_{x,h}[\phi>0]=1$ and
\[
\mathbb{E}_{x,h}\bigg[
    \gamma_\tau\phi
    +
    \int_0^\tau\gamma_s c_s\,ds
\bigg]
=
x,
\]
where $\gamma_\tau\phi=0$ on $\{\tau=+\infty\}$.
Then there exists an $\mathbb{F}$-progressively measurable portfolio
process $\pi$, defined up to time $\tau$, such that
\[
\int_0^{t\wedge\tau}|\pi_s|^2\,ds<\infty,
\qquad
\mathbb{P}_{x,h}\text{-a.s. for every }t\geq 0,
\]
and the corresponding pre-investment wealth process satisfies 
\begin{align*}
    X_t^{c,\pi,\tau}\geq 0, \quad
0\leq t\leq\tau, \qquad X_\tau^{c,\pi,\tau}=\phi
\quad\text{on }\{\tau<\infty\}.
\end{align*}
Moreover, $\pi$ can be chosen so that the associated pre-investment
wealth process admits the conditional budget representation
\[
X_t^{c,\pi,\tau}
=
\frac{1}{\gamma_t}
\mathbb{E}_{x,h}\!\left[
    \int_t^\tau\gamma_s c_s\,ds+\gamma_\tau\phi
    \,\middle|\,\mathcal{F}_t
\right],
\qquad\text{on }\{t\leq\tau\}.
\]
\end{lemma}

\begin{proof}
The result follows by the same martingale-representation argument as
in the proof of Lemma~6.3 in \cite{karatzas2000utility}, applied on
the stochastic interval $[0,\tau]$ and with the convention
$\gamma_\tau\phi=0$ on $\{\tau=\infty\}$. The details are omitted.
\end{proof}
}}

\begin{proof}[Proof of Theorem \ref{dualityrelation2}]\label{produalityrelation2}
First, we show the duality relations. Since \((c,\pi,\tau)\in\mathcal A(x,h)\)
is arbitrary, taking the supremum over
\((c,\pi,\tau)\in\mathcal A(x,h)\) on the left-hand side of \eqref{3-8}, we
get, for any \(z>0\) and \(x\geq0\),
\[
    V(x,h)\leq J(z,h)+zx.
\]
Hence
\[
    V(x,h)\leq \inf_{z>0}\{J(z,h)+zx\},
    \qquad
    J(z,h)\geq \sup_{x\geq0}\{V(x,h)-zx\}.
\]

We now prove the reverse inequality. It follows from Proposition
\ref{convex} that, for every \(x>0\), there exists a unique
\(z^*=z^*(x,h)>0\) such that $    x=-J_z(z^*,h).$
{For later use, when \(\tau^*(z,h)=\infty\), we set
\[
\gamma_{\tau^*(z,h)}
\left(
    -W_z(Z_{\tau^*(z,h)}^1,H_{\tau^*(z,h)}^1)+\frac Ir
\right)
:=0.
\]
This convention is justified by the transversality condition
\[
\lim_{T\to\infty}
\mathbb E_{z,h}\!\left[
    \gamma_T
    \left(
        -W_z(Z_T^1,H_T^1)+\frac Ir
    \right)
    \mathds{1}_{\{\tau^*(z,h)>T\}}
\right]
=0,
\]
which follows from Assumption~\ref{assume1} and the explicit expression of
\(W_z\) in Proposition~\ref{Wregularity}.}
For \(z>0\), define
\begin{equation}
    \label{def:Chihat}
\widehat{\mathcal X}(z,h)
:=
\mathbb E_{z,h}\bigg[
    \int_0^{\tau^*(z,h)}
        \gamma_s\mathcal I^u(zP_s^1,H_s^1)\,ds
    +
    \gamma_{\tau^*(z,h)}
    \left(
        -W_z(Z_{\tau^*(z,h)}^1,H_{\tau^*(z,h)}^1)+\frac Ir
    \right)
\bigg].
\end{equation}

From Theorem~\ref{dualityrelation}, we know that, on
\(\{\tau^*(z,h)<\infty\}\), at the investment time,
the post-investment optimal wealth is
$    -W_z(Z_{\tau^*(z,h)}^1,H_{\tau^*(z,h)}^1)+\frac Ir.$
For the given \((x,h)\), introduce the candidate optimal consumption process
\[
    c_s^*
    :=
    \mathcal I^u(z^*P_s^1,H_s^1)
    =
    \mathcal I^u(Z_s^{1*},H_s^1),
    \qquad 0\leq s<\tau^*(z^*,h),
\]
where \(Z^{1*}\) denotes the solution to \eqref{3-9} with initial condition
\(Z_0^{1*}=z^*\).

{
We first claim that
\begin{equation}\label{eq:Xhat-Jz}
    \widehat{\mathcal X}(z,h)=-J_z(z,h).
\end{equation}
We prove this identity through the linear equation satisfied by \(J_z\).
If \((z,h)\in\mathcal I\), then \(\tau^*(z,h)=0\), \(J=W-zI/r\) in the
stopping region, and the claim follows immediately. Hence, in what follows,
we assume that \((z,h)\in\mathcal W\).

By Proposition~\ref{regularity1}, \eqref{relation}, \eqref{3-14-1}, and
Proposition~\ref{Wregularity}, direct computations show that \(J\) solves,
on \(\mathcal W\), the linear equation
\begin{equation}
\label{eq:PDEforJ}
   \mathcal L J+\widehat u=0,
\end{equation}
where \(\mathcal L\) is the operator defined in \eqref{defineoprator}.
Moreover, Proposition \ref{prop:sobolev-regularity}, together with
Proposition~\ref{regularity1}, \eqref{relation}, \eqref{3-14-1}, and
Proposition~\ref{Wregularity}, yields that \(z\mapsto J(z,h)\) is
continuously differentiable on \(\mathbb R_+\). By the standard interior parabolic regularity (see, e.g., Theorem 10 in Chapter 3 of \cite{friedman2008partial}), one can differentiate
\eqref{eq:PDEforJ} with respect to \(z\). Hence \(q:=J_z\) satisfies, on
\(\mathcal W\),
\begin{align}\label{eq:Jz-linear}
    \frac12\theta^2z^2 q_{zz}
    +(\rho-r+m^0+m^1h^{-\kappa}+\theta^2)zq_z
    -\delta h q_h
    -rq
    =
    -\widehat u_z(z,h).
\end{align}
The boundary condition is inherited from the stopping payoff, together with
the continuity of \(J_z(\cdot,h)\) and the smooth-fit property established
above; namely,
\[
    q(z,h)=W_z(z,h)-\frac Ir,
    \qquad (z,h)\in\mathcal I\cup\partial\mathcal W.
\]

Recall that \(Z_t^{1,z}=zP_t^1\). Let \((\mathcal K_n)_{n\ge1}\) be an
increasing sequence of compact subsets of \(\mathcal W\) exhausting
\(\mathcal W\), and set
\[
    \sigma_n:=\inf\{t\ge0:(Z_t^1,H_t^1)\notin\mathcal K_n\}.
\]
Since \(q\) is \(C^{2,1}\) in the interior of \(\mathcal W\), It\^o's formula
can be applied to
\[
    \left(
    e^{-\int_0^t(\rho+M_u^{1})du}
    P_t^1 q(Z_t^1,H_t^1)
    \right)_{0\le t\le T\wedge\tau^*(z,h)\wedge\sigma_n}.
\]
On this stopped interval, the relevant derivatives of \(q\) are bounded, and
therefore the stochastic integral has zero expectation. Taking expectations
and using \eqref{eq:Jz-linear}, we obtain
\[
\begin{aligned}
J_z(z,h)
=
\mathbb E_{z,h}\bigg[
&\int_0^{T\wedge\tau^*(z,h)\wedge\sigma_n}
e^{-\int_0^s(\rho+M_u^{1})du}
\widehat u_z(Z_s^1,H_s^1)P_s^1\,ds
\\
&+
e^{-\int_0^{T\wedge\tau^*(z,h)\wedge\sigma_n}(\rho+M_u^{1})du}
P_{T\wedge\tau^*(z,h)\wedge\sigma_n}^1
q(Z_{T\wedge\tau^*(z,h)\wedge\sigma_n}^1,
  H_{T\wedge\tau^*(z,h)\wedge\sigma_n}^1)
\bigg].
\end{aligned}
\]
Letting \(n\uparrow\infty\) and then \(T\uparrow\infty\), dominated
convergence gives
\[
\begin{aligned}
J_z(z,h)
=
\mathbb E_{z,h}\bigg[
&\int_0^{\tau^*(z,h)}
e^{-\int_0^s(\rho+M_u^{1})du}
\widehat u_z(Z_s^1,H_s^1)P_s^1\,ds
\\
&+
e^{-\int_0^{\tau^*(z,h)}(\rho+M_u^{1})du}
\left(
W_z(Z_{\tau^*(z,h)}^1,H_{\tau^*(z,h)}^1)-\frac Ir
\right)
P_{\tau^*(z,h)}^1
\bigg].
\end{aligned}
\]
Here, on the event \(\{\tau^*(z,h)=\infty\}\), the terminal term is understood
as zero. To justify this passage, note that the terminal term at time \(T\)
contains
\[
\mathbb E_{z,h}\!\left[
    \gamma_T
    J_z(Z_T^1,H_T^1)
    \mathds{1}_{\{\tau^*(z,h)>T\}}\mathds{1}_{\{\sigma_n>T\}}
\right].
\]
By Proposition~\ref{convex}, \(J_z\le0\). Moreover, since
$    J=\widehat J+W-z\frac Ir,$
the monotonicity of \(\widehat J\) in \(z\), together with the differentiability
of \(z\mapsto\widehat J(z,h)\), gives \(\widehat J_z\ge0\). Hence
\[
0\le -J_z(z,h)
\le
-W_z(z,h)+\frac Ir .
\]
Using the explicit power expression of \(W_z\) obtained in (\ref{3-9-3}), together with Assumption~\ref{assume1}, this
bound yields the transversality condition
\begin{equation}
\label{eq:limitJ_z}
\lim_{T\to\infty}\lim_{n\to\infty}
\mathbb E_{z,h}\!\left[
    \gamma_T
    J_z(Z_T^1,H_T^1)
    \mathds{1}_{\{\tau^*(z,h)>T\}}\mathds{1}_{\{\sigma_n>T\}}
\right]
=0 .
\end{equation}

Using
\[
    \widehat u_z(z,h)=-\mathcal I^u(z,h),
    \qquad
    e^{-\int_0^s(\rho+M_u^{1})du}P_s^1=\gamma_s,
\]
we obtain
\[
\begin{aligned}
-J_z(z,h)
=
\mathbb E_{z,h}\bigg[
\int_0^{\tau^*(z,h)}
\gamma_s\mathcal I^u(Z_s^1,H_s^1)\,ds
+
\gamma_{\tau^*(z,h)}
\left(
-W_z(Z_{\tau^*(z,h)}^1,H_{\tau^*(z,h)}^1)+\frac Ir
\right)
\bigg].
\end{aligned}
\]
Comparing the last display with the definition \eqref{def:Chihat}, we obtain $    \widehat{\mathcal X}(z,h)=-J_z(z,h)$, as claimed. In particular, if \(z^*\) is chosen so that
\(-J_z(z^*,h)=x\), then
\begin{align}\label{mathcalx}
    \widehat{\mathcal X}(z^*,h)=x .
\end{align}
}

{To simplify notation, set
\(\tau^*:=\tau^*(z^*,h)\), and define the random variable
\[
\phi^*
:=
\begin{cases}
-W_z(Z_{\tau^*}^{1*},H_{\tau^*}^1)+\dfrac Ir,
& \tau^*<\infty,\\[0.4em]
\dfrac Ir,
& \tau^*=\infty.
\end{cases}
\]
We adopt the convention
$\gamma_{\tau^*}\phi^*
:=
\gamma_{\tau^*}\phi^*\mathds{1}_{\{\tau^*<\infty\}},$
so that \(\gamma_{\tau^*}\phi^*=0\) on \(\{\tau^*=\infty\}\).

By Theorem~\ref{dualityrelation}, on \(\{\tau^*<\infty\}\), for each
post-investment state \((Z_{\tau^*}^{1*},H_{\tau^*}^1)\), the quantity 
$    -W_z(Z_{\tau^*}^{1*},H_{\tau^*}^1)+\frac Ir$
is the corresponding optimal post-investment wealth level. Moreover, since
\(W_z<0\) on \(\mathcal O\), we have $\phi^*>\frac Ir$
on  $\{\tau^*<\infty\}.$ On \(\{\tau^*=\infty\}\), the above definition sets \(\phi^*=I/r\). Hence
$    \phi^*\geq \frac Ir>0.$
Moreover, since \(Z_{\tau^*}^{1*}\) and \(H_{\tau^*}^1\) are
\(\mathcal F_{\tau^*}\)-measurable on \(\{\tau^*<\infty\}\), and since
\(W_z\) is continuous, \(\phi^*\) is \(\mathcal F_{\tau^*}\)-measurable.
}

{We now verify the budget identity before investment. By the definition of the
candidate consumption process,
\[
    c_s^*:=\mathcal I^u(Z_s^{1*},H_s^1),
    \qquad 0\leq s<\tau^*,
\]
and by the identity \(\widehat{\mathcal X}(z^*,h)=x\) obtained in
\eqref{mathcalx}, we have
\[
\begin{aligned}
    x
    &=
    \widehat{\mathcal X}(z^*,h)=
    \mathbb E_{z^*,h}
    \bigg[
        \int_0^{\tau^*}
        \gamma_s\mathcal I^u(Z_s^{1*},H_s^1)\,ds
        +
        \gamma_{\tau^*}
        \left(
            -W_z(Z_{\tau^*}^{1*},H_{\tau^*}^1)+\frac Ir
        \right)
        \mathds{1}_{\{\tau^*<\infty\}}
    \bigg]\\
    &=
    \mathbb E_{z^*,h}
    \bigg[
        \int_0^{\tau^*}\gamma_s c_s^*\,ds
        +
        \gamma_{\tau^*}\phi^*
    \bigg].
\end{aligned}
\]
Therefore the pair \((c^*,\phi^*)\) satisfies the static budget constraint up to
the stopping time \(\tau^*\).
}

{We can now apply Lemma~\ref{budget}. It yields the existence of a progressively
measurable portfolio process \(\pi^*\) such that the corresponding wealth
process satisfies
\[
    X_s^{c^*,\pi^*,\tau^*}\geq0,
    \qquad 0\leq s\leq \tau^*,
\]
and
\[
    X_{\tau^*}^{c^*,\pi^*,\tau^*}
    =
    \phi^*
    =
    -W_z(Z_{\tau^*}^{1*},H_{\tau^*}^1)+\frac Ir
    \quad\text{on }\{\tau^*<\infty\}.
\]
In particular, the pre-investment strategy \((c^*,\pi^*)\) is admissible on
\([0,\tau^*]\).
}

{It remains to define the controls after investment. On
\(\{\tau^*<\infty\}\), starting from the
\(\mathcal F_{\tau^*}\)-measurable post-investment state
\[
    \left(
        X_{\tau^*}^{c^*,\pi^*,\tau^*},
        H_{\tau^*}^1
    \right)
    =
    \left(
        \phi^*,
        H_{\tau^*}^1
    \right),
\]
we apply Theorem~\ref{dualityrelation} conditionally on
\(\mathcal F_{\tau^*}\). Equivalently, for \(s\geq\tau^*\), let
\(H_s^{2,*}\) solve
\[
    dH_s^{2,*}=(-\delta H_s^{2,*}+K)ds,
    \qquad H_{\tau^*}^{2,*}=H_{\tau^*}^1,
\]
and let \(Z_s^{2,*}\) be the solution of \eqref{eqz2} on
\([\tau^*,\infty)\) with initial value $    Z_{\tau^*}^{2,*}=Z_{\tau^*}^{1*}.$
Since
\[
    \phi^*
    =
    -W_z(Z_{\tau^*}^{1*},H_{\tau^*}^1)+\frac Ir,
\]
the relation for the optimal wealth process stated in
Theorem~\ref{dualityrelation} is satisfied at the random initial state
\((Z_{\tau^*}^{1*},H_{\tau^*}^1)\). We therefore set, for \(s\geq\tau^*\),
\[
    c_s^*
    =
    \mathcal I^u(Z_s^{2,*},H_s^{2,*}),
    \qquad
    \pi_s^*
    =
    \frac{\theta}{\sigma}
    Z_s^{2,*}W_{zz}(Z_s^{2,*},H_s^{2,*}).
\]
These controls are progressively measurable and admissible after
\(\tau^*\), and the associated post-investment wealth process is
\[
    X_s^{c^*,\pi^*,\tau^*}
    =
    -W_z(Z_s^{2,*},H_s^{2,*})+\frac Ir
    \geq \frac Ir,
    \qquad s\geq\tau^*.
\]
On \(\{\tau^*=\infty\}\), the post-investment controls are irrelevant. Hence
the strategy obtained by pasting the pre-investment controls on
\([0,\tau^*)\) with the post-investment optimal controls on
\([\tau^*,\infty)\) satisfies
$    X_s^{c^*,\pi^*,\tau^*}\geq0,
     0\leq s<\tau^*,$
and
$    X_s^{c^*,\pi^*,\tau^*}\geq \frac Ir,
 s\geq \tau^*.$
Therefore,
$    (c^*,\pi^*,\tau^*)\in\mathcal A(x,h).$
}

We next show that this admissible strategy attains the dual upper bound. Using
the definition of \(\widehat{\mathcal X}\), the identity
\(\widehat{\mathcal X}(z^*,h)=x\), and the relations for
\(\widehat u\) and \(\widehat V\), with all terminal terms at \(\tau^*\)
{interpreted according to the convention above}, we obtain
\[
\begin{aligned}
z^*\widehat{\mathcal X}(z^*,h)
=&\
\mathbb E_{z^*,h}\bigg[
    \int_0^{\tau^*}
        e^{-\int_0^s(\rho+M_u^{1})du}
        Z_s^{1*}c_s^*\,ds +
    e^{-\int_0^{\tau^*}(\rho+M_u^{1})du}
    Z_{\tau^*}^{1*}\phi^*
    \mathds{1}_{\{\tau^*<\infty\}}
\bigg]\\
=&\
\mathbb E_{z^*,h}\bigg[
    \int_0^{\tau^*}
        e^{-\int_0^s(\rho+M_u^{1})du}
        \Big(
            u(c_s^*,H_s^1)-\widehat u(Z_s^{1*},H_s^1)
        \Big)\,ds\\
&\qquad\qquad+
    e^{-\int_0^{\tau^*}(\rho+M_u^{1})du}
    \Big(
        \widehat V(\phi^*,H_{\tau^*}^1)
        -
        W(Z_{\tau^*}^{1*},H_{\tau^*}^1)
        +
        Z_{\tau^*}^{1*}\frac Ir
    \Big)
    \mathds{1}_{\{\tau^*<\infty\}}
\bigg].
\end{aligned}
\]
Rearranging and using the definition of \(J(z^*,h)\), we get
\[
\begin{aligned}
J(z^*,h)+z^*x
=
\mathbb E_{z^*,h}\bigg[
\int_0^{\tau^*}
    e^{-\int_0^s(\rho+M_u^{1})du}
    u(c_s^*,H_s^1)\,ds+
e^{-\int_0^{\tau^*}(\rho+M_u^{1})du}
\widehat V(\phi^*,H_{\tau^*}^1)
\mathds{1}_{\{\tau^*<\infty\}}
\bigg].
\end{aligned}
\]
The right-hand side is the payoff generated by the admissible strategy
constructed above. Hence
$    V(x,h)\geq J(z^*,h)+z^*x.$
Together with
$    V(x,h)\leq \inf_{z>0}\{J(z,h)+zx\},$
and the fact that \(z^*\) minimizes \(J(z,h)+zx\), we conclude that
\[
 {   V(x,h)
    =
    \inf_{z>0}\{J(z,h)+zx\}
    =
    J(z^*,h)+z^*x.}
\]
This also implies $J(z,h)=\sup_{x>0}\{V(x,h)-zx\}$.

It remains to identify the feedback form of the optimal wealth and portfolio
processes before the investment time \(\tau^*=\tau^*(z^*,h)\). {By the budget
representation in Lemma~\ref{budget}, \eqref{def:Chihat}, and the strong
Markov property applied to the shifted process starting from
\((Z_t^{1*},H_t^1)\), the pre-investment wealth generated by the replicating
portfolio satisfies, on \(\{t<\tau^*\}\),
\[
\begin{aligned}
    X_t^*
    =
    \frac{1}{\gamma_t}
    \mathbb E\!\left[
        \int_t^{\tau^*}\gamma_s c_s^*\,ds
        +
        \gamma_{\tau^*}\phi^*
        \,\middle|\,\mathcal F_t
    \right] =
    \widehat{\mathcal X}(Z_t^{1*},H_t^1).
\end{aligned}
\]
Using \eqref{eq:Xhat-Jz}, we then obtain
\[
    X_t^*
    =
    \widehat{\mathcal X}(Z_t^{1*},H_t^1)
    =
    -J_z(Z_t^{1*},H_t^1),
    \qquad 0\leq t<\tau^*.
\]
}
{Since \(J_z\) is smooth in the interior of the waiting region \(\mathcal W\)
(see, e.g., Theorem 10 in Chapter 3 of \cite{friedman2008partial}, as well
as Proposition~\ref{regularity1} and \eqref{relation}, \eqref{3-14-1},
\eqref{3-14}), we can apply It\^o's formula to}
\[
    X_t^*=-J_z(Z_t^{1*},H_t^1),
    \qquad 0\leq t<\tau^*(z^*),
\]
on the time interval \([0,\tau_n]\), where
$   \tau_n
    :=
    \inf\{t\geq0:(Z_t^{1*},H_t^1)\notin \mathcal K_n\}
    \wedge \tau^*(z^*),$
with \((\mathcal K_n)_{n\geq 1}\) being a sequence of compact subsets
exhausting the open set \(\mathcal W\). Comparing the diffusion coefficient
with that of the wealth dynamics in \eqref{2-4} yields
\[
    \pi_t^*\sigma
    =
    \theta Z_t^{1*}J_{zz}(Z_t^{1*},H_t^1),
    \qquad 0\leq t<\tau_n.
\]
Letting \(n\to\infty\), we obtain
\[
    \pi_t^*
    =
    \frac{\theta}{\sigma}Z_t^{1*}J_{zz}(Z_t^{1*},H_t^1),
    \qquad 0\leq t<\tau^*(z^*).
\]
Finally,
\[
    c_t^*=\mathcal I^u(Z_t^{1*},H_t^1),
    \qquad 0\leq t<\tau^*(z^*),
\]
by construction. This completes the proof.

\end{proof}

\section{{Numerical Scheme}}\label{numerical}
\renewcommand\theequation{B.\arabic{equation}}
\renewcommand\thesubsection{B.\arabic{subsection}}

{To compute the free boundary numerically, we use the time-inhomogeneous
boundary \(\widetilde b\) introduced in Section~\ref{sec4}. The original boundary is recovered from
$    b(h)=\widetilde b(0)
$ (cf.\ (\ref{relationb})).
From Theorem~\ref{thm:Jtilde-representation}, the boundary \(\widetilde b\)
satisfies
\begin{align}\label{d1-tilde}
0
=
\int_0^\infty
e^{-\int_0^s(\rho+\widetilde m(t+u))du}
\mathbb E_{\widetilde b(t)}\bigg[
\Big(
I Z_s^1
-
K W_h(Z_s^1,\widetilde h(t+s))
\Big)
\mathds{1}_{\{Z_s^1\geq \widetilde b(t+s)\}}
\bigg]ds,
\end{align}
where \(Z^1\) under $\mathbb P_{\widetilde b(t)}$ satisfies (cf.\ (\ref{strong}))
\[
Z_s^1
=
\widetilde b(t)
\exp\left\{
\int_0^s
\left(\rho-r+\widetilde m(t+u)-\frac12\theta^2\right)du
-\theta B_s
\right\}.
\]}

{Using the power form
\[
W_h(z,h)=z^{\frac{\alpha}{\alpha-1}}\Gamma(h),
\]
where \(\Gamma\) is given in \eqref{a11}, we obtain
\[
W_h(Z_s^1,\widetilde h(t+s))
=
(Z_s^1)^{\frac{\alpha}{\alpha-1}}\Gamma(\widetilde h(t+s)).
\]
Therefore,
\begin{align*}
&\mathbb E_{\widetilde b(t)}\bigg[
\Big(
I Z_s^1
-
K W_h(Z_s^1,\widetilde h(t+s))
\Big)
\mathds{1}_{\{Z_s^1\geq \widetilde b(t+s)\}}
\bigg]
\\
&=
I\mathbb E_{\widetilde b(t)}\left[
Z_s^1\mathds{1}_{\{Z_s^1\geq \widetilde b(t+s)\}}
\right]
-
K\Gamma(\widetilde h(t+s))
\mathbb E_{\widetilde b(t)}\left[
(Z_s^1)^{\frac{\alpha}{\alpha-1}}
\mathds{1}_{\{Z_s^1\geq \widetilde b(t+s)\}}
\right].
\end{align*}}

{By direct computations, using the log-normal distribution of $Z^1$ under $\mathbb P_{\widetilde b(t)}$, we have
\begin{align*}
\mathbb E_{\widetilde b(t)}\left[
Z_s^1\mathds{1}_{\{Z_s^1\geq \widetilde b(t+s)\}}
\right]
&=
\widetilde b(t)
\exp\left\{
\int_0^s(\rho-r+\widetilde m(t+u))du
\right\}
\Phi\left(d^1\left(t,s,\frac{\widetilde b(t+s)}{\widetilde b(t)}\right)\right),
\\
\mathbb E_{\widetilde b(t)}\left[
(Z_s^1)^{\frac{\alpha}{\alpha-1}}
\mathds{1}_{\{Z_s^1\geq \widetilde b(t+s)\}}
\right]
&=
\widetilde b(t)^{\frac{\alpha}{\alpha-1}}
\exp\left\{
\frac{\alpha}{\alpha-1}
\int_0^s\left(\rho-r+\widetilde m(t+u)-\frac12\theta^2\right)du
\right.
\\
&\qquad\qquad\left.
+
\frac12\left(\frac{\alpha}{\alpha-1}\right)^2\theta^2s
\right\}
\Phi\left(d^2\left(t,s,\frac{\widetilde b(t+s)}{\widetilde b(t)}\right)\right),
\end{align*}
where
\begin{align*}
d^1(t,s,y)
&:=
\frac{
\int_0^s\left(\rho-r+\widetilde m(t+u)+\frac12\theta^2\right)du-\log y
}{
{|\theta|}\sqrt{s}
},
\\
d^2(t,s,y)
&:=
\frac{
\int_0^s\left(\rho-r+\widetilde m(t+u)-\frac12\theta^2\right)du
+
\frac{\alpha}{\alpha-1}\theta^2s
-\log y
}{
{|\theta|}\sqrt{s}
}.
\end{align*}
Here \(\Phi(\cdot)\) denotes the cumulative distribution function of a standard
normal random variable.}

{Then \eqref{d1-tilde} becomes
\begin{align}\label{eqb-2-tilde}
0
=
\int_0^\infty
e^{-\int_0^s(\rho+\widetilde m(t+u))du}
\bigg[
&I\widetilde b(t)
\exp\left\{\int_0^s(\rho-r+\widetilde m(t+u))du\right\}
\Phi\left(d^1\left(t,s,\frac{\widetilde b(t+s)}{\widetilde b(t)}\right)\right)
\nonumber\\
&-
K\Gamma(\widetilde h(t+s))
\widetilde b(t)^{\frac{\alpha}{\alpha-1}}
\exp\left\{
\frac{\alpha}{\alpha-1}
\int_0^s\left(\rho-r+\widetilde m(t+u)-\frac12\theta^2\right)du
\right.
\nonumber\\
&\qquad\qquad\left.
+
\frac12\left(\frac{\alpha}{\alpha-1}\right)^2\theta^2s
\right\}
\Phi\left(d^2\left(t,s,\frac{\widetilde b(t+s)}{\widetilde b(t)}\right)\right)
\bigg]ds .
\end{align}}

{
To solve \eqref{eqb-2-tilde} numerically, we use a recursive integration
scheme. For a candidate boundary $\mathcal{P}$ and $(t,y)\in[0,\infty)\times(0,\infty)$,
let $F(t,y; \mathcal{P})$ denote the right-hand side of \eqref{eqb-2-tilde} with
$\widetilde b(t)$ replaced by $y$ and $\widetilde b(t+s)$ replaced by
$\mathcal{P}(t+s)$. Thus the boundary equation can be written as
\begin{align*}
F(t,\widetilde b(t);\widetilde b)=0,\qquad t\geq0.
\end{align*}
Starting from an initial guess $\widetilde b^{(0)}$, we define
$\widetilde b^{(n)}$ recursively by solving, at each grid point $t_i$,
\begin{align*}
F(t_i,y;\widetilde b^{(n-1)})=0
\end{align*}
for the positive root $y$, and then setting
$\widetilde b^{(n)}(t_i)=y$. Between grid points,
$\widetilde b^{(n-1)}$ is evaluated by interpolation. The integral in
\eqref{eqb-2-tilde} is truncated at a sufficiently large upper bound and
computed by numerical quadrature. The iteration is stopped when
\begin{align*}
\max_i|\widetilde b^{(n)}(t_i)-\widetilde b^{(n-1)}(t_i)|
\end{align*}
falls below a prescribed tolerance. For each fixed initial health level $h$,
the original dual boundary is recovered as $b(h)=\widetilde b(0)$.
}


\section*{Acknowledgments}
{The authors are grateful to the Co-Editor, Associate Editor, and two anonymous referees for their stimulating comments, which {greatly helped improve} previous versions of this work.} Funded by the Deutsche Forschungsgemeinschaft (DFG, German Research Foundation) – Project-ID 317210226 – SFB 1283. The work of Shihao Zhu was also supported by {the China Scholarship Council}. The authors also thank Salvatore Federico for useful discussions and for suggesting Lemma A.7 in \cite{federico2017impact}.

\bibliographystyle{plain}
\bibliography{healthcare}

\end{document}